\documentclass[preprint]{imsart}

\RequirePackage[OT1]{fontenc}
\RequirePackage{amsthm,amsmath,amssymb,amsfonts}
\RequirePackage{graphicx}
\RequirePackage[numbers]{natbib}
\RequirePackage[colorlinks,citecolor=blue,urlcolor=blue]{hyperref}

\usepackage{tikz}
\usetikzlibrary{backgrounds}
\usetikzlibrary{patterns,fadings}
\usetikzlibrary{arrows,decorations.pathmorphing}
\usetikzlibrary{calc}
\definecolor{light-gray}{gray}{0.95}
\usepackage{float}

\usepackage{color}
\startlocaldefs
\numberwithin{equation}{section}
\theoremstyle{plain}
\newtheorem{theorem}{Theorem}[section]
\newtheorem{lemma}[theorem]{Lemma}
\newtheorem{proposition}[theorem]{Proposition}
\newtheorem{corollary}[theorem]{Corollary}
\newtheorem{remark}[theorem]{Remark}

\newtheorem{definition}[theorem]{Definition}
\newcommand{\mc}[1]{{\mathcal #1}}
\newcommand{\mf}[1]{{\mathfrak #1}}

\newcommand{\bb}[1]{{\mathbb #1}}
\newcommand{\bs}[1]{{\boldsymbol #1}}
\newcommand{\eps}{\varepsilon}

\newcommand{\<}{\langle}
\renewcommand{\>}{\rangle}
\newcommand{\p}{\partial}
\newcommand{\pfrac}[2]{\genfrac{}{}{}{1}{#1}{#2}}

\newcommand{\at}[2]{\genfrac{}{}{0pt}{}{#1}{#2}}

\newcommand{\ioe}{\iota_\eps}
\newcommand{\iog}{\iota_\gamma^{\textrm{s}}}
\newcommand{\ioz}{\iota_\zeta}

\newcommand{\radon}{\textrm{\textbf{d}}\bb P^H_{\mu_N}/\textrm{\textbf{d}}\bb P_{\mu_N}}
\newcommand{\radonN}{\frac{\textrm{\textbf{d}}\bb P^H_{\mu_N}}{\textrm{\textbf{d}}\bb P_{\mu_N}}}
\newcommand{\radonNinv}{\frac{\textrm{\textbf{d}}\bb P_{\mu_N}}{\textrm{\textbf{d}}\bb P^H_{\mu_N}}}
\newcommand{\C}{C^{1,2}([0,T]\times [0,1])}
\newcommand{\T}{\mathbb{T}_N}
\newcommand{\LN}{{L}_N}
\newcommand{\LNH}{{L}_{N,t}^H}
\newcommand{\ON}{\Omega_N}
\newcommand{\Sob}{L^2(0,T;\mc H^1(0,1))}

\newcommand{\Ck}{C^{0,1}_{\textrm{\rm k}}([0,T]\times (0,1))}
\newcommand{\Ddiscreto}{\mc D_{\Omega_N}}
\newcommand{\DM}{\mc D_{\mc M}}
\newcommand{\DMO}{\mc D_{\mc M_0}}

\newcommand{\DME}{\mc D_{\mc M_0}^{\textrm{\rm eq}}}
\newcommand{\DMS}{\mc D^{\mc S}_{ \mc M_0}}

\def\centerarc[#1](#2)(#3:#4:#5){\draw[#1] ($(#2)+({#5*cos(#3)},{#5*sin(#3)})$) arc (#3:#4:#5);}
\endlocaldefs

\begin{document}

\begin{frontmatter}
\title{Large deviations for the exclusion\\ process with a slow bond} %\thanksref{T1}}
\runtitle{Large deviations for exclusion with a slow bond}
%\thankstext{T1}{Footnote to the title with the ``thankstext'' command.}

\begin{aug}
\author{\fnms{Tertuliano} \snm{Franco}\thanksref{}\ead[label=e1]{tertu@ufba.br}}
\and
\author{\fnms{Adriana} \snm{Neumann}\thanksref{}\ead[label=e2]{aneumann@mat.ufrgs.br}}

%\thankstext{t1}{First supporter of the project}
%\thankstext{t2}{Second supporter of the project}
\runauthor{T. Franco and A. Neumann}

\affiliation{UFBA\thanksmark{m1} and UFRGS\thanksmark{m2}}

\address{Tertuliano Franco\\
 Instituto de Matem\'atica,\\
  Campus de Ondina, \\ Av. Adhemar de Barros, S/N.\\ CEP 40170-110,
Salvador, Brasil\\
\printead{e1}\\
\phantom{E-mail:\ }}

\address{Adriana Neumann\\
Instituto de Matem\'atica, \\
Campus do Vale,\\ Av. Bento Gon\c calves, 9500. \\ 
CEP 91509-900, Porto Alegre, Brasil\\
\printead{e2}}
\end{aug}

\begin{abstract}
We consider the one-dimensional symmetric simple exclusion process with a slow bond. In this model, whilst all the transition rates  are equal to one, a particular bond, the \emph{slow bond}, has associated transition rate  of value $N^{-1}$, where $N$ is the scaling parameter. This model has been considered in previous works on the subject of hydrodynamic limit and fluctuations. 
In this paper, assuming uniqueness for weak solutions of hydrodynamic equation associated to the perturbed process, we obtain dynamical large deviations estimates in the diffusive scaling. The main challenge here is the fact that the presence of the slow bond gives rise to Robin's boundary conditions in the \emph{continuum},   substantially complicating the large deviations scenario.
\end{abstract}



\end{frontmatter}

\tableofcontents

\section{Introduction}\label{s1}
In this paper we present dynamical large deviations estimates for the Symmetric Simple Exclusion Process (SSEP) with a slow bond. 
The SSEP is a largely studied process both in Probability and Statistical Mechanics. It consists of particles that perform   independent random walks in a certain graph,  except for the exclusion rule that prevents two or more particles from occupying the same site.

The SSEP with \emph{a slow bond} is characterized by a defect at a fixed bond. The graph here considered is $\bb T_N=\bb Z / {N\bb Z}$, the discrete one-dimensional torus with $N$ sites. Let us describe this process in terms of clocks. At each bond we associate a different Poisson clock, all of them independent.  When a clock rings, the occupation at the sites connected by the corresponding bond are exchanged. Of course, if both sites are empty or occupied, nothing happens. We call the parameters of those Poisson clocks of \emph{exchange rates}. 
All exchange rates are equal to one, except  at the slow bond which has exchange rate $N^{-1}$, which slows down the passage of particles there. Notice that the choice of the exchange rates  characterizes the non-homogeneity of the environment.

This model has  origin in the models considered in \cite{fjl,fl}. In \cite{fjl}, 
 the exchange rate at a bond of vertices $x$ and $x+1$ is taken as $[N(W(x+1/N)-W(x/N))]^{-1}$, 
where $W$ is a $\alpha$-stable subordinator of a 
L\'evy process.  In
the same line, \cite{fl} dealt with exchange rates driven by a general, non-random, strictly increasing function $W$. The SSEP with a slow bond is in fact a particular case of the model considered in \cite{fl}.

In order to understand the collective behavior of the microscopic system, a natural question is the limit for the time evolution of the spatial density of particles, usually called \emph{hydrodynamic limit}, see \cite{kl} and references therein.
The limiting density of a given system is usually characterized as the weak solution of some partial differential equation, being the associated equation denominated \emph{hydrodynamic equation}.

By \cite{fl,fgn1,fgn2}, the hydrodynamic limit of the SSEP with a slow bond is  well understood, being the hydrodynamic equation given by following heat equation with Robin's boundary conditions:
\begin{equation}\label{edp0}
\begin{cases}
\;\p_t\rho(t,u)=\p_u^2 \rho(t,u), &t>0, u\in\bb T\backslash\{0\},\\
  \; \p_u\rho(t,0^+)=\p_u\rho(t,0^-)= \rho(t,0^+)-\rho(t,0^-), &t>0,\\
  \rho(0,u)=\gamma(u), & u\in\bb T,\\
\end{cases}
\end{equation}
 where $\bb T$ denotes the continuous one-dimensional torus, $0^+$ and $0^-$ denote the side limits around $0\in \bb T$ and $\gamma:\bb T\to [0,1]$ is a density profile. The boundary condition above can be interpreted as the \emph{Fick's Law}: the rate in which mass is exchanged between two media is proportional to the difference of concentration in each medium.
 
 The natural questions that emerge in the sequence are fluctuations and large deviations with respect to the expected limit. Equilibrium fluctuations for the SSEP with a slow bond has been studied in \cite{fgn3}. In this work we analyze the corresponding large deviations, consisting in the occurrence rate of events differing from the expected limit in the scaling of the hydrodynamic limit. The large deviations of a Markov process comes from two origins. One part are deviations from the initial measure, and the second are deviations from the dynamics. These are called statical and dynamical large deviations, respectively. Since the invariant measures for the dynamics here considered are Bernoulli product measure, for which the large deviations are well known, we will treat only the dynamical large deviations: the system will start  from \textit{deterministic} initial configurations associated in some sense (Definition \ref{associated}) to a macroscopic profile.

 The main difficulty for establishing  large deviations for the SSEP with a slow bond of parameter $N^{-1}$ 
comes from the fact that the limiting occupations at the vertices of the slow bond depend on time, as we can see in the Robin's boundary conditions above. In important previous papers \cite{blm} and \cite{flm}, the authors have considered exclusion process with fixed rate of incoming and outcoming particles at the boundaries leading to Dirichlet's boundary conditions, therefore with time independent values at the boundaries. 

Here it has been considered a single slow bond. An extension to a finite number of slow bonds (in the setting of \cite{fgn1}) would be straightforward, with no additional obstacles. However, it would carry on the notation and probably would imply a loss of clarity. For this reason we decided to focus in the single slow bond case.
What is still far from manageability are the large deviations for the model of \cite{fl}, which deals with much stronger spatial non-homogeneity (a dense set of slow bonds is allowed there). This is a very interesting  and challenging problem.

An important ingredient in the large deviations proof consists in establishing the law of large numbers for a suitable set of perturbations of the original systems.
The family of perturbations we have considered is \emph{the weakly asymmetric exclusion process (WASEP) with a slow bond}. Its hydrodynamic equation is a non-linear diffusive partial differential equation with non-linear Robin's boundary conditions. Assuming uniqueness of weak solutions of this equation, which is a delicate question due to the non-linearity at the boundary, we prove the corresponding hydrodynamic limit. Existence of weak solutions is granted by the tightness  of the processes.

The Radon-Nikodym derivative of the perturbed process  with respect 
to the original process naturally leads to the expression of the large deviations rate functional.  A difficulty in the proof of the upper bound comes from fact the Radon-Nikodym derivative obtained is not a function of the empirical measure. To overpass of this obstacle, we  show that the Radon-Nikodym derivative  is superexponentially close to a function of the empirical measure. Moreover, following steps of \cite{blm,flm}  
we define an energy and then proving that trajectories with infinity energy are not relevant in the large deviations regime. Carefully handling this facts together we organize the scenario in order to invoke the Minimax Lemma attaining the upper bound for compact sets. Exponential tightness finally leads to the upper bound for closed sets.

Since  the upper bound is achieved via an optimization over perturbations, the rate functional obtained  turns to be expressed by a variational expression.  On the other hand, for the large deviations lower bound, it is required to find the cheapest perturbation that leads the system to a given profile distinct of the expected limit. In other words, it is necessary to solve the variational expression of the rate function, at least for a sufficiently large class of density profiles. This is precisely what we do in the large deviations lower bound, by means of a proof surprisingly simple. In fact, the proof (of Proposition \ref{charact_H_suave}) consists essentially in checking that the perturbation $H$ that leads the system to a limit $\rho^H$ is the cheapest one. Indeed, a difficult part of the work was to find the correct class of perturbations for the dynamics and fulfil the technical details. 

Then, since the rate functional is convex in a specific sense, by a density argument we extend the lower bound for the class of smooth profiles. The extension for general profiles is  a hard problem of convex analysis and illustrates that there is much to be  develop in terms in of Orlicz Spaces as devices in large deviations schemes. This is subject of future work.

The paper is divided as follows.
In Section~\ref{s2}, we introduce notation and state the main results, namely:  Theorem~\ref{t02} and Theorem~\ref{t03}.  In
Section~\ref{chapter super}, we establish the replacement lemma and the energy estimates.
In Section~\ref{weakly chapter}, we prove the Theorem \ref{t02}.
In Section~\ref{upper}, we prove the upper bound. Finally, the lower bound for smooth profiles is  presented in the Section~\ref{lower}.

%%%%%%%%%%%%%%%%%%%%%%%%%%%%%%%%%%% Section 2 %%%%%%%%%%%%%%%%%%%%%%%%%%%%%%%%%%%%%%%%%%%%

\section{Model and statements}\label{s2}
Let  $\T=\bb Z/N\bb Z=\{0,1,2,\ldots, N-1\}$ be the one-dimensional discrete torus with $N$ points. In each site of $\T$ we allow at most one particle.  In other words, we consider configurations of particles $\eta\in \{0,1\}^{\bb T_{N}}$. We say that $\eta(x)=0$, 
if the site $x\in \bb T_{N}$ is vacant and $\eta(x)=1$, if the site $x\in \T$ is occupied. Notice that $x=0$ and $x=N$ are the same site.
Denote by $\Omega_N=\{0,1\}^{\bb T_{N}}$ this state space.

The exclusion process with a slow bond at the bond of vertices $-1,0$, which has been considered in \cite{fl,fgn1,fgn2}, can be described as follows. To each bond of $\bb T_N$ we associate
a Poisson clock, all of them independent. If the bond is that one of vertices $-1,0$, the parameter of the Poisson is taken as $1/N$. All the others Poisson clocks have parameter one.
When a clock rings,  the occupation values of $\eta$ at the vertices of the associated bond are exchanged. The smaller parameter at the bond of vertices $-1,0$ slows the passage of particles cross it, from where the name \emph{slow bond}.
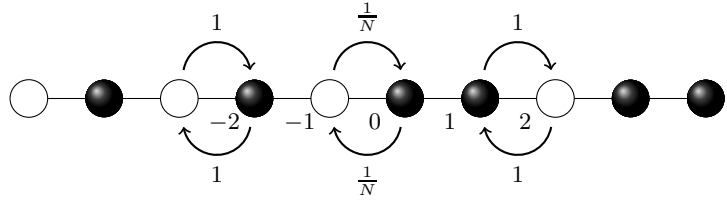
\begin{figure}[H]
\centering
\begin{tikzpicture}
%%%%%%arcos%%%%%%%%%%%%%
\centerarc[thick,<-](2.5,0.3)(10:170:0.45);
\centerarc[thick,->](2.5,-0.3)(-10:-170:0.45);
\centerarc[thick,->](4.5,-0.3)(-10:-170:0.45);
\centerarc[thick,<-](4.5,0.3)(10:170:0.45);
\centerarc[thick,->](6.5,-0.3)(-10:-170:0.45);
\centerarc[thick,<-](6.5,0.3)(10:170:0.45);

%%%%%%%%%%%%%%segmento%%%%%%%%%
\draw (0,0) -- (9,0);

%%%%%%%%%% partículas em azul %%%%%%%%%%%
\shade[ball color=black](1,0) circle (0.25);
\shade[ball color=black](3,0) circle (0.25);
\shade[ball color=black](5,0) circle (0.25);
\shade[ball color=black](6,0) circle (0.25);
\shade[ball color=black](8,0) circle (0.25);
\shade[ball color=black](9,0) circle (0.25);

%%%%%%%%%%% partículas em branco %%%%%%%%%%%
\filldraw[fill=white, draw=black]
(0,0) circle (.25)
(2,0) circle (.25)
(4,0) circle (.25)
(7,0) circle (.25);

%%%%%%%%%% rótulos %%%%%%%%%%%%%%
\draw (2.6,-0.1) node[anchor=north] {\small $-2$}
(3.6,-0.1) node[anchor=north] {\small $-1$}
(4.6,-0.1) node[anchor=north] {\small $0$}
(5.6,-0.1) node[anchor=north] {\small $1$}
(6.6,-0.1) node[anchor=north] {\small $2$};
\draw (2.5,0.8) node[anchor=south]{\small $1$};
\draw (2.5,-0.8) node[anchor=north]{\small $1$};
\draw (6.5,0.8) node[anchor=south]{\small $1$};
\draw (6.5,-0.8) node[anchor=north]{\small $1$};
\draw (4.5,0.8) node[anchor=south]{\small $\frac{1}{N}$};
\draw (4.5,-0.8) node[anchor=north]{\small $\frac{1}{N}$};
\end{tikzpicture}
\caption{The bond of vertices $\{-1,0\}$, the slow bond, has   particular rates associated to it.}\label{fig:1}
\end{figure}
This Markov process can also be characterized in terms of its infinitesimal generator $\LN$, which acts on  functions $f:\Omega_N\rightarrow \bb{R}$ as
\begin{equation}\label{ln}
(\LN f)(\eta)=\frac{1}{N}\big[f(\eta^{-1,0})-f(\eta)\big]+\sum_{\at{x\in \bb T_{N}}{x\neq -1}} \big[f(\eta^{x,x+1})-f(\eta)\big]\,,
\end{equation}
where $\eta^{x,x+1}$ is the configuration obtained from $\eta$ by exchanging the variables $\eta(x)$, $\eta(x\!+\!1)$:
\begin{equation}\label{eta}
\eta^{x,x+1}(y)\;=\;\left\{\begin{array}{cl}
\eta(x\!+\!1),& \mbox{if}\,\,\, y=x\,,\\ 
\eta(x),& \mbox{if} \,\,\,y=x+1\,,\\ 
\eta(y),& \mbox{otherwise\,.}
\end{array}
\right.
\end{equation}
Denote by $\{\eta_t ; t\ge 0\}$ the Markov process on $\Omega_N=\{0,1\}^{\bb
  T_N}$ associated to the generator $\LN$, defined in \eqref{ln}, \emph{speeded up} by
$N^2$. The dependency of $\eta_t$ in $N$ will be omitted to keep notation as simple as possible.
\begin{remark} \rm
This is a notion that often causes confusion and for this reason we explain it in detail. By $\eta_t$ we mean the Markov process which generator is $N^2\LN$.  Equivalently, $\eta_t$ could be defined as the Markov process with generator $\LN$ (without the factor $N^2$) seen at time $tN^2$.
\end{remark}

 Let $\mc D\big([0,T], \ON\big)$ be the path space of
c\`adl\`ag time trajectories with values in $\ON=\{0,1\}^{\bb T_N}$. For short, we will denote this space just by $\Ddiscreto$. Given a
measure $\mu_N$ on $\ON$, denote by $\bb P_{\mu_N}$ the
probability measure on $\Ddiscreto$ induced by the
initial state $\mu_N$ and the Markov process $\{\eta_t^N ; t\ge 0\}$.
Expectation with respect to $\bb P_{\mu_N}$ will be denoted by $\bb
E_{\mu_N}$.
Let $\nu^N_\alpha$ be the Bernoulli product
measure on $\Omega_N$ with marginals given by
\begin{equation*}
\nu^N_\alpha \{\eta ; \eta(x) =1\} \;=\; \alpha\,,\quad\forall\; x\in\bb T_N\,.
\end{equation*}
These measures $\{\nu^N_\alpha ; 0\le \alpha \le 1\}$ are invariant, in fact
reversible, for the dynamics described above. 
Denote by $\bb T=[0,1]$ the one-dimensional continuous torus, where we identify the points $0$ and $1$.

\begin{definition}\label{associated}
A sequence of probability measures $\{\mu_N ; N\geq 1 \}$  is
said to be associated to a profile $\rho_0 :\bb T \to [0,1]$ if
\begin{equation}
\label{f09}
\lim_{N\to\infty}
\mu_N \Big[\,\eta\;;\; \Big\vert \pfrac 1N \sum_{x\in\bb T_N} H(\pfrac{x}{N}) \eta(x)
- \int H(u) \rho_0(u) du \Big\vert > \delta\, \Big] \,=\, 0\,,
\end{equation}
for every $\delta>0$ and every continuous functions $H: \bb T \to \bb
R$.
\end{definition}

The quantity introduced in the definition above can be reformulated in terms of
empirical measures.
We start by defining the set
\begin{equation}\label{M}
\mc M\,=\,\big\{\mu\,;\,\mu\mbox{ is a positive measure on }\bb T\mbox{ with }  \mu(\bb T)\leq 1\,\big\},
\end{equation} this space is endowed with the weak topology. 
Consider the measure $\pi^{N}\in\mc M$, which is obtained by reescaling space by $N$ and
by assigning mass $N^{-1}$ to each particle:
\begin{equation*}
% \label{med imp}
\pi^{N}(\eta,du) \,=\, \pfrac{1}{N} \sum _{x\in \bb T_N} \eta(x)\,
\delta_{\frac{x}{N}}(du)\,,
\end{equation*}
where $\delta_u$ is the Dirac measure concentrated on $u$. The measure 
$\pi^{N}(\eta,du)$ is called the empirical measure associated to the configuration $\eta$.
With this notation, $\frac 1N \sum_{x\in\bb T_N} H(\frac{x}{N}) \eta(x)$ is the integral of 
$H$ with respect to the empirical measure $\pi^N$,
denoted by $\<\pi^N, H\>$.  

We consider the time evolution of the empirical measure
$\pi^N_t$ associated to the Markov process $\{\eta_t ; t\ge 0\}$ by:
\begin{equation}
 \label{med imp temp}
\pi_t^{N}(du) \,=\, \pi^{N}(\eta_t,du) \,=\,\pfrac{1}{N} \sum _{x\in \bb T_N} \eta_t(x)\,
\delta_{\frac{x}{N}}(du)\,.
\end{equation}
Note that \eqref{f09} 
is equivalent to say that $\pi^N_0$ converges in distribution to $\rho_0(u)du$.
Throughout the entire paper, it is fixed a time-horizon $T>0$. Let $\mc D\big([0,T], \mc M\big)$ be the space of $\mc M$-valued
c\`adl\`ag trajectories $\pi:[0,T]\to\mc M$  endowed with the
\emph{Skorohod} topology. For short, we will use the notation $\DM=\mc D\big([0,T], \mc M\big)$.
Denote by $\bb Q_{\mu_N}^N$ the measure on
the path space $\DM$ induced by the measure $\mu_N$ and
the empirical process $\pi^N_t$ introduced in \eqref{med imp temp}.

\subsection{Frequently used notations} Before stating results we present some important notations to be used in the entire paper.\medskip

\textbullet \; The indicator function of a set $A$ will be written by $\textbf 1_{A}(u)$, which is
one when $u\in A$ and zero otherwise.\medskip

\textbullet\; 	Given a function $H:\bb T\to \bb R$, we will denote $H(0^-)$ and $H(0^+)$, respectively, for the left and right side limits of $H$ 
at the point $0\in \bb T$.\medskip

\textbullet\; Given a function $H:\bb T\to \bb R$, denote $\delta H(0)=H(0^+)-H(0^-)$ its jump size at zero.
And denote $\delta_N H_x= H(\pfrac{x+1}{N})-H(\pfrac{x}{N})$. Hence, provided $H$ is right continuous at zero, $\delta_N H_{-1}$ converges to  $\delta H(0)$.
\medskip

\textbullet\;   Given a function $g:[0,T]\times \bb T$, we write down
 $g_t(u)$ to denote  $g(t,u)$. 
It should   not be misunderstood  with the notation for time derivative, namely $\partial_t g(t,u)$.\medskip

\textbullet\; Given a non-negative integer $k$, denote by $C^k(\mathbb{T})$ the set of real-valued functions with domain $\mathbb{T}$  with continuous derivatives
 up to order $k$. As natural,  $C(\bb T)$ denotes the set of continuous functions.
For  non-negative integers $j$ and $k$, denote by $C^{j,k}([0,T]\times \bb T)$ the set of real valued functions 
with domain $[0,T]\times \bb T$ with continuous derivatives
  up to order $j$ in the first variable (time), and continuous derivatives up to order $k$ in the second variable (space).\medskip
  
  \textbullet\; The notation $C_{\textrm{k}}$  means compact support contained in $[0,T]\times (0,1)$. For instance,
  $C^{j,k}_{\textrm{k}}([0,T]\times (0,1))$ denotes the subset  of $C^{j,k}([0,T]\times (0,1))$ composed of  functions with compact support contained in $[0,T]\times (0,1)$.  
\medskip

%\textbullet\; The subset of $\DM$ consisting in trajectories taking values on measures which have density with respect to the Lebesgue measure between zero and one will be denoted by $\DMO$. 
%\medskip

\textbullet\; The notation $g(N)=O(f(N))$ means  $g(N)$ is bounded from above by $Cf(N)$, where the constant $C$ does not depend on $N$. The notation $g(N)=o(f(N))$ means 
$\displaystyle \lim_{N\to\infty} g(N)/f(N)=0$.

\textbullet\; Despite we have denoted $\<\pi^N_t, H\>=\frac 1N\sum_{x\in \bb T_N} H(\frac{x}{N})\eta_t(x)$, the bracket $\<\cdot,\cdot\>$ will also mean  the inner product  in $L^2(\mathbb{T})$ and in $L^2[0,1]$. 
The double bracket $\<\!\<\cdot,\cdot\>\!\>$ will denote  the inner product in $L^2([0,T]\times\mathbb{T})$.

\subsection{The  hydrodynamic equation}
\label{ss2.3}

The slow bond, as we will see, yields a discontinuity at the origin in the continuum limit. 
Therefore, discontinuous functions at the origin are naturally required.
\begin{definition}\label{C2descont} 
Denote by $\C$ the space of functions $H:[0,T]\times \bb T\to \bb R$ such that
\begin{enumerate}
 \item $H$ restricted to $[0,T]\times \bb T\backslash \{0\}$ belongs to $C^{1,2}([0,T]\times \bb T\backslash \{0\})$;\medskip
 
 \item Identifying  $\bb T\backslash \{0\}$ with the open interval $(0,1)$, $H$ has a $C^{1,2}$ extension to $[0,T]\times [0,1]$;\medskip
 
 \item For any $t\in [0,T]$, $H$ is right continuous  at zero,  i.e., $H(t,0)=\lim_{x\to 0^+}H(t,x)$.
\end{enumerate}
\end{definition}
This space of test functions should not be misunderstood with 
$C^{1,2}([0,T]\times \bb T)$. In words, a function $H$ belongs to this space $\C$ if, ``opening" the torus at $0$, the function has a $C^{1,2}$ extension  to the closed interval $[0,1]$. 

The bracket $\<\cdot,\cdot\>$ will denote indistinctly the inner product in $L^2(\mathbb{T})$ and in $L^2[0,1]$. 
Let $L^2([0,T]\times \bb T)$ be the Hilbert space of
measurable functions $H: [0,T]\times \bb T\to\bb R$ such that
\begin{equation*}
\int_0^T  \int_{\bb T}  \, \big(H (s, u)\big)^2\, du\, ds\,<\, \infty\,,
\end{equation*}
endowed with the inner product $\<\!\< \cdot ,\cdot \>\!\>$ defined by
\begin{equation*}
\<\!\< H,G \>\!\> \,=\, \int_0^T \int_{\bb T} \, H (s, u)
\, G(s,u)\,du\, ds\,.
\end{equation*}

\begin{definition}[Sobolev Space]\label{Sobolev}
Let  $\mc H^1(0,1)$ be the set of all locally summable functions $\zeta: (0,1)\to\bb R$ such that
there exists a function $\p_u\zeta\in L^2(0,1)$ satisfying
$ \<\partial_uG,\zeta\>\,=\,-\<G,\partial_u\zeta\>$,
for all $G\in C^{\infty}_{\textrm{k}}((0,1))$.
For $\zeta\in\mc H^1(0,1)$, we define the norm
\begin{equation*}
 \Vert \zeta\Vert_{\mc H^1(0,1)}\,:=\, \Big(\Vert \zeta\Vert_{L^2(0,1)}^2+\Vert\partial_u\zeta\Vert_{L^2(0,1)}^2\Big)^{1/2}\,.
\end{equation*}
Let $\Sob$ be the space of
 all measurable functions
$\xi:[0,T]\to \mc H^1(0,1)$ such that
\begin{equation*}
\Vert\xi \Vert_{\Sob}^2 \,
:=\,\int_0^T \Vert \xi_t\Vert_{\mc H^1(0,1)}^2\,dt\,<\,\infty\,.
\end{equation*}
\end{definition}
We refer the reader  to \cite{e,l,te} for more on Sobolev spaces.
\begin{remark}\rm  An  equivalent and useful definition  for the Sobolev space 
$\Sob$ is the  set  of bounded functions
$\xi:[0,T]\times \bb T\to \bb R$ such that there exists a function 
$\partial\xi\in L^2([0,T]\times \bb T)$ 
satisfying 
\begin{equation*}%\label{eqsob}
\<\!\< \partial_u H,\xi \>\!\> \,=\,- \<\!\<  H,\partial\xi \>\!\>\,,
\end{equation*}
for all functions $H\in \Ck$.
\end{remark}

\begin{definition}[The hydrodynamic equation for the SSEP with a slow bond]
 Consider a measurable density profile $\gamma : \bb T
\to [0,1]$. A  function $\rho : [0,T] \times \bb T \to [0,1]$
is said to be a weak solution of the parabolic differential equation with Robin boundary conditions
\begin{equation}\label{edp}
\left\{
\begin{array}{l}
{\displaystyle \partial_t \rho \, =\, \Delta \rho } \\
{\displaystyle \rho_0(\cdot) \,=\, \gamma(\cdot) }\\
\p_u \rho_t (0^+)=\p_u \rho_t(0^-)=\rho_t(0^+)-\rho_t (0^-)\,,
\end{array}
\right.
\end{equation}
if the following two conditions are fulfilled:\medskip

(1) $\rho\in \Sob$\,;\medskip

(2) For all functions  $G\in\C$ and for all $t\in [0,T]$, $\rho$ satisfies the integral equation
\begin{equation}\label{eqint1}
\begin{split}
\< \rho_t, G_t\> - \< \gamma , G_0\> 
= & \int_0^t \< \rho_s , (\partial_s+\Delta)  G_s \>ds\\
& +\int_0^t\big\{\rho_s(0^+)\partial_uG_s(0^+)
-\rho_s(0^-)\partial_uG_s(0^-)\big\}ds\\
&-\int_0^t\big(\rho_s(0^+)-\rho_s(0^-)\big)\,\big(G_s(0^+)-G_s(0^-)\big)ds.
\end{split}
\end{equation} 
\end{definition}
  Assumption (1) guarantees that the boundary integrals are well defined, see \cite{e,l} on the notion of trace of Sobolev spaces.   
The Robin (mixed) boundary conditions in \eqref{edp} can be interpreted as the Fick Law at the point $x=0$. This is discussed in more detail in \cite{fgn1}. The uniqueness and existence of weak solutions of \eqref{edp} was proved in \cite{fgn2}. Moreover, it was proved in \cite{fl,fgn1,fgn2} that

\begin{theorem}\label{t01}
Fix a measurable density profile $\gamma : \bb T
\to [0,1]$ and
consider a sequence of probability measures $\mu_N$ on $\ON$ associated to $\gamma$ in the sense \eqref{f09}. Then, for any $t\in [0,T]$,
\begin{equation}\label{eq888}
\lim_{N\to\infty}
\bb P_{\mu_N} \Big[\, \Big\vert \pfrac 1N \sum_{x\in\bb T_N} 
G(\pfrac{x}{N})\; \eta_t(x) - \int G(u)\, \rho_t(u) \;du \,\Big\vert 
> \delta\, \Big] \,=\, 0\,,
\end{equation}
for every $\delta>0$ and every function $G\in C(\bb{T})$. Here, $\rho$
is the unique weak solution of the linear partial differential equation \eqref{edp}
with $\rho_0=\gamma$.
\end{theorem}
We notice that the result above is a particular case of that considered in \cite{fl}, being the  characterization in terms of a classical partial differential equation given in \cite{fgn1,fgn2}. Moreover, the statement \eqref{eq888} is equivalent to say that $\pi^N_t$ converges in probability to $\rho_t(u)du$.

\subsection{The Weakly Asymmetric Exclusion Process with a slow bond}
 In order to obtain the Large Deviations of a Markov process, a natural step is to prove the LLN for a class of perturbations of the original Markov process. In our case, the correct perturbations will be given by the class of \emph{weakly asymmetric exclusion processes with a slow bond}, to be defined ahead. For short, we will call it just \emph{WASEP with a slow bond}. 

Recall Definition \ref{C2descont}. Given a function $H\in \C$, consider the time non-homogeneous Markov 
process whose generator at time $t$ acts on functions $f:\ON\rightarrow \bb{R}$ as
\begin{equation}\label{lhn}
\begin{split}
(L_{N,t}^H f)(\eta)=&\sum_{x\in \bb T_N}\xi^N_x\,e^{H_t(\frac{x+1}{N})-H_t(\frac{x}{N})}\,\eta(x)\big(1-\eta(x\!+\!1)\big)\Big[f(\eta^{x,x+1})-f(\eta)\Big]\\
+&\sum_{x\in \bb T_N} \xi^N_x\, e^{-H_t(\frac{x+1}{N})+H_t(\frac{x}{N})}\,\eta(x\!+\!1)\big(1-\eta(x)\big)\Big[f(\eta^{x,x+1})-f(\eta)\Big],\\
\end{split}
\end{equation}
where $\eta^{x,x+1}$ is defined in \eqref{eta} and
\begin{equation}\label{xi}
\xi^N_x\;=\;\
\begin{cases}
1\;,\; & \textrm{ if } x\in \bb T_N\backslash \{-1\}\,,\\
N^{-1} \;,\; & \textrm{ if } x=-1\,.
\end{cases}
\end{equation}
In the particular case   $H$ is a constant function, 
the generator  $ L_{N,t}^H$ turns out to be equal to the  generator  $\LN$ defined in \eqref{ln}. We emphasize that the asymmetry is weak in all the bonds except at the bond of vertices $-1,0$. Since the function $H$ is possibly discontinuous at the origin, the asymmetry in that bond does not go to zero in the limit, appearing indeed in the hydrodynamical equation. 

Let $\{\eta_t^H\,; t\geq 0\}$ be the non-homogeneous Markov process  with
 generator $\LNH$ defined in \eqref{lhn} \emph{speeded up by} 
$N^2$.
Given a probability measure $\mu_N$ on $\ON$, denote by $\bb P^H_{\mu_N}$  the probability measure on the space of trajectories $\Ddiscreto$ induced by the Markov process $\{\eta_t^H\,; t\geq 0\}$  starting from the measure $\mu_N$.  

The empirical measure $\pi^{N}_t$ corresponding to $\{\eta_t^H\,; t\geq 0\}$ is defined in the same way of  \eqref{med imp temp}.
Denote  $\chi(\alpha)=\alpha(1-\alpha)$ the mobility function and $\delta H_t(0)=H_t(0^+)-H_t(0^-)$. Next, we present the hydrodynamic equation for the WASEP with a slow bond.
\begin{definition}\label{6.1.1}
\quad
Let   $\gamma : \bb T
\to \bb R$ be a bounded density profile and  fix $H\in \C$.  A function $\rho : [0,T] \times \bb T \to [0,1]$
is said to be a weak solution of the partial differential equation
\begin{equation}\label{edpasy}
\begin{cases}
 \;\partial_t \rho \; =\; \Delta \rho -2\,\p_u\big(\chi(\rho)\p_u H \big) \\
\; \rho_0(\cdot) \;=\; \gamma(\cdot)\\
\;\p_u \rho_t (0^+)\,=\, 2\,\chi\big(\rho_t(0^+)\big)\,\p_u H_t(0^+)-\varphi_t(\rho,H) \,,\\
\;\p_u \rho_t (0^-)\;=\;2\,\chi\big(\rho_t(0^-)\big)\,\p_u H_t(0^-)-\varphi_t(\rho,H) \,,\\
\end{cases}
\end{equation}
where
 \begin{equation} \label{varphi}\varphi_t(\rho,H) =\rho_t(0^-)\big(1-\rho_t(0^+)\big)\,e^{\delta H_t(0)}-\rho_t(0^+)\big(1-\rho_t(0^-)\big)\,e^{-\delta H_t(0)}\,,
 \end{equation}
 if the following two conditions are fulfilled:
\medskip

(1) $\rho\in \Sob$\,;
\medskip

(2) For all functions  $G$ in $\C$, and all $t\in[0,T]$, $\rho$ satisfies the integral equation
\begin{equation}\label{eqint12}
\begin{split}
&\< \rho_t, G_t\> -\< \gamma, G_0\> 
=\int_0^t \big\< \rho_s , (\p_s+\Delta)   G_s \big\> \,ds+2 \int_0^t \big\< \chi (\rho_s) \partial_u H_s , \partial_u G_s \big\> \,ds\\
&+\int_0^t\big\{\rho_s(0^+)\partial_u G_s(0^+)-\rho_s(0^-)\partial_u G_s(0^-) \big\}\,ds +\int_0^t \varphi_s(\rho,H) \,\delta G_s(0)\,ds\,.
\end{split}
\end{equation}
\end{definition}

 \begin{remark} \rm Any classical solution  of \eqref{edpasy} is actually a weak solution of \eqref{edpasy}. To verify it, suppose that $\rho$ is a classical solution.  Then, multiply both sides of the partial differential equation \eqref{edpasy} by a test function $G$ and integrate in time and space. Performing twice integration by parts and applying the boundary conditions leads to the integral equation \eqref{eqint12}. 
\end{remark}
We emphasize the fact we were not able to show uniqueness of weak solutions of \eqref{edpasy} despite the effort of different techniques we have tried. The non-linearity in mixed boundary conditions of \eqref{edpasy} lead to a very complicated problem of uniqueness. Sustaining our point of view that this is only a technical question, in Subsection \ref{unique asy} we prove uniqueness of strong solutions of \eqref{edpasy}. 

Existence of weak solutions of \eqref{edpasy} is a consequence of the tightness of the process, as we will see in Section \ref{weakly chapter}.  
The assumption on uniqueness of weak solutions of \eqref{edpasy} is also needed in the proof of large deviations, because its proof depends on the hydrodynamic limit for the WASEP with a slow bond.

Our first  result is the hydrodynamic limit for the WASEP with a slow bond:

\begin{theorem}\label{t02} Suppose uniqueness of weak solutions of PDE \eqref{edpasy}.
 Let $H\in\C$.
Fix a continuous initial profile $\gamma : \bb T \to [0,1]$ and
consider a sequence of probability measures $\mu_N$ on $\{0,1\}^{\bb
  T_N}$ associated to $\gamma$ in the sense \eqref{f09}. Then, for any $t\in [0,T]$,
\begin{equation*}
\lim_{N\to\infty}
\bb P_{\mu_N}^H \Big[ \, \Big\vert \pfrac 1N \sum_{x\in\bb T_N} 
G(\pfrac{x}{N}) \,\eta_t^H(x) - \int G(u)\, \rho_t(u) \;du\, \Big\vert 
> \delta\, \Big]\,=\, 0\,,
\end{equation*}
for every $\delta>0$ and every function $G\in C(\bb{T})$, where $\rho$
is the unique weak solution of \eqref{edpasy} with $\rho_0=\gamma$.
\end{theorem}

\subsection{Large deviations principle}\label{statement LDP}

Denote by $\mc M_0$ the subset of $\mc M$ of all absolutely continuous
measures with density bounded by $1$:
\begin{equation*}
\mc M_0=\Big\{\omega\in \mc M\;;\;\omega(du)=\rho(u)\,du \quad\mbox{and}\quad
0\leq \rho\leq 1 \quad \mbox{almost surely }\Big\}\,.
\end{equation*}
The set $\mc M_0$ is a closed subset of  $\mc M$ endowed with the weak topology. This property is 
inherited by $\mc D \big([0,T], \mc M_0\big)$, which  is a closed subset of $\DM$ 
for the Skorohod topology. We will denote $\mc D \big([0,T], \mc M_0\big)$ simply by $\DMO$.\medskip

\begin{definition}\label{energy} Given $H\in \Ck$ define
$\mc E_H: \DM\to \bb R\cup\{\infty\}$ by
\begin{equation*}
\mc E_H(\pi)\,=\,\left\{\begin{array}{cl}
\<\!\< \p_u H,\rho \>\!\>
- 2 \<\!\< H,H\>\!\> \,, &  \mbox{if}\,\,\,\,\pi\in \DMO\mbox{ and }\pi(du)=\rho(t,u)\,du\,,\\ 
\infty\,, &\mbox{otherwise\,.}
\end{array}
\right. 
\end{equation*}
Furthermore, define the energy functional $\mc E:\DM\to \bb R_+\cup\{\infty\}$ by
\begin{equation}\label{energia}
 \mc E(\pi)=\sup_{H}\mc E_H(\pi)\,, 
\end{equation}
where the supremum is taken over functions $H\in \Ck$.
\end{definition}
In Section \ref{sobolev section} we prove that  if $\pi\in \DM$ and  $\mc E(\pi)<\infty$, then
there exists $\rho\in \Sob$ such that 
$\pi(t,du)=\rho_t(u)du$. Keeping this in mind, given $H\in \C$ and $\pi\in \DM$,  define
\begin{equation}\label{J hat}
\begin{split}
\hat{J}_H(\pi) & =  \ell_H(\pi)-\Phi_H(\pi)\,,
\end{split}
\end{equation}
where 
\begin{equation}\label{ell}
\begin{split}
  \ell_H(\pi)\;=\;&\<\rho_T,H_T\>-\<\rho_0,H_0\>-\int_0^T \<\rho_t,(\partial_t +\Delta) H_t\> \,dt\\
  &  -\int_0^T\{\rho_t(0^+)\partial_u H_t(0^+)-\rho_t(0^-)\,\partial_u H_t(0^-)\}\,dt\\
&+\int_0^T(\rho_t(0^+)-\rho_t(0^-))\,\delta H_t(0)\,dt
\end{split}
\end{equation}
and  
\begin{equation}\label{Phi}
\begin{split}
\Phi_H(\pi)\;=\;&\int_0^T\<\chi(\rho_t),(\partial_u H_t)^2\>\,dt +\int_0^T\rho_t(0^-) (1-\rho_t( 0^+))\,\psi(\delta H_t(0))\,dt\\
 &+\int_0^T\rho_t( 0^+) (1-\rho_t(0^-))\,\psi(-\delta H_t(0))\,dt\,,
\end{split}
\end{equation}
where $\psi(x)=e^{x}-x-1$ and $\delta H_t(0)=H_t(0^+)-H_t(0^-)$. It is worth highlighting that, as functions of $H$, $\ell_H(\pi)$ is linear and $\Phi_H(\pi)$ is convex.

\begin{definition}\label{J def}
Given $H\in \C$, define the functional  $J_H:\DM\to\bb R\cup\{\infty\}$ by
\begin{equation*}
J_H(\pi)\,=\,\left\{\begin{array}{cl}
\hat{J}_H(\pi), &  \mbox{if}\,\,\,\,\mc E(\pi)<\infty\,,\\ 
\infty, &\mbox{otherwise.}
\end{array}
\right. 
\end{equation*}
\end{definition}

\begin{definition}\label{I def} Let the rate functional $I:\DM\to\bb R_+\cup\{\infty\}$ be
 \begin{equation*}
 I(\pi)=\sup_{H}\, J_H(\pi)\,,
 \end{equation*}
being the supremum above over functions $H\in \C$.
\end{definition}

The large deviations study is decomposed in the study of deviations from the initial measure and deviations from the expected trajectory, see \cite[Chapter 10]{kl}. Since the large deviations for Bernoulli product measures are well known, we restrict ourselves to the deviations from the expected trajectory. We start henceforth the process from a sequence of \textit{deterministic} initial configurations. This avoids the analysis of statical large deviations, since we interested here in dynamical large deviations.
Recall that $\bb Q_{\mu_N}^N$ is the measure on
the path space $\DM$ induced by the initial measure $\mu_N$ and the empirical process $\pi^N_t$ introduced in \eqref{med imp temp}.
We are now in position to state the main result of the paper.

\begin{theorem}\label{t03}
 Let  $\mu_N$ be a sequence of deterministic initial configurations associated to a bounded density profile $\gamma : \bb T
\to \bb R$  in the sense of the Definition \ref{associated}. Then, the sequence of measures $\{\bb Q_{\mu_N}^N;N\geq 1\}$ satisfies the following large deviation estimates:\medskip

  {(i)} {\bf Upper bound:} For any $\mc C$  closed subset of $\DM$, 
 \begin{equation*}
 \varlimsup_{N\to\infty}\pfrac{1}{N}\log \bb Q_{\mu_N}^N\big[\,\mc C\,\big]
\;\leq\; -\inf_{\pi\in\mc C} I(\pi)\,.
 \end{equation*}
 \medskip
 
{(ii)} {\bf Lower bound for smooth profiles:} For any $\mc O$ open subset of $\DM$,
\begin{equation*}
 \varliminf_{N\to\infty}
\pfrac{1}{N}\log\bb Q_{\mu_N}^N\big[\,\mc O\,\big]\; \geq \; -\inf_{\pi\in \mc O\cap \mc D^{\mc S}_{ \mc M_0}}I(\pi)\,,
\end{equation*}
where $\mc D^{\mc S}_{ \mc M_0}$ denotes the set of paths $\pi\in \DM$ such that
$\pi_t(du)=\rho_t(u)\,du$ with $\rho\in \C$.

\end{theorem}

The item $(i)$ of theorem above is proved in Section \ref{upper}. The proof of item $(ii)$  is presented in Section \ref{lower}.

\section{Superexponential replacement lemmas and energy estimate}\label{chapter super}
Both in the proof of  hydrodynamic limit for the WASEP with a slow bond and in the proof of the large deviations principle for the SSEP with a slow bond, replacement lemma and energy estimates play an important role. 
By a \emph{replacement lemma} we mean a result  that allows to replace the average time occupation in a site for the average time occupation in a box around that site. And by \emph{energy estimates} we mean a result assuring that time trajectories of the empirical measure are asymptotically close to elements of a certain Sobolev space.

In the proof of large deviations we will need such results in a superexponential setting.  In other words,   the corresponding probabilities must converge to one in a faster way than exponentially. 
\subsection{Definitions and estimates lemmas}\label{estimates}

Denote by $\bs{H} (\mu_N | \nu_\alpha^N)$ the entropy of a probability
measure $\mu_N$ with respect to the invariant measure $\nu_\alpha^N$. 
For a precise definition and properties of the entropy, see \cite{kl}. 
It is well known the existence of a constant $K_0:=K_0(\alpha)$ such that
\begin{equation}
\label{f06}
\bs{H} (\mu_N | \nu_\alpha^N) \,\le\, K_0 N\,,
\end{equation}
for any probability measure $\mu_N$ in $\ON$. See for instance the appendix of \cite{fgn1}.

Denote by $\< \cdot, \cdot \>_{\nu_\alpha^N}$ the scalar product of
$L^2(\nu_\alpha^N)$  and denote by $\mf D_N$  the Dirichlet form, which is the convex and lower
semicontinuous functional (see \cite[Corollary A1.10.3]{kl}) defined by
\begin{equation*}
\mf D_N (f) \,=\, \< - \LN \sqrt f \,,\, \sqrt f\>_{\nu_\alpha^N}\,,
\end{equation*}
where $f$ is a probability density with respect to $\nu_\alpha^N$ (i.e.
$f\ge 0$ and $\int \! f d\nu_\alpha^N =1$). An elementary computation shows
that
\begin{equation*}
\mf D_N (f) \,=\, \sum_{x\in \bb T_N} \frac{ \xi_{x}^N}{2}
\int  \Big( \sqrt{f(\eta^{x,x+1})} -
\sqrt{f(\eta)} \Big)^2 \, d\nu_\alpha^N(\eta) \,,
\end{equation*}
where $\xi_{x}^N$ is defined in \eqref{xi}.

\medskip

From this point on, abusing of notation, we denote the  biggest integer small or equal to $\eps N$ simply by $\eps N$.
Next, we define the local average of particles,
which corresponds to the mean occupation in a box around a given site. The idea is to define a box around the site $x$ in such a  way it avoids the slow bond.

\begin{definition}\label{media}
If $x\in \bb T_N$ is such that $\pfrac{x}{N}\in\bb T\backslash (-\eps,0)$,  we define the local average by
\begin{equation*}
\eta^{\eps N}(x)\,=\,\pfrac{1}{\eps N}\sum_{y=x+1}^{x+\eps N}\eta(y)\,.
\end{equation*}
If  $\pfrac{x}{N}\in (-\eps,0)$, define the local average by
\begin{equation*}
\eta^{\eps N}(x)\,=\,\pfrac{1}{\eps N}\sum_{y=-\eps N}^{-1}\eta(y)\,.
\end{equation*}
\end{definition}
In accordance with to the previous definition of local density of particles, we define an approximation of identity $\iota_\eps$ in the continuous torus  by
\begin{equation}\label{iota}% as funções indicadoras precisam ser 
%indicadoras de conjuntos abertos, pois usaremos isso na semicontinuidade inferior nos Grandes Desvios.
\iota_\eps(u,v)=\left\{\begin{array}{ll}
\pfrac{1}{\eps}\,\textbf 1_{(v,v+\eps)}(u)\,, &\mbox{~if}
\,\,\,\,v\in \bb T\backslash (-\eps,0)\,,\\
\quad \\
\pfrac{1}{\eps}\,\textbf 1_{(-\eps,0)}(u)\,, &  \mbox{~if}\,\,\,\,v\in (-\eps,0)\,.\\ 
\end{array}
\right.
\end{equation}
We also define the convolution 
\begin{equation*}
 (\psi*\ioe)(v)=\< \psi,\iota_\eps(\cdot,v)\>\,,
\end{equation*}
for a function $\psi:\bb T\to \bb R$ or a measure $\psi$ on the torus $\bb T$. The following identity is relevant:
\begin{equation}\label{identity}
(\pi^N*\ioe)(\pfrac{x}{N})\;=\;\eta^{\eps N}(x)\,,\quad \textrm{ for all }\; x\in\bb T_N\,.
\end{equation}
To simplify notation,  define the functions
\begin{equation}\label{g1}
g_1: \{0,1\}^{\bb T} \to \bb R \,\,\, \textrm{by}\,\,\, g_1(\eta)=\eta(0)(1-\eta(1))
\end{equation}
and
\begin{equation*}
\tilde g_1 :[0,1]\times[0,1]\to \bb R\,\,\, \textrm{by}\,\,\,
\tilde g_1 (\alpha,\beta) \,=\,\alpha(1-\beta)\,.
\end{equation*}
Also,
\begin{equation}\label{g2}
g_2: \{0,1\}^{\bb T} \to \bb R \,\,\, \textrm{by}\,\,\,
 g_2(\eta)=\eta(1)(1-\eta(0))
\end{equation}
and
\begin{equation*}
\tilde g_2 :[0,1]\times[0,1]\to \bb R\,\,\, \textrm{by}\,\,\,
\tilde g_2 (\alpha,\beta) \,=\,\beta(1-\alpha)\,.
\end{equation*}

\begin{lemma}\label{s01}
Fix  $F:\bb T\to\bb R $ and let $f$ be a density with respect to $\nu_\alpha^N$. Then, for any $A>0$ hold the inequalities
\begin{equation}\label{for super1}
\displaystyle
\begin{split}
&\pfrac{1}{N}\sum_{x\neq  -1} 
\int F(\pfrac{x}{N})\Big\{\tau_x g_i(\eta)-\tilde{g}_i(\eta^{\eps N}(x), 
\eta^{\eps N}(x\!+\!1))\Big\}
f(\eta)\,d\nu_\alpha^N (\eta)\\
&\leq \,12 A\eps\sum_{x\neq  -1} \big(F(\pfrac{x}{N})\big)^2\,+\,\pfrac{3}{A}\, \mf D_N(f),
\end{split}
\end{equation}
\begin{equation}\label{for replace}
\begin{split}
&\pfrac{1}{N} \sum_{x\in{\bb{T}_N}}\int  F(\pfrac{x}{N})\{\eta(x)-\eta^{\eps N}(x)\}f(\eta)\,d\nu_\alpha^N (\eta) 
\\
&\leq\, 4 A \eps\sum_{x\in{\bb{T}_N}} \big(F(\pfrac{x}{N})\big)^2\,+\,\pfrac{1}{A}\,\mf D_N(f)\,,
\end{split}
\end{equation}
\begin{equation}\label{for super2}
\displaystyle
\begin{split}
 F(\pfrac{-1}{N})&\int\Big\{\tau_{-1}
 g_i(\eta)-\tilde{g}_i(\eta^{\eps N}(-1),\eta^{\eps N}(0))\Big\}
f(\eta)\,d\nu_\alpha^N (\eta) \\
& \leq \,6A\eps N \big(F(\pfrac{-1}{N})\big)^2\,+\,\pfrac{3}{A}\,\mf D_N(f),
\end{split}
\end{equation}
\begin{equation}\label{for super one}
\int  \{\eta(x)-\eta^{\eps N}(x)\}f(\eta)\,d\nu_\alpha^N (\eta) 
\leq \, 4NA \eps\,+\,\pfrac{1}{A}\,\mf D_N(f)\,,\qquad\forall x\in\bb T_N\,,
\end{equation}
with $i=1,2$.
\end{lemma}

\begin{proof} The method of proof for the four inequalities is exactly the same. For this reason, we detail only the inequality \eqref{for super1} with $i=1$. The reader can check the remaining inequalities.
First, adding and subtracting terms, we rewrite $\tau_x g_1(\eta)-\tilde{g}_1(\eta^{\eps N}(x), 
\eta^{\eps N}(x\!+\!1))$ as
\begin{equation}\label{eq27a}
 \eta(x)-\eta^{\eps N}(x)-\eta(x)(\eta(x\!+\!1)-\eta^{\eps N}(x\!+\!1))-\eta^{\eps N}(x\!+\!1)(\eta(x)-\eta^{\eps N}(x))\,.
\end{equation}
We handle the parcel $\eta(x)(\eta(x\!+\!1)-\eta^{\eps N}(x\!+\!1))$ of above first.
We claim that for $f$ density with respect to $\nu_\alpha^N$ and for any $A>0$, it is true that
\begin{equation}\label{af 1}
\begin{split}
&\pfrac{1}{N}\sum_{x\neq  -1} 
\int F(\pfrac{x}{N})\eta(x)\Big\{\eta(x\!+\!1)-\eta^{\eps N}(x\!+\!1)\Big\}
f(\eta)\,d\nu_\alpha^N (\eta)\\
&\leq  4A\eps\sum_{x\neq  -1} \big(F(\pfrac{x}{N})\big)^2+\pfrac{1}{A}  \mf D_N(f)\,.
\end{split}
\end{equation}
Recall Definition \ref{media}. Let $x$ be such that $\pfrac{x+1}{N}\notin (-\eps,0]$.
In this case,
\begin{equation*}
\begin{split}
&\int F(\pfrac{x}{N})\eta(x)(\eta(x\!+\!1)-\eta^{\eps N}(x\!+\!1))f(\eta)\,d\nu_\alpha^N (\eta)\\
&=\int F(\pfrac{x}{N})\eta(x)\Big\{\pfrac{1}{\eps N}\sum_{y=x+2}^{x+1+\eps N}(\eta(x\!+\!1)
-\eta(y))\Big\}f(\eta)\,d\nu_\alpha^N (\eta)\,.
\end{split}
\end{equation*}
Replacing $\eta(x\!+\!1)-\eta(y)$ by a telescopic sum, one can rewrite the expression above as
\begin{equation*}
\int F(\pfrac{x}{N})\eta(x)\Big\{\pfrac{1}{\eps N}\sum_{y=x+2}^{x+1+\eps N}
\sum_{z=x+1}^{y-1}(\eta(z)-\eta(z+1))\Big\}f(\eta)\,d\nu_\alpha^N (\eta)\,.
\end{equation*}
Rewriting the last expression as twice the half and making the change of variables
$\eta\mapsto \eta^{z,z+1}$ (and using that the probability $\nu_\alpha^N$ is invariant for this map) it becomes
\begin{equation*}
 \pfrac{1}{2\eps N}\sum_{y=x+2}^{x+1+\eps N}
\sum_{z=x+1}^{y-1}F(\pfrac{x}{N})\int \eta(x)(\eta(z)-\eta(z+1))(f(\eta)-f(\eta^{z,z+1}))\,d\nu_\alpha^N (\eta)\,.
\end{equation*}
By means of  $a-b=(\sqrt{a}-\sqrt{b})(\sqrt{a}+\sqrt{b})$ and the Cauchy-Schwarz inequality,  we bound the previous expression
from above by
\begin{equation}\label{eqpag17}
\begin{split}
&\pfrac{1}{2\eps N}\sum_{y=x+2}^{x+1+\eps N}
\sum_{z=x+1}^{y-1}\pfrac{A}{\xi^N_{z}}\big(F(\pfrac{x}{N})\big)^2
\int \Big(\sqrt{f(\eta)}+\sqrt{f(\eta^{z,z+1})}\Big)^2\,d\nu_\alpha^N (\eta)\\
&+\pfrac{1}{2\eps N}\sum_{y=x+2}^{x+1+\eps N}
\sum_{z=x+1}^{y-1} \pfrac{\xi^N_{z}}{A}\int
\Big(\sqrt{f(\eta)}-\sqrt{f(\eta^{z,z+1})}\Big)^2 \,d\nu_\alpha^N (\eta),
\end{split}
\end{equation}
for any $A>0$ and where $\xi^N_{z}$ was defined in \eqref{xi}. The second sum above is bounded by
\begin{equation*}
\pfrac{1}{A\eps N}\sum_{y=x+2}^{x+1+\eps N}
\sum_{z\in \bb T_N} \pfrac{\xi^N_{z}}{2}\int \Big(\sqrt{f(\eta)}-\sqrt{f(\eta^{z,z+1})}\Big)^2 
\,d\nu_\alpha^N (\eta)\leq \pfrac{1}{A} \mf D_N(f)\,.
\end{equation*}
Since $\xi^N_{z}=1$ for all $z\in\{x+1,\ldots,x+\eps N\}$ and  $f$ is   a density with respect to $\nu_\alpha^N$,
the first term in \eqref{eqpag17} is bounded  by
\begin{equation*}
\pfrac{1}{\eps N}\sum_{y=x+2}^{x+1+\eps N}
\sum_{z=x+1}^{y-1}2A\big(F(\pfrac{x}{N})\big)^2\leq 2A\eps N\big(F(\pfrac{x}{N})\big)^2\,.
\end{equation*}
Therefore, for any site $x$ such that $\pfrac{x+1}{N}\notin (-\eps,0]$,
\begin{equation*}
\int\! F(\pfrac{x}{N})\eta(x)(\eta(x\!+\!1)-\eta^{\eps N}(x\!+\!1))f(\eta)\,d\nu_\alpha^N (\eta)\leq 
  2A\eps N\big(F(\pfrac{x}{N})\big)^2+\pfrac{1}{A} \mf D_N(f)\,.
\end{equation*}
Now, let $x$ be a site such that $\pfrac{x+1}{N}\in (-\eps,0]$. In this case,
\begin{equation*}
\begin{split}
&\int F(\pfrac{x}{N})\eta(x)(\eta(x\!+\!1)-\eta^{\eps N}(x\!+\!1))f(\eta)\,d\nu_\alpha^N (\eta)\\
&=\int F(\pfrac{x}{N})\eta(x)\Big\{\pfrac{1}{\eps N}\sum_{y=-\eps N}^{-1}(\eta(x\!+\!1)-\eta(y))\Big\}f(\eta)\,d\nu_\alpha^N (\eta)\,.
\end{split}
\end{equation*}
We split the last sum into two blocks: $\{-1-\eps N+1,\ldots,x\}$ and $\{x+1,\ldots,-1\}$. Then we proceed by
 writing $\eta(x\!+\!1)-\eta(y)$ as a telescopic sum, getting
\begin{equation*}
\begin{split}
&F(\pfrac{x}{N})\int \eta(x)\Big\{\pfrac{1}{\eps N}\sum_{y=-\eps N}^{x}\sum_{z=y}^{x}(\eta(z+1)-\eta(z))\Big\}f(\eta)\,d\nu_\alpha^N (\eta)\\
&+F(\pfrac{x}{N})\int \eta(x)\Big\{\pfrac{1}{\eps N}\sum_{y=x+2}^{-1}\sum_{z=x+1}^{y-1}(\eta(z)-\eta(z+1))\Big\}f(\eta)\,d\nu_\alpha^N (\eta)\,.
\end{split}
\end{equation*}
By the same arguments used above and since $\xi^N_{z}=1$
for all $z$ in the range $\{-\eps N,\ldots,-2\}$,
we bound the previous expression by
\begin{equation*}
4A\eps N\big(F(\pfrac{x}{N})\big)^2+\pfrac{1}{A} \mf D_N(f)\,.
\end{equation*}
This proves \eqref{af 1}. Analogous bounds for the remaining parcels in \eqref{eq27a} lead to \eqref{for super1}. 

\end{proof}

\begin{lemma}\label{certo}
Fix any function $H:\bb T\to\bb R$ and let $f$ be a density with respect to $\nu_\alpha^N$. Then, 
\begin{equation}\label{eq_do_lema35}
\begin{split}
&\int\pfrac 1{\eps N}\sum_{x\in\bb T_N} H(\pfrac{x}{N})  
\, \big\{ \eta (x-\eps N) -
\eta (x) \big\}
f(\eta)\,d\nu_\alpha^N (\eta)\\
&\leq N \,\mf D_N(f) +\pfrac {2}{N} \sum_{x\in\bb T_N} 
\big(H(\pfrac{x}{N})\big)^2 \Big\{ 1+\pfrac{1}{\eps}{\mathbf{1}}_{(-\eps,0]}(\pfrac{x}{N})\Big\}\,.
\end{split}
\end{equation}
Moreover, this inequality  remains valid replacing $\{ \eta (x-\eps N) -\eta (x) \}$ by $\{ \eta (x) -\eta (x+\eps N) \}$.  
\end{lemma}

\begin{proof} We begin by writing the left hand side of inequality \eqref{eq_do_lema35} as a telescopic sum: 
 \begin{equation*}
\begin{split}
 & \int\pfrac 1{\eps N}\sum_{x\in\bb T_N} H(\pfrac{x}{N})  
\, \big\{ \eta (x_0) -
\eta (x_1) \big\}
f(\eta)\,d\nu_\alpha^N (\eta)\\
 & \;=\;\pfrac 1{\eps N}\sum_{x\in\bb T_N} H(\pfrac{x}{N})\sum_{y=x_0}^{x_1-1}\int  
\, \big\{ \eta (y) -
\eta (y+1) \big\}
f(\eta)\,d\nu_\alpha^N (\eta)\,,
\end{split}
\end{equation*}
where $x_0=x-\eps N$ and $x_1=x$ (or $x_0=x$ and $x_1=x+\eps N$ for the second case). Rewrite the expression above as twice the half. Then, making the changing of variables $\eta\mapsto \eta^{x,x+1}$ on one piece and applying Young's inequality, we bound the previous expression by 
\begin{equation}\label{2sum}
\begin{split}
&\pfrac{1}{\eps N}\sum_{x\in\bb T_N} \big(H(\pfrac{x}{N})\big)^2\,\sum_{y=x_0}^{x_1-1}\pfrac{A}{2\xi^N_{y}}\int  
\, \Big\{\sqrt{f(\eta)}+\sqrt{f(\eta^{y,y+1})}\,\Big\}^2\,d\nu_\alpha^N (\eta)\\
&+\pfrac{1}{\eps N}\sum_{x\in\bb T_N}\,\sum_{y=x_0}^{x_1-1}\pfrac{\xi^N_{y}}{2A}\int  
\Big\{\sqrt{f(\eta)}-\sqrt{f(\eta^{y,y+1})}\,\Big\}^2\,d\nu_\alpha^N (\eta)\,,\qquad \forall A>0\,,
\end{split}
\end{equation}
where $\xi^N_y$ was defined in \eqref{xi}.
The second sum above is less than or equal to
$\pfrac{1}{A }\mf D_N(f)$.
Since $f$ is  a density with respect to $\nu_\alpha^N$,
the first sum in \eqref{2sum} is smaller or equal than
\begin{equation*}
 \pfrac{1}{\eps N}\sum_{x\in\bb T_N} \big(H(\pfrac{x}{N})\big)^2 \sum_{y=x_0}^{x_1-1}\pfrac{2A}{\xi^N_{y}}
\;\leq\; \pfrac{2A}{\eps N}\sum_{x\in\bb T_N}\big(H(\pfrac{x}{N})\big)^2\Big\{\eps N+N\mathbf{1}_{(-\eps,0]}(\pfrac{x}{N})\Big\}\,.
\end{equation*}
This inequality is true for $x_0=x-\eps N$ and $x_1=x$ or $x_0=x$ and $x_1=x+\eps N$.
Choosing $A=\frac{1}{N}$ completes the proof.
\end{proof}

\subsection{Superexponential replacement lemmas}\label{super_replac_lem}
In the large deviations proof, the replacement lemma presented in Section \ref{replace} is not enough,
 because we need to prove that the difference between cylinder
functions and functions of the density field are superexponentially
 small, that is, of order
smaller that $\exp \{-CN\}$, for any $C>0$. We begin by exhibiting a superexponential replacement for the invariant measure $\nu_\alpha^N$.

\begin{proposition}\label{super}
Let  $F_i\!:[0,T]\times\bb T\to \bb R$, $i=1,2$, such that 
\begin{equation*}
\varlimsup_{N\to\infty} \int_0^T  \Big( (F_2(t,\pfrac{-1}{N}))^2+
\pfrac{1}{N}\sum_{x\neq -1}\big(F_1(t,\pfrac{x}{N})\big)^2\Big)\,dt\;<\;\infty\,.
\end{equation*}
For each $\eps>0$, consider
\begin{equation*}
\begin{split}
V^{F_1,F_2}_{N, \eps}(t,\eta)\;:=\;&\pfrac{1}{N}\sum_{x\neq -1} 
F_1(t,\pfrac{x}{N})\Big\{\tau_x g_1(\eta)-\tilde{g}_1(\eta^{\eps N}(x), 
\eta^{\eps N}(x+1))\Big\}\\
&+F_2(t,\pfrac{-1}{N})\Big\{\tau_{-1}
 g_1(\eta)-\tilde{g}_1(\eta^{\eps N}(-1),\eta^{\eps N}(0))\Big\}\,,
\end{split}
\end{equation*}
where $g_1$ and  $\tilde{g}_1$ have been defined in \eqref{g1}.
Then, for any $\delta>0$, 
\begin{equation}\label{super lim}
 \varlimsup_{\eps\downarrow 0}\varlimsup_{N\to\infty}\pfrac 1N \log \bb P_{\nu_\alpha^N}\Big[\,
\Big\vert \int_0^T V^{F_1,F_2}_{N, \eps}(t,\eta_t)\,dt\Big\vert>\delta\,\Big]\;=\;-\infty\,.
\end{equation}
Finally, it is true the same result with $g_2$ and
$\tilde g_2$ in lieu of $g_1$ and $\tilde g_1$.
\end{proposition}

\begin{proof}

By 
\begin{equation}\label{bound_log}
\varlimsup_N N^{-1} \log \{a_N + b_N\}\;= \;
\max\Big\{\varlimsup_N N^{-1} \log a_N,\varlimsup_N N^{-1} \log b_N\Big\}\,,
\end{equation} it is enough to prove \eqref{super lim} without the absolute value for 
$ V^{F_1,F_2}_{N,\eps}$ and $ V^{-F_1,-F_2}_{N,\eps}$.
Let $C>0$. By Chebyshev exponential inequality, we get 
\begin{equation*}
\begin{split}
 &\bb P_{\nu_\alpha^N}\Big[\int_0^T V^{F_1,F_2}_{N,\eps}(s,\eta_s)\,ds>\delta\Big]\\
 &\leq\;\exp{\{-C\delta N\}}\bb E_{\nu_\alpha^N}\Big[\exp\Big\{CN\int_0^T V^{F_1,F_2}_{N,\eps}(s,\eta_s) \,ds\Big\}\Big]\,.
\end{split}
 \end{equation*}
To conclude the proof it is enough to assure that
\begin{equation}\label{3.3}
 \varlimsup_{\eps\downarrow 0}\varlimsup_{N\to\infty}\pfrac 1N \log \bb E_{\nu_\alpha^N}\Bigg[\exp{\Big\{
\int_0^T \!\! CN \,\,V^{F_1,F_2}_{N, \eps}(t,\eta_t)\,dt\Big\}}\Bigg]\;\leq\; 0\,,
\end{equation}
for every $C>0$, because in this case we would have proved that left hand side of \eqref{super lim} 
is bounded from above by $-C\delta$ for an arbitrary $C>0$. By the Feynman-Kac formula, for each fixed $N$ the previous expectation is
bounded from above by
\begin{equation*}
\exp\Big\{\int_0^T\sup_{f}\Big[\int C\!N \,V^{F_1,F_2}_{N,\eps}(t,\eta)f(\eta)d\nu_\alpha^N(\eta)-N^2\mf D_N(f)\Big] \,dt\Big\}\,,
\end{equation*}
where  the supremum  is carried  over all  density functions  $f$ with
respect to  $\nu_\alpha^N$. Replacing the expression of $V^{F_1,F_2}_{N,\eps}(t,\eta)$ 
and using the Lemma \ref{s01} (notice that this lemma works for $g_1$ and $g_2$), we bound the expression in \eqref{3.3}  by
\begin{equation*}
\int_0^T \!\!\sup_{f}\Big[6CA\eps \Big(2\!\sum_{x\neq -1}\!\big(F_1(t,\pfrac{x}{N})\big)^2+
 N(F_2(t,\pfrac{-1}{N}))^2\Big)+\pfrac{6C}{A}\mf D_N(f)-N\mf D_N(f)\Big]dt.
\end{equation*}
Choosing $A=\pfrac{6C}{N}$, the expression above becomes
\begin{equation*}
36C^2\eps\int_0^T  \Big(\pfrac{2}{N}\sum_{x\neq -1}\big(F_1(t,\pfrac{x}{N})\big)^2+ (F_2(t,\pfrac{-1}{N}))^2\Big)\,dt\,,
\end{equation*}
which vanishes as $N\to\infty$ and then
$\eps\downarrow 0$.
\end{proof}

\begin{corollary}\label{super_lemma}
Under the same hypothesis of the Proposition \ref{super}, 
 for any $\delta>0$, 
\begin{equation}\label{super lim_Non_equi}
 \varlimsup_{\eps\downarrow 0}\varlimsup_{N\to\infty}\pfrac 1N \log \bb P_{\mu_N}\Big[\,
\Big\vert \int_0^T V^{F_1,F_2}_{N, \eps}(t,\eta_t)\,dt\Big\vert>\delta\,\Big]\;=\;-\infty\,.
\end{equation}
Finally, the same result is still valid with $g_2$ and
$\tilde g_2$ in lieu of $g_1$ and $\tilde g_1$.
\end{corollary}

\begin{proof} 
By the bound $\nu_\alpha^N(\eta)\geq (\alpha\wedge(1-\alpha))^N$, we get 

%\begin{equation*}
%\begin{split}
%&  \bb P_{\mu_N}\Big[\,
%\Big\vert \int_0^T V^{F_1,F_2}_{N, \eps}(t,\eta_t)\,dt\Big\vert>\delta\,\Big]
%%\;=\; \pfrac 1N \log \Bigg(\sum_{\eta\in \Omega_N}\bb P_{\eta}\Big[\,
%%\Big\vert \int_0^T V^{F_1,F_2}_{N, \eps}(t,\eta_t)\,dt\Big\vert>\delta\,\Big]\,\mu_N(\eta)\Bigg)\\
%\;=\; \sum_{\eta\in \Omega_N}\bb P_{\eta}\Big[\,
%\Big\vert \int_0^T V^{F_1,F_2}_{N, \eps}(t,\eta_t)\,dt\Big\vert>\delta\,\Big]\,\frac{\mu_N(\eta)}{\nu_\alpha^N(\eta)}\,\nu_\alpha^N(\eta)\\\end{split}
%\end{equation*}

\begin{equation*}
\begin{split}
&  \pfrac 1N \log \bb P_{\mu_N}\Big[\,
\Big\vert \int_0^T V^{F_1,F_2}_{N, \eps}(t,\eta_t)\,dt\Big\vert>\delta\,\Big]\\
&\;=\;\pfrac 1N \log \Bigg(\sum_{\eta\in \Omega_N}\bb P_{\eta}\Big[\,
\Big\vert \int_0^T V^{F_1,F_2}_{N, \eps}(t,\eta_t)\,dt\Big\vert>\delta\,\Big]\,\frac{\mu_N(\eta)}{\nu_\alpha^N(\eta)}\,\nu_\alpha^N(\eta)\Bigg)\\
&\leq\; \pfrac 1N \log \Bigg(\frac{1}{(\alpha\wedge(1-\alpha))^N}\,\sum_{\eta\in \Omega_N}\bb P_{\eta}\Big[\,
\Big\vert \int_0^T V^{F_1,F_2}_{N, \eps}(t,\eta_t)\,dt\Big\vert>\delta\,\Big]\,\nu_\alpha^N(\eta)\Bigg)\\
&=\;  \log \Big(\frac{1}{(\alpha\wedge(1-\alpha))}\Big)+\pfrac 1N \log \bb P_{\nu_\alpha^N}\Big[\,
\Big\vert \int_0^T V^{F_1,F_2}_{N, \eps}(t,\eta_t)\,dt\Big\vert>\delta\,\Big]\;.\end{split}
\end{equation*}
 Recalling Proposition \ref{super} and $0<\alpha<1$, the limit \eqref{super lim_Non_equi} follows.
\end{proof}

\begin{corollary}\label{super one}
Given a bounded function $F:[0,T]\times \bb T$ and $x=-1$ or $x=0$, let
\begin{equation*}
\hat{V}^{F,x}_{N, \eps}(t,\eta)=F(t,\pfrac{x}{N})\{\eta(x)-\eta^{\eps N}(x)\}\,.
\end{equation*}
Then, for any $\delta>0$, 
\begin{equation}\label{Fhat}
 \varlimsup_{\eps\downarrow 0}\varlimsup_{N\to\infty}\pfrac 1N \log \bb P_{\mu_N}\Big[\,
\Big\vert\! \int_0^T \hat{V}^{F,x}_{N, \eps}(t,\eta_t)\,dt\Big\vert>\delta\,\Big]\;=\;-\infty\,.
\end{equation}
\end{corollary}

\begin{proof}
We will prove  the limit \eqref{Fhat} for $\mu_N=\nu_\alpha^N$, to do this    is enough to prove
\begin{equation*}
 \varlimsup_{\eps\downarrow 0}\varlimsup_{N\to\infty}\pfrac 1N \log \bb P_{\nu_\alpha^N}\Big[\,
\int_0^T \hat{V}^{G,x}_{N, \eps}(t,\eta_t)\,dt\,>\delta\,\Big]\;=\;-\infty\,,
\end{equation*}
for $G=F$ and $G=-F$. This limit follows in the same sense as in the Proposition \ref{super} and  using \eqref{for super one} from Lemma \ref{s01}.  The extension for a general $\mu_N$
   follows the same scheme in the proof of Corollary \ref{super_lemma} and it is omitted here.

\end{proof}

The next  lemma is  useful to get the results of the Subsection \ref{replace_section}  from the results of this subsection.

\begin{lemma}\label{superexp_replace}
If for any function $W_\eps^N(t,\eta_t)$ uniformly bounded by $ C$ and 
 for any $\delta>0$ we have
\begin{equation*}
 \varlimsup_{\eps\downarrow 0}\varlimsup_{N\to\infty}\pfrac 1N \log \bb P_{\mu_N}\Big[\,
\Big\vert \int_0^T W_\eps^N(t,\eta_t)\,dt\Big\vert>\delta\,\Big]\;=\;-\infty\,,
\end{equation*} 
then 
\begin{equation*}
 \varlimsup_{\eps\downarrow 0}\varlimsup_{N\to\infty} \bb E_{\mu_N}\Big[\,
\Big\vert \int_0^TW_\eps^N(t,\eta_t)\,dt\Big\vert\,\Big]\;=\;0\,.
\end{equation*}
\end{lemma}

\begin{proof}
Using that $W_\eps^N(t,\eta_t)$ is uniformly bounded by $ C$, the expectation  
$ \bb E_{\mu_N}\Big[\,
\Big\vert \int_0^T W_\eps^N(t,\eta_t)\,dt\Big\vert\,\Big]$ is bounded from above by
\begin{equation*}
\delta\,\bb P_{\mu_N}\Big[\,
\Big\vert \int_0^T W_\eps^N(t,\eta_t)\,dt\Big\vert\leq \delta\,\Big]+CT\,\bb P_{\mu_N}\Big[\,
\Big\vert \int_0^T W_\eps^N(t,\eta_t)\,dt\Big\vert>\delta\,\Big]\,,
\end{equation*} 
for any $\delta>0$.
Since for all $\delta>0$ and $M>0$,  there exists $\eps_0$ and $N_0$ such that
\begin{equation*}
 \bb P_{\mu_N}\Big[\,
\Big\vert \int_0^T W_\eps^N(t,\eta_t)\,dt\Big\vert>\delta\,\Big]\;\leq \;e^{-N\,M}<\delta/C\,,\quad\forall N\geq N_0\mbox{ and }\forall \eps<\eps_0\,,
\end{equation*} 
we have
\begin{equation*}
 \bb E_{\mu_N}\Big[\,
\Big\vert \int_0^T W_\eps^N(t,\eta_t)\,dt\Big\vert\,\Big]\;\leq \;
2\delta\,,\quad\forall N\geq N_0\mbox{ and }\forall \eps<\eps_0\,,
\end{equation*} which finishes the proof.
\end{proof}

\subsection{Superexponential energy estimate}\label{energy section}

Our goal here is to exclude trajectories
with infinite energy in the large deviations regime.  The  next proposition is the key in the energy estimates. 
\begin{proposition}\label{-l}  Recall the Definition \ref{energy} of $\mc E_H$. For any function  $H\in\Ck$, the following inequality holds:
\begin{equation*}
 \varlimsup_{\eps\downarrow 0}\varlimsup_{N\to \infty}\pfrac{1}{N}
\log \bb P_{\mu_N}\Big[\,\mc E_H(\pi^N*\ioe)\geq l\,\Big]\;\leq\; -l+K_0\;,\qquad\, \forall \;l\in \bb R\,.
\end{equation*} 
\end{proposition}

\begin{proof}

We begin by claiming that, for  $\eps>0$  small enough, 
\begin{equation}\label{claim}
\int_0^T\!\!\!\int_{\bb T}\p_v H(t,v)(\pi^N_t*\ioe)(v)\,dv\,dt
=\int_ 0^T\!\!\pfrac{1}{\eps N}\sum_{x\in {\bb T}_N}H(t,\pfrac{x}{N})\,[\eta_t(x)-\eta_t(x+\eps N)]\,
 dt.
\end{equation}
Since $H$ has support contained in $[0,T]\times (\bb T\backslash \{0\})$ (using the identification of $(0,1)$ with $\bb T\backslash \{0\}$), there exists some $\eps_0>0$ such that
$H(t,v)$ vanishes if $v\in  (-\eps_o, \eps_0)$, for all $t\in [0,T]$. Applying Fubini's Theorem,
\begin{equation*}
\begin{split}
&\int_ 0^T\!\!\!\!\int_{\bb T} \partial_u H(t,v)(\pi^N_t*\ioe)(v)dv\,dt \!=\!\!\int_ 0^T\!\!\!\!\pfrac 1N\!\! \sum_{x\in {\bb T}_N}
\eta_t(x)\Big(\!\int_{\bb T} \partial_u H(t,v)\ioe (\pfrac{x}{N},v)dv\!\Big)dt.
\end{split}
\end{equation*}
From the definition of $\ioe$ given in \eqref{iota} and taking $0<\eps<\eps_0$, the last expression is equal to
\begin{equation*}
\begin{split}
&\int_ 0^T\pfrac 1N \sum_{x\in {\bb T}_N}\eta_t(x)
\Big( \int_{\bb T\backslash (-\eps,\eps)} \partial_u H(t,v)
\pfrac{1}{\eps}{\bf 1}_{(v,v+\eps)}(\pfrac{x}{N})\,dv\Big)\,dt \\
&\;=\;\int_ 0^T\pfrac 1N \sum_{x\in {\bb T}_N}
\eta_t(x)\Big(\pfrac{1}{\eps}{\bf 1}_{\bb T\backslash (-\eps,\eps)}(\pfrac{x}{N})
[H_t(\pfrac{x}{N})-H_t(\pfrac{x}{N}-\eps)]\Big)\,dt\,.
\end{split}
\end{equation*}
 Using again that $H(t,v)$ vanishes if $v\in  (-\eps, \eps)$, for all $t\in [0,T]$, the expression above is equal to
\begin{equation*}
\int_ 0^T\pfrac{1}{\eps N} \sum_{x\in {\bb T}_N}
\eta_t(x)
[H_t(\pfrac{x}{N})-H_t(\pfrac{x}{N}-\eps)]\,dt\,,
\end{equation*}
proving the claim. Applying the definition of energy and \eqref{claim}, for  $\eps>0$ sufficiently small we have
\begin{equation*}
\begin{split}
 \mc E_H(\pi^N*\ioe)\;=\; &\int_ 0^T\!\pfrac{1}{\eps N}\sum_{x\in {\bb T}_N}H(t,\pfrac{x}{N})\,[\eta_t(x)-\eta_t(x+\eps N)]
 dt\\
 & - \, 2 \int_0^T  \int_{\bb T} \big(H (t, u)\big)^2\, du \,dt\,.
\end{split}
\end{equation*}Let us introduce the notation 
\begin{equation*}
V_N(\eps, H , \eta ) \; :=\; \pfrac{1}{\eps N}\sum_{x\in {\bb T}_N}H(\pfrac{x}{N})\{\eta(x)-\eta(x+\eps N)\}
- \,  \pfrac{2}{N} \sum_{x\in\bb T_N} \big(H(\pfrac{x}{N})\big)^2 \,.
\end{equation*}
To achieve the statement of the proposition it is enough to have
\begin{equation*}
 \varlimsup_{\eps\downarrow 0}\varlimsup_{N\to \infty}\pfrac{1}{N}
\log \bb P_{\mu_N}\Big[\int_0^T V_N(\eps, H_t , \eta_t )\,dt\geq l\Big]\;\leq\; -l+K_0\,.
\end{equation*}
By the Chebyshev exponential inequality, 
\begin{equation*}
\begin{split}
 &\pfrac{1}{N} \log \bb P_{\mu_N}\Big[\int_0^T V_N(\eps, H_t , \eta_t )dt\geq l\,\Big] \\
 &\leq\: \pfrac{1}{N} \log \bb E_{\mu_N}\Big[\exp{\Big\{N\int_0^T V_N(\eps, H_t , \eta_t )dt\Big\}}\Big]-l\,.\\
\end{split}
\end{equation*}
From Jensen's inequality, the entropy's inequality and the bound \eqref{f06} of the relative entropy, the
expectation in the right hand side of inequality above is bounded  from above by 
\begin{equation*}
K_0 +\pfrac 1{N} \log \bb E_{\nu_\alpha^N}\Big[\exp{\Big\{\!N\!\!\int_0^T\!\! \!\!V_N(\eps, H_t , \eta_t )dt\Big\}}\Big]\,.
\end{equation*}
By the Feynman-Kac formula and the variational formula for the
largest eigenvalue of a symmetric operator,
\begin{equation*}
\begin{split}
&\pfrac 1N \log \bb E_{\nu_\alpha^N}\Big[\exp{\Big\{\!N\!\!\int_0^T\!\! \!\!V_N(\eps, H_t , \eta_t )dt\Big\}}\Big] \\
&\leq\; \int_0^T \, \sup_{f} \Big\{  \int V_N(\eps, H_t , \eta )
f(\eta)d \nu^N_{\alpha} (\eta) - N \mf D_N (f) \Big\}\, dt\,,
\end{split}
\end{equation*}
being the supremum above taken over all probability densities
$f$ with respect to $\nu^N_{\alpha}$. Recalling Lemma \ref{certo}, we bound the last expression by
\begin{equation*}
 \int_0^T \pfrac{2}{\eps N}\sum_{\frac{x}{N}\in(-\eps,0]}\big(H_t(\pfrac{x}{N})\big)^2\,dt\,.
\end{equation*}
Since $H$ has compact support, for $\eps>0$  small enough the expression above vanishes.
\end{proof}

\begin{corollary}\label{4.7}
 For any functions $H_1,\ldots,H_k\in \Ck$ holds
\begin{equation}\label{energia l}
 \varlimsup_{\eps\downarrow 0}\varlimsup_{N\to \infty}\pfrac{1}{N}
\log \bb P_{\mu_N}\Big[\,\max_{1\leq j\leq k}\mc E_{H_j}(\pi^N*\ioe)\geq l\,\Big]\;\leq\; -l+K_0\,.
\end{equation}
\end{corollary}
\begin{proof} Straightforward from Proposition \ref{-l} and inequalities \eqref{bound_log} and
\begin{equation}\label{bound_exp}
\exp\Big\{
\max_{1\le j\le k} a_j \Big\} \leq \sum_{1\le j\le k}
\exp\{a_j\}\,.
\end{equation}

\end{proof}

We present the following lemma to get the results of the Subsection \ref{sobolev section}  from the results of this subsection.
% \ref{energy section}.

\begin{lemma}\label{superexp_energy}
If for any function $W_\eps^N(\eta_\cdot)$ uniformly bounded by $ C$ and 
 for any $\ell\in \mathbb R$, we have
\begin{equation*}
 \varlimsup_{\eps\downarrow 0}\varlimsup_{N\to\infty}\pfrac 1N \log \bb P_{\mu_N}\Big[\,
W_\eps^N(\eta_\cdot)\geq \ell\,\Big]\;\leq\;-\ell+K_0\,,
\end{equation*} 
then 
\begin{equation*}
 \varlimsup_{\eps\downarrow 0}\varlimsup_{N\to\infty} \bb E_{\mu_N}\Big[\,
W_\eps^N(\eta_\cdot)\,\Big]\;\leq\;K_0\,.
\end{equation*}
\end{lemma}

\begin{proof}
As in Lemma \ref{superexp_replace}, the expectation  
$ \bb E_{\mu_N}\Big[\,
W_\eps^N(\eta_\cdot)\,\Big]$ is bounded from above by
\begin{equation*}
\ell\,+CT\,\bb P_{\mu_N}\Big[\,
W_\eps^N(\eta_\cdot)\geq \ell\,\Big]\,,
\end{equation*} 
for any $\ell\in \mathbb R$.
Let $\delta>0$. Take $\ell=K_0+\delta$, then 
there exists $\eps_0$ and $N_0$ such that
\begin{equation*}
 \bb P_{\mu_N}\Big[\,
W_\eps^N(\eta_\cdot)\geq \ell\,\Big]\;\leq \;e^{-N\,\delta}\,,\quad\forall N\geq N_0\mbox{ and }\forall \eps<\eps_0\,,
\end{equation*} 
we have
\begin{equation*}
 \bb E_{\mu_N}\Big[\,
W_\eps^N(\eta_\cdot)\,\Big]\;\leq \;
\ell+e^{-N\,\delta}\,,\quad\forall N\geq N_0\mbox{ and }\forall \eps<\eps_0\,.
\end{equation*} 
Thus
\begin{equation*}
 \varlimsup_{\eps\downarrow 0}\varlimsup_{N\to\infty} \bb E_{\mu_N}\Big[\,
W_\eps^N(\eta_\cdot)\,\Big]\;\leq\;\ell=K_0+\delta\,,
\end{equation*}
for all $\delta>0$,
which finishes the proof.
\end{proof}

\subsection{Replacement Lemma}\label{replace_section}

\begin{proposition}[Replacement Lemma]\label{replace} Let   $F:\bb T\to \bb R$ be a bounded function and $(\mu_N)_{N\geq 1}$ any sequence of measures. Then, with $i=1,2$, 
\begin{equation}\label{eqlimit}
\varlimsup_{\eps\downarrow 0}\varlimsup_{N\to \infty}\bb E_{\mu_N}\Big[\,\Big|\int_0^t\!\pfrac{1}{N}\!\sum_{x\neq -1} F(\pfrac{x}{N})
\Big\{\tau_x g_i(\eta_s)-\tilde{g}_i(\eta_s^{\eps N}(x), 
\eta_s^{\eps N}(x\!+\!1))\Big\}ds\,\Big|\,\Big]\!=0,
\end{equation}
\begin{equation*}
\varlimsup_{\eps\downarrow 0}\varlimsup_{N\to \infty}\bb E_{\mu_N}\Big[\,\Big|\int_0^t\pfrac{1}{N}\sum_{x\in{\bb{T}_N}} F(\pfrac{x}{N})\{
\eta_s(x)-\eta^{\eps N}_s(x)\}\,ds\,\Big|\,\Big]\,=\,0\,,
\end{equation*}
\begin{equation*}
\varlimsup_{\eps\downarrow 0}\varlimsup_{N\to \infty}\bb E_{\mu_N}\Big[\,\Big|\int_0^t 
\Big\{\tau_{-1} g_i(\eta_s)-\tilde{g}_i(\eta_s^{\eps N}(-1), 
\eta_s^{\eps N}(0))\Big\}\,ds\,\Big|\,\Big]=0\,, \; \forall t\in [0,T],\
\end{equation*}
\begin{equation*}
 \varlimsup_{\eps\downarrow 0}\varlimsup_{N\to\infty} \bb E_{\mu_N}\Big[
\Big\vert\! \int_0^T \{\eta_s(x)-\eta_s^{\eps N}(x)\}\,ds\Big\vert\Big]=0\,,\quad \textrm{ for } x=-1,0\,.
\end{equation*}
\end{proposition}
 This proposition can be obtained as a consequence of the Lemma \ref{superexp_replace} and the Corollaries \ref{super_lemma} and \ref{super one}, but  just for the sake of completeness we present here an alternative proof.
\begin{proof}%[Proof of the Proposition \ref{replace}]
We detail the proof of the first limit, being the others similar.
% The only difference is the inequality of the Lemma  \ref{s01} that we should use.
By the definition of the entropy and Jensen's inequality, the
expectation in \eqref{eqlimit} is bounded from above by 
\begin{equation*}
\frac {\bs{H}(\mu_N | \nu_\alpha^N)}{\gamma N} \,+\,
\frac 1{\gamma N} \log \bb E_{\nu_\alpha^N} \Big[ 
\exp\Big\{\! \gamma \Big| \int_0^t\!\pfrac{1}{N}\sum_{x\in{\bb{T}_N}} F(\pfrac{x}{N})\{
\eta_s(x)-\eta^{\eps N}_s(x)\}\,ds \Big|\Big\} \Big]\,, 
\end{equation*}
for all $\gamma >0$. In view of \eqref{f06},   it
is enough to show that the second term vanishes as $N\to\infty$
and then $\eps\downarrow 0$ for every $\gamma>0$. Since $e^{|x|} \le e^x +
e^{-x}$, we may remove
the absolute value inside the exponential. Thus, to complete the prove of this proposition, we need to show
that
\begin{equation*}
\begin{split}
\frac 1{\gamma N} \log \bb E_{\nu_\alpha^N} \Big[ 
\exp\Big\{\! \gamma \int_0^t\!\pfrac{1}{N}\sum_{x\in{\bb{T}_N}} F(\pfrac{x}{N})\{
\eta_s(x)-\eta^{\eps N}_s(x)\}\,ds \Big\} \Big] 
\end{split}
\end{equation*}
goes to zero when $N\to\infty$ and then $\eps\downarrow 0$.
By the Feynman-Kac formula\footnote{c.f.
 \cite[Lemma 7.2, page 336]{kl}}, for each fixed $N$ the previous expression is
bounded from above by
\begin{equation*}
t\,\sup_{f} \Big\{ \int \pfrac{1}{N}
\sum_{ x\in\bb T_N} 
F(\pfrac{x}{N})\big\{\eta(x)-\eta^{\eps N}(x)\big\}    
f (\eta) \, d \nu_\alpha^N(\eta)
- \pfrac{N}{\gamma}  \mf D_N (f) \Big\}\,,
\end{equation*}
where  the supremum  is carried  over all  density functions  $f$ with
respect to  $\nu_\alpha^N$. From inequality \eqref{for replace} of the Lemma  \ref{s01},  the previous expression is less than or equal to
\begin{equation*}
t\,\sup_{f} \Big\{ 4 A\eps\sum_{ x\in\bb T_N} \big(F(\pfrac{x}{N})\big)^2+\pfrac{1}{A} \mf D_N(f)
- \pfrac{N}{\gamma}  \mf D_N (f) \Big\}\,.
\end{equation*}
Choosing $A=\pfrac{\gamma}{N}$,  last expression becomes
\begin{equation*}
 \frac{4\gamma\eps t}{N}\sum_{ x\in\bb T_N} \big(F(\pfrac{x}{N})\big)^2\,,
\end{equation*}
which vanishes as $N\to\infty$ and then
$\eps\downarrow 0$, concluding the proof of first limit in 
statement of the lemma. 
\end{proof}

\subsection{Sobolev space}\label{sobolev section}

We prove in this section that any limit point $\bb Q^*$ of the
sequence $\bb Q_{\mu_N}^{N}$ is concentrated on trajectories
$\rho(t,u) du$ which belongs to a certain Sobolev space to
be defined ahead.
Let $\bb Q^*$ be a limit point of the sequence $\bb
Q_{\mu_N}^{N}$ along some subsequence.

\begin{proposition}
\label{s05simetrico}
The measure $\bb Q^*$ is concentrated on paths $\rho_t(u) du$ such that
$\rho\in \Sob$.
\end{proposition}
The proof is based on the Riesz Representation Theorem and follows from the  next lemma.

\begin{lemma}
\label{s03}
\begin{equation*}
 E_{\bb Q^*} \Big[ \sup_H \Big\{ \int_0^T  \int_{\bb T} \partial_u H (s, u)  \, \rho_s(u)\,du\,ds\, 
- \, 2 \int_0^T  \int_{\bb T} H (s, u)^2 \, du \,ds\Big\} \Big] \, \le \, K_0 \,,
\end{equation*}
where the supremum is carried over all functions $H$ in $\Ck$.
\end{lemma}

There are two ways to prove this lemma, the classical one  is a consequence of several lemmas, which we present after the proof of Proposition \ref{s05simetrico}, following the ideias of \cite{fl}. The other one is just a consequence of the Corollary \ref{4.7} and the Lemma \ref{superexp_energy}.

\begin{proof}[Proof of Proposition \ref{s05simetrico}]
Denote by $\ell : \Ck \to \bb R$ the linear
functional defined by
\begin{equation*}
\ell (H) \,=\, \int_0^T  \int_{\bb T}
\partial_u H (s, u) \,\rho_s(u)\,du\,ds\,.
\end{equation*}
Since $\Ck$ is dense in $L^2([0,T]\times \bb
T)$, by Lemma \ref{s03}, $\bb Q^*$-almost surely $\ell$ is a  bounded linear functional on
$\Ck$. Therefore  we can extend $\ell$ to a $\bb Q^*$-almost surely bounded functional
in $L^2([0,T]\times \bb T)$. By the Riesz Representation
Theorem, there exists a function $G$ in $L^2([0,T]\times \bb T)$
such that
\begin{equation*}
\ell (H) \,=\, - \int_0^T  \int_{\bb T}
 H (s, u) \, G(s,u)\,du\,ds\,,
\end{equation*}
concluding the proof.
\end{proof}

%%%%%este lema é usado na na estimativa da energia superexponencial

For a function $H:\bb T\to \bb R$,  $\eps
>0$ and a positive integer $N$, define $U_N(\eps,  H,
\eta)$ by
\begin{equation}\label{U_N}
\begin{split}
U_N(\eps, H, \eta ) \;=\;&
\pfrac 1{\eps N}\sum_{x\in\bb T_N} H(\pfrac{x}{N})  
\, \Big\{ \eta(x- \eps N) -\eta (x) \Big\}\\
& - \pfrac {2}{N} \sum_{x\in\bb T_N} 
\big(H(\pfrac{x}{N})\big)^2 \{ 1+\pfrac{1}{\eps}\mathbf{1}_{(-\eps,0]}(\pfrac{x}{N})\}\,.
\end{split}
\end{equation}
Recall the definition of the constant $K_0$ given in \eqref{f06}.

\begin{lemma}
\label{s04}
For $k\ge 1$, let  $H_1,\ldots,H_k:[0,T]\times \bb T\to \bb R $  be bounded functions.
Then, for every $\eps >0$,
\begin{equation*}
\varlimsup_{\delta\downarrow 0} \varlimsup_{N\to\infty}
\bb E_{\mu_N} \Big[ \max_{1\le i\le k} \Big\{
\int_0^T U_N(\eps,  H_i (s, \cdot) , \eta^{\delta N}_s ) \, 
ds \Big\} \Big] \,\le\, K_0\,.
\end{equation*}
\end{lemma}

\begin{proof}
By Proposition \ref{replace},  in order to prove
this lemma it is sufficient to show that
\begin{equation*}
\varlimsup_{N\to\infty} \bb E_{\mu_N} \Big[ \max_{1\le i\le k} \Big\{
\int_0^T U_N(\eps,  H_i (s, \cdot) , \eta_s ) \, ds \Big\}
\Big] \,\le\, K_0\,.
\end{equation*}
By the definition of the entropy and Jensen's inequality,
 the previous expectation is bounded from above by
\begin{equation*}
\frac {\bs{H}(\mu_N \vert \nu^N_{\alpha})}{ N} \, +\, \pfrac 1{N}
\log \bb E_{\nu^N_{\alpha}} \Big[ \exp\Big\{
\max_{1\le i\le k} \Big\{ N \int_0^T \,
U_N(\eps,H_i (s, \cdot) , \eta_s )\,ds \Big\} \Big\} \Big] \,.
\end{equation*}
By \eqref{f06}, the first parcel above is smaller than $K_0$.  By \eqref{bound_exp} and \eqref{bound_log},
 the limit as $N\to\infty$ of 
the previous expression is bounded from above by
\begin{equation*}
K_0\,+\,\max_{1\le i \le k} \varlimsup_{N\to\infty} \pfrac 1{N} \log
\bb E_{\nu^N_{\alpha}} \Big[ \exp
\Big\{ N  \int_0^T \,  U_N(\eps,  H_i (s, \cdot) , \eta_s )\,ds
\Big\} \Big] \,.
\end{equation*}
We claim  that the the $\limsup$  above  is non positive for each fixed $i$ (and therefore the maximum in $i=1,\ldots, k$). Fix $1\le i\le k$.
By Feynman-Kac's formula\footnote{See  \cite{kl}, Lemma 7.2, p. 336.} and the variational formula for the
largest eigenvalue of a symmetric operator, for each fixed
$N$, the second term in the previous expression is bounded from above by
\begin{equation*}
\int_0^T \, \sup_{f} \Big\{  \int U_N(\eps,  H_i (s, \cdot) , \eta_s )
f(\eta)d \nu^N_{\alpha} (\eta) - N \mf D_N (f) \Big\}\, ds\,.
\end{equation*}
In last formula the supremum is taken over all probability densities
$f$ with respect to $\nu^N_{\alpha}$.
Applying  Lemma \ref{certo} finishes the proof.
\end{proof}

%
%\begin{lemma}
%\label{s03}
%%\sup_H \Big\{ \int_0^T \! \int_{\bb T} \partial_u H (s, u)  \, \rho_s(u)\,du\,ds\, - \, 2 \int_0^T \! \int_{\bb T} H (s, u)^2 \, du \,ds\Big\}
%\begin{equation*}
% E_{\bb Q^*} \big[ \mc E(\rho)  \big] \, \le \, K_0 \,.
%\end{equation*}
%%where the supremum is carried over all functions $H$ in $\Ck$.
%\end{lemma}

%\begin{lemma}
%\label{s03}
%\begin{equation*}
% E_{\bb Q^*} \Big[ \sup_H \Big\{ \int_0^T  \int_{\bb T} \partial_u H (s, u)  \, \rho_s(u)\,du\,ds\, 
%- \, 2 \int_0^T  \int_{\bb T} H (s, u)^2 \, du \,ds\Big\} \Big] \, \le \, K_0 \,,
%\end{equation*}
%where the supremum is carried over all functions $H$ in $\Ck$.
%\end{lemma}

\begin{proof}[Proof of the Lemma \ref{s03}]
Consider a sequence $\{H_\ell, \, \ell\ge 1\}$  dense in  $\Ck$ 
with respect to the norm $\Vert H\Vert_\infty+ \Vert \partial_u H \Vert_\infty$. Recall that we suppose that $\bb Q_{\mu_N}^{N}$ converges
to $\bb Q^*$. By Lemma \ref{s04}, for every $k\ge 1$,
\begin{equation*}
 \begin{split}
  \varlimsup_{\delta\downarrow 0} E_{\bb Q^*}\Big[ \max_{1\le i\le k} 
&\Big\{ \pfrac 1{\eps} \int_0^T \int_{\bb T} \,
H_i (s,u) \, [ \rho^\delta_s (u-\eps) -
\rho^\delta_s (u ) ] \,du\,ds\\
&-\, 2 \int_0^T  
\int_{\bb T}  \, \big(H_i(s,u)\big)^2  \, \{ 1+\pfrac{1}{\eps}\textbf 1_{(-\eps,0]}(u)\}\,du\,ds
\Big\} \Big]\, \le \, K_0\,,
 \end{split}
\end{equation*}
where $\rho^\delta(u):=(\rho * \iota_\delta)(u)$.
Letting $\delta\downarrow 0$, we obtain
\begin{equation*}
 \begin{split}
E_{\bb Q^*}\Big[ \max_{1\le i\le k} 
&\Big\{\pfrac 1{\eps} \int_0^T \int_{\bb T} \,
H_i (s,u) \, [ \rho_s (u-\eps) -
\rho_s (u ) ] \,du \,ds\\
& -\, 2 \int_0^T 
\int_{\bb T}  \, \big(H_i(s,u)\big)^2  \,\{ 1+\pfrac{1}{\eps}\textbf 1_{(-\eps,0]}(u)\}\,du\,ds
\Big\} \Big]\, \le \, K_0\,.
\end{split}
\end{equation*}
Changing variables in the first integral,
\begin{equation*}
 \begin{split}
E_{\bb Q^*}\Big[ \max_{1\le i\le k} 
&\Big\{  \int_0^T\int_{\bb T} \,
\pfrac{1}{\eps}[ H_i (s,u+\eps) -
 H_i (s,u) ]\rho_s (u) \,  \,du \,ds\\
&-\, 2 \int_0^T  
\int_{\bb T}  \, \big(H_i(s,u)\big)^2  \, \{ 1+\pfrac{1}{\eps}\textbf 1_{(-\eps,0]}(u)\}\,du\,ds
\Big\} \Big]\, \le \, K_0\,.
\end{split}
\end{equation*}
Since $H_i\in \Ck$, this function vanishes in a neighborhood of zero. Making $\eps\downarrow 0$ in the last inequality, we obtain
\begin{equation*}
\begin{split}
E_{\bb Q^*}\Big[ \max_{1\le i\le k}& \Big\{
\int_0^T \!\! \int_{\bb T} \partial_u H_i (s,u)
\rho_s (u)  duds - 2 \int_0^T \!\! \int_{\bb T} \big(H_i(s,u)\big)^2\, du\,ds
\Big\} \Big] \le  K_0.
\end{split}
\end{equation*}
To conclude  it remains to apply the Monotone Convergence
Theorem and recall that $\{H_\ell, \, \ell\ge 1\}$ is a dense sequence.

\end{proof}

%
%\begin{proof}[Proof of Proposition \ref{s05simetrico}]
%Denote by $\ell : \Ck \to \bb R$ the linear
%functional defined by
%\begin{equation*}
%\ell (H) \,=\, \int_0^T  \int_{\bb T}
%\partial_u H (s, u) \,\rho_s(u)\,du\,ds\,.
%\end{equation*}
%Since $\Ck$ is dense in $L^2([0,T]\times \bb
%T)$, by Lemma \ref{s03}, $\bb Q^*$-almost surely $\ell$ is a  bounded linear functional on
%$\Ck$. Therefore  we can extend $\ell$ to a $\bb Q^*$-almost surely bounded functional
%in $L^2([0,T]\times \bb T)$. By the Riesz Representation
%Theorem, there exists a function $G$ in $L^2([0,T]\times \bb T)$
%such that
%\begin{equation*}
%\ell (H) \,=\, - \int_0^T  \int_{\bb T}
% H (s, u) \, G(s,u)\,du\,ds\,,
%\end{equation*}
%concluding the proof.
%\end{proof}

\section{Hydrodynamic limit of the WASEP with a slow bond}\label{weakly chapter}

 Fix  a function $H\in C^{1,2}([0,T]\times [0,1]\,)$. The probability $\bb P^H_{\mu_N}$
corresponds to the non-homogeneous Markov process $\eta_t=\eta_t^H$  with generator $L^H_N$ defined in \eqref{lhn} accelerated by $N^2$ and with initial measure $\mu_N$. We remark that $\mu_N$ is not invariant.
Denote by $\bb Q^H_{\mu_N}$ the probability measure on the space of trajectories 
$\DM$
induced by the empirical measure $\pi_t^N$.

\begin{proposition}\label{hid asy}
Consider a bounded density profile $\rho_0 : \bb T
\to \bb R$ and  $H\in\C$. 
The sequence of probabilities $\{\bb Q_{\mu_N}^{H};\,N\geq 1\}$ converges in distribution to the probability measure
concentrated on the absolutely continuous path $\pi_t(du)=\rho_t(u)du$, where density $\rho_t(u)$ 
is the unique weak solution of the partial differential equation \eqref{edpasy}.
\end{proposition}

Observe that the Theorem \ref{t02} is a corollary of the previous proposition.
The proof of above is divided in two parts. In Subsection \ref{tight asy}, we show that
the sequence  $\{\bb Q_{\mu_N}^{H};\,N\geq 1\}$  is tight. Subsection \ref{charact asy} is reserved to the characterization of
limit points of the sequence.

Uniqueness of limit points is assumed, since we were not able to prove uniqueness of weak solutions of the partial differential equation \eqref{edpasy}.  Additionally, uniqueness of strong solutions of
 \eqref{edpasy} is presented in Appendix \ref{unique asy}. 

\subsection{Tightness}\label{tight asy} In this subsection we present the tightness of $\{\bb Q_{\mu_N}^{H}\}$. 

\begin{proposition}\label{tight asy prop}
For   fixed $H\in\C$, the sequence of measures  $\{\bb Q_{\mu_N}^{H};\,N\geq 1\}$  is tight in the Skorohod topology of 
$\DM$.
\end{proposition}
\begin{proof}
In order to prove tightness of the sequence of measures $\{\bb Q^H_{\mu_N}:N\geq 1\}$ induced in the Skorohod space 
$\DM$ by the random elements $\{\pi_t^N: 0\leq t\leq T\}$ it is suficient to prove that the sequence of stochastic processes $\<\pi_t^N,H\>$ is tight. We begin by considering the martingale
\begin{equation}\label{MH}
M^H_{N,t}(G)\;=\;\<\pi^{N}_{t}, G_t\>- \<\pi^{N}_{0}, G_0\>-\int_{0}^{t}\<\pi^{N}_{s},\p_s G_s\>+N^{2}L^H_{N,s}\<\pi^{N}_{s},G_s\>\,ds\,,
\end{equation}
with $H,G\in\C$.
To prove tightness would be enough to handle the martingale above in the case $G \in C^2(\bb T)$. However, for future applications in the characterization of limit points, we  treat here the slightly more general  setting $G\in \C$.

First, let us show that the $L^2(\bb P^H_{\mu_N})$-norm of this martingale
vanishes as $N\to \infty$. The  quadratic variation of  $M^{H}_{N,t}(G)$ is given by
\begin{equation*}
\<M^{H}_N(G)\>_t\;=\;\int_{0}^{t} N^{2} \Big[ L^H_{N,s}\<\pi^{N}_{s},G_s\>^{2}-2\<\pi^{N}_{s},G_s\> L^H_{N,s}\<\pi^{N}_{s},G_s\>\Big] ds\,.
\end{equation*}
Applying  definition \eqref{lhn}, the quadratic variation $ \<M^{H}_N(G)\>_t$ can be rewritten as
\begin{equation*}
 \begin{split}
 &\int_{0}^{t}\!\!\!\sum_{x\in\bb T_{N}} \xi^{N}_{x}\big(\delta_{N} G_x\big)^2\!\Big[\!e^{\delta_{N} H_x}\!\eta_{s}(x)(1\!-\!\eta_{s}(x+1)) \!+\!e^{-\delta_{N} H_x}\!\eta_{s}(x\!+\!1)(1\!-\!\eta_{s}(x))\Big]ds
 \end{split}
\end{equation*}
where $\delta_{\!N} F_x$ denotes  $F_s(\pfrac{x+1}{N})-F_s(\pfrac{x}{N})$ and $\xi^N_x=N^{-1}$ if $x=-1$, and $\xi^N_x=1$ otherwise.  We observe that  $\delta_{\!N} F_x$ depends on $s$, but the dependence is dropped by convenience of notation. 
Since  $H,G\in C^{1,2}([0,T]\times[0,1]\,)$,
  the expression above for the quadratic variation can be easily bounded  by $N^{-1}$ times a constant not depending on $N$. 
By Doob inequality, we conclude that the supremum norm of the martingale $M^H_{N,t}(G)$ goes to zero in probability as $N$ goes to infinity. Hence $\{M^H_{N,t}(G)\}_{N\in\bb N}$ is tight.

Expression $\<\pi^{N}_{0}, G_0\>$ is bounded and constant in time, thus tight as well. 
It remains to analyze the tightness of the integral  term in \eqref{MH}. Using Taylor's expansion in the exponentials and performing some elementary computations, expression $N^{2}L^H_{N,s}\<\pi^{N}_{s},G_s\>$  can be written in the form
\begin{equation}\label{formula_gerador}
 \begin{split}
 & N\sum_{x\neq -1,0} \eta_{s}(x)\Big[G_s(\pfrac{x+1}{N})+G_s(\pfrac{x-1}{N})-2G_s(\pfrac{x}{N})\Big]\\
&+\,N\sum_{x\neq -1} \Big[\eta_{s}(x)(1-\eta_{s}(x\!+\!1))+\eta_{s}(x\!+\!1)(1-\eta_{s}(x))\Big]\,\big(\delta_{N} H_x\big)
\,\big(\delta_{N} G_x\big)\\
&+\Big[\,\eta_{s}(-1)(1-\eta_{s}(0))\,e^{\delta_{N} H_{-1}}-\,\eta_{s}(0)(1-\eta_{s}(-1))\,e^{-\delta_{N} H_{-1}}\Big]\,\delta_{N} G_{-1}\\
&+N\Big[\,\eta_{s}(0)\delta_{N} G_{0}\,-\,\eta_{s}(-1)\delta_{N} G_{-2}\Big]+O_{H,G}(\pfrac{1}{N})\,.
 \end{split}
 \end{equation}
 Again by smoothness of $H$ and $G$, the expression above is uniformly bounded in time. Hence, this integral term in \eqref{MH} is uniformly continuous. By Arzel\`a-Ascoli, the integral term is relatively compact, therefore tight. Since a finite sum of tight stochastic processes is tight, the  proof is finished.

\end{proof}

\subsection{Radon-Nikodym derivative}\label{section radon} In this section we deal with the Radon-Nikodym derivative between the SSEP with a slow bond and the WASEP with a slow bond. Its formula will be usefull both in the proof of the hydrodynamic limit for the WASEP with a slow bond and in the proof of the large deviations for the SSEP with a slow bond.

By  $(\radon)(t)$ we denote the Radon-Nikodym derivative of $\bb P^H_{\mu_N}$ 
with respect to $\bb P_{\mu_N}$ restricted to the $\sigma$-algebra generated by $\{\eta_s,\,0\leq s \leq t\}$. 
It is a general fact of stochastic processes
that $(\radon)(t)$ is a mean-one positive martingale.
The explicit formula of the Radon-Nikodym derivative between two Markov process on a countable space state\footnote{See Appendix of \cite{kl}} shows that $(\radon)(T)$  is equal to
\begin{equation}\label{radon1}
 \exp{\Bigg\{\!N\Big[\<\pi^N_T,H_T\>\!-\!\<\pi^N_0,H_0\>\!-\!\pfrac{1}{N}
\int_0^T\! e^{-N\<\pi^N_t,H_t\>}(\partial_t+N^2\LN) 
e^{N\<\pi^N_t,H_t\>} dt\Big]\!\Bigg\}}.
\end{equation}
We are going to write just $\radon$ for $\radon(T)$, since the  time horizon $T>0$ is fixed. Recall the notation $\delta_N H_x=H_t(\pfrac{x+1}{N})-H_t(\pfrac{x}{N})$. 
 Performing elementary calculations, we can rewrite \eqref{radon1} as 
\begin{equation}\label{r2}
\begin{split}
&  \exp \Bigg\{ N\<\pi^N_T,H_T\>-N\<\pi^N_0,H_0\>  -N\int_0^T \<\pi^N_t,\partial_t H_t\>\, dt \\
& - N^2 \int_0^T \sum_{x\in\bb T_N}\xi^{N}_{x} \eta_t(x)\big(1-\eta_t(x\!+\!1)\big)\big(e^{\delta_N H_x}-1\big)\,dt\\
& - N^2 \int_0^T \sum_{x\in\bb T_N} \xi^{N}_{x}\eta_t(x+1)\big(1-\eta_t(x)\big)\big(e^{-\delta_N H_{x}}-1\big)\, dt\,\Bigg\}\,.
\end{split}
\end{equation}
Since $H\in  \C$, by Taylor's expansion  and the inequality
 $\vert e^u-1-u-(1/2)u^2\vert\leq (1/6)\vert u\vert^3 e^{\vert u\vert}$, we observe that all the expressions 
 
 \begin{itemize}
 \item $\pfrac{1}{N}\sum_{x\neq -1,0} \eta_t(x)N^2(\delta_N H_x-\delta_N H_{x-1})
-\pfrac{1}{N}\sum_{x\in\bb T_N} \eta_t(x)\p_u^{2} H_t(\pfrac{x}{N})$,\vspace{0.3cm}

\item $ N^{2}\big(e^{\pm \delta_N H_x}\mp \delta_N H_x-1\big)- 
\pfrac{1}{2}(\p_u H_t)^2(\pfrac{x}{N}) $
\vspace{0.3cm}

\item $ N\delta_N H_{0}-\p_u H_t(\pfrac{0}{N})$ \quad and \quad  $ N\delta_N H_{-2}-\p_u H_t(\pfrac{-1}{N})$
 \end{itemize}
  are, in modulus,  of order $\frac{1}{N}$.
Putting together the facts above, we can rewrite 
\eqref{r2} as
\begin{equation}\label{radon_primeira_versao}
\begin{split}
&\exp{\Bigg\{N\Big[\<\pi^N_T,H_T\>-\<\pi^N_0,H_0\>-\int_0^T \<\pi^N_t,(\partial_t +\Delta) H_t\>\,dt}\\
& -  \int_0^T \Big\{\eta_t(0)\p_u H_t(\pfrac{0}{N})-\eta_t(-1)\p_u H_t(\pfrac{-1}{N})\Big\}\,dt +  O_{H,T}(\pfrac1N)\\
& -  \int_0^T \pfrac{1}{N}\sum_{x\neq -1}\Big[ \eta_t(x)\big(1-\eta_t(x+1)\big)+\eta_t(x+1)\big(1-\eta_t(x)\Big]\pfrac{1}{2}(\partial_u H_t)^2(\pfrac{x}{N})\, dt \\
& - \! \int_0^T\!\!\!\! \eta_t(-1)\!(1\!-\!\eta_t(0)\!)\big(e^{\delta_N H_{-1}}\!-\!1\big)dt -\!\!  \int_0^T\!\!\!\!\eta_t(0)(1\!-\!\eta_t(-1)\!)\big(e^{-\delta_N H_{-1}}\!-\!1\big)dt\Big]\! \Bigg\}\!.
\end{split}
\end{equation}
As we shall see, the expression above is enough in order to prove the hydrodynamical limit of the WASEP with a slow bond. Further estimates on the Radon-Nikodym derivative will be presented at Section \ref{upper}.

 \subsection{Sobolev space}
 In this section, we prove that any limit point $\bb Q^H_*$ of the sequence  $\bb Q_{\mu_N}^{H}$ is concentrated on trajectories $\rho_t(u)du$
belonging the Sobolev space of Definition \ref{Sobolev}.
By expression \eqref{radon_primeira_versao}, there exists a constant
$C(H,T)>0$  not depending on $N$ such that
\begin{equation}\label{RN}
 \left\Vert \frac{\textrm{\textbf{d}}\bb P^H_{\mu_N}}{\textrm{\textbf{d}}\bb P_{\mu_N}}\right\Vert_{\infty}\;\leq\; \exp\big\{C(H,T)\,N\big\}\,.
\end{equation}

\begin{proposition}\label{replaceasy} Given a bounded function $G:\bb T\to \bb R$, then, for all $t\in [0,T]$,
\begin{equation*}
\varlimsup_{\eps\to 0}\varlimsup_{N\to \infty}\bb E_{\mu_N}^H\Big[\,\Big|\int_0^t\pfrac{1}{N}\sum_{x\in{\bb{T}_N}} G(\pfrac{x}{N})\{
\eta_s(x)-\eta^{\eps N}_s(x)\}\,ds\,\Big|\,\Big]\,=\,0\,,
\end{equation*}
\begin{equation*}
\varlimsup_{\eps\to 0}\varlimsup_{N\to \infty}\bb E_{\mu_N}^H\Big[\Big|\int_0^t\!\pfrac{1}{N}\!\sum_{x\neq -1} G(\pfrac{x}{N})
\Big\{\tau_x g_i(\eta_{\color{blue}{s}})-\tilde{g}_i(\eta_{\color{blue}{s}}^{\eps N}(x), 
\eta_{\color{blue}{s}}^{\eps N}(x\!+\!1))\Big\}ds\Big|\Big]\!=\!0,
\end{equation*}
and
\begin{equation*}
\varlimsup_{\eps\to 0}\varlimsup_{N\to \infty}\bb E_{\mu_N}^H\Big[\,\Big|\int_0^t G(\pfrac{-1}{N})
\Big\{\tau_{-1} g_i(\eta_{\color{blue}{s}})-\tilde{g}_i(\eta_{\color{blue}{s}}^{\eps N}(-1), 
\eta_{\color{blue}{s}}^{\eps N}(0))\Big\}\,ds\,\Big|\,\Big]\,=\,0\,,
\end{equation*}
where $g_i$ and  $\tilde{g}_i$ with $i=1,2$ have been defined in \eqref{g1} and \eqref{g2}.
\end{proposition}
\begin{proof}
Let us prove the first limit above. Fix $\gamma >0$. From definition of the entropy and Jensen's inequality, the
expectation appearing there is bounded from above by 
\begin{equation*}
\pfrac {\bs{H} (\mu_N | \nu_\alpha^N)}{\gamma N} \,+\,
\pfrac 1{\gamma N} \log \bb E^H_{\nu_\alpha^N} \Bigg[ 
\exp\Big\{ \gamma \,\Big| \int_0^t  
\sum_{x\in{\bb{T}_N}}  G(\pfrac{x}{N})\Big\{\eta_s(x)-\eta_s^{\eps N}(x)\Big\}  
   ds \,\Big|\Big\}  \Bigg]\,. 
\end{equation*}
In view of \eqref{f06},  it
is enough to show that the second term vanishes as $N\to\infty$
and then $\eps\downarrow 0$ for every $\gamma>0$. By \eqref{RN}, the expression above is bounded by 
\begin{equation*}
\pfrac {\bs{H} (\mu_N | \nu_\alpha^N)}{\gamma N} \,+\,\pfrac{C(H,T)}{\gamma}+\pfrac 1{\gamma N} \log \bb E_{\nu_\alpha^N} \Bigg[ 
\exp\!\Big\{\! \gamma \Big|\! \int_0^t  \!\!\!
\sum_{x\in{\bb{T}_N}}\!  G(\pfrac{x}{N})\Big\{\!\eta_s(x)-\eta_s^{\eps N}(x)\!\Big\}  
   ds \Big|\Big\} \! \Bigg]\!.
\end{equation*}
Invoking proof of the Proposition \ref{replace} and noticing that $\gamma$ is arbitrary large gives the result. The remaining limits follow analogous steps.
\end{proof}

\begin{proposition}
\label{p5.3.1}
The measure $\bb Q^H_*$ is concentrated on paths $\rho_t(u) du$ such that
$\rho\in\Sob$.
\end{proposition}
\begin{proof}
As before, the proof of this result follows from Proposition \ref{s05simetrico} put together with estimate \eqref{RN}.  
\end{proof}

\subsection{Characterization of limit points}\label{charact asy}

Here, we prove that all limit points of the
sequence $\{\bb Q^{H}_{\mu_N}: N\geq{1}\}$ are concentrated on trajectories of measures absolutely
continuous with respect to the Lebesgue measure: $\pi(t,du) = \rho_t(u) du$, whose density
$\rho_t(u)$ is a weak solution of the hydrodynamic equation
 \eqref{edpasy}.

Let $\bb Q_*^H$ be a limit point of the sequence $\{\bb Q^{H}_{\mu_N}: N\geq{1}\}$
and assume, without loss of generality, that $\{\bb Q^{H}_{\mu_N}: N\geq{1}\}$
converges to $\bb Q_*^H$. The existence of $\bb Q^{H}_{*}$ is guaranteed by Proposition \ref{tight asy prop}.

In Proposition \ref{p5.3.1}, we have proved that 
$\rho_t(\cdot)$ belongs to $\Sob$.
It is well known that the Sobolev space $\mc H^1(0,1)$ has special properties: its elements are
 absolutely continuous functions  with bounded variation, c.f. \cite{e}, 
with well-defined side limits at zero. Such property is inherited by
 $\Sob$ in the sense that we can
integrate in time the side limits at the boundaries. Let $G\in\C$. We begin by claiming  that
\begin{equation}\label{QH}
\begin{split}
&\bb Q_{*}^H \Big[\pi_\cdot:\,  \<\rho_t, G_t \> \,-\,  \<\rho_0,G_0 \> \,-\, \int_0^t  \,  
\<\rho_s , (\p_s +\Delta )G_s\> \,ds\\
& -2\int_0^t\!\<\chi(\rho_s) ,\,\p_u H_s \p_u G_s \> \,ds -\,\int_0^t \varphi_s(\rho,H)\,\delta G_s(0)\,ds\\
&-\int_0^t\!\big\{\rho_s(0^+)\partial_u G_s(0^+)-\rho_s(0^-)\partial_u G_s(0^-) \big\}\,ds =0\,,\;\forall t\in[0,T] \Big]=1,
\end{split}
\end{equation}
where $\varphi_s(\rho,H) $ was defined in \eqref{varphi}.
In order to prove the equality above, its enough to show
 \begin{equation*}
\begin{split}
&\bb Q_{*}^H \Big[\pi_\cdot:\sup_{0\leq t\leq T}\,\Big\vert\,   \<\rho_t, G_t \> \,-\,  \<\rho_0,G_0 \> \,-\, \int_0^t  \,  
\<\rho_s ,(\p_s +\Delta )G_s\> \,ds \\
&-2\int_0^t\<\chi(\rho_s) ,\,\p_u H_s\p_u G_s  \> \,ds-\,\int_0^t \varphi_s(\rho,H)\,\delta G_s(0)\,ds\\
&-\,\int_0^t\big\{\rho_s(0^+)\partial_u G_s(0^+)-\rho_s(0^-)\partial_u G_s(0^-) \big\}\,ds  \,\Big\vert\,>\,\zeta\, \Big]\;=\;0\,,
\end{split}
\end{equation*}
for every $\zeta >0$.
Since the boundary integrals and the integral involving $\chi(\rho_s)$ are not defined in the whole Skorohod 
space $\DM$, we cannot use
directly Portmanteau's Theorem to obtain the claim above. To overcome this technical obstacle, fix $\eps>0$, which will be taken small later.
Recall \eqref{iota}. Adding and subtracting the convolution of $\rho_t(u)$ with $\ioe$, we bound the probability above by the sum of the probabilities
\begin{equation}\label{prob 1}
\begin{split}
&\bb Q_{*}^H \Big[\pi_\cdot:\sup_{0\leq t\leq T}\,\Big\vert\,   \<\rho_t, G_t \> -  \<\rho_0,G_0 \> \,- \int_0^t  \,  
\<\rho_s ,(\p_s +\Delta )G_s\> \,ds\\
& -2\int_0^t\<\chi(\rho_s*\ioe) ,\,\p_u H_s\p_u G_s  \> \,ds -\int_0^t \varphi_s(\rho*\ioe,H)\,\delta G_s(0)\,ds \\
&-
\int_0^t\big\{(\rho_s*\ioe)(0^+)\partial_u G_s(0^+)-(\rho_s*\ioe)(0^-)\partial_u G_s(0^-) \big\}\,ds\,\Big\vert>\zeta/4\, \Big]\,,
\\
\end{split}
\end{equation}
\begin{equation}\label{prob2}
\begin{split}
\bb Q_{*}^H \Big[\pi_\cdot:\sup_{0\leq t\leq T}&\,\Big\vert\,  
2\int_0^t\<\chi(\rho_s*\ioe) - \chi(\rho_s) ,\,\p_u H_s\p_u G_s  \> \,ds\,\Big\vert\,>\,\zeta/4\, \Big],
\end{split}
\end{equation}
\begin{equation}\label{prob3}
\begin{split}
\bb Q_{*}^H \Big[\pi_\cdot:\!\sup_{0\leq t\leq T}\Big\vert \!
&\int_0^t\!\Big\{\big[(\rho_s*\ioe)(0^+)\!-\!\rho_s(0^+)\big]\partial_u G_s(0^+)\\
&-\big[(\rho_s*\ioe)(0^-)\!-\!\rho_s(0^-)\big]\partial_u G_s(0^-) \Big\}ds   \Big\vert>\zeta/4\, \Big],
\end{split}
\end{equation}
and
\begin{equation}\label{prob4}
\begin{split}
\bb Q_{*}^H \Big[\pi_\cdot:\sup_{0\leq t\leq T}&\,\Big\vert\, 
\int_0^t\Big[ \varphi_s(\rho*\ioe,H) - \varphi_s(\rho,H)\Big]\,\delta G_s(0)\,ds\,\Big\vert\,>\,\zeta/4\, \Big]\,.
\end{split}
\end{equation}
By the Proposition \ref{p5.3.1}, the sets in \eqref{prob2}, \eqref{prob3} and \eqref{prob4} decrease to sets of null probability as $\eps\downarrow 0$. It remains to deal with \eqref{prob 1}.
By Portmanteau's Theorem, \eqref{prob 1} is bounded from above by
\begin{equation*}
\begin{split}
&\varliminf_{N\to\infty} \bb Q^{H}_{\mu_N} \Big[\pi_\cdot:\sup_{0\leq t\leq T}\,\Big\vert\, 
  \<\pi_t, G_t \> \,-\,  \<\pi_0,G_0 \> \,-\, \int_0^t  \,  
\<\pi_s ,(\p_s +\Delta )G_s\> \,ds\\
& -2\int_0^t\<\chi(\pi_s*\ioe) ,\,\p_u H_s\p_u G_s  \> \,ds-\,\int_0^t \varphi_s(\pi*\ioe,H)\,\delta G_s(0)\,ds\\
&-\int_0^t\big\{(\pi_s*\ioe)(0^+)\partial_u G_s(0^+)-(\pi_s*\ioe)(0^-)\partial_u G_s(0^-) \big\}\,ds \,\Big\vert \,>\,\zeta/4\, \Big]\,.
\end{split}
\end{equation*}
Recalling the identity \eqref{identity}, the definition of $\varphi_s(\cdot, H)$ given in \eqref{varphi}, and the definition of  $\bb Q^{H}_{\mu_N}$, 
we can rewrite the previous expression as
\begin{equation*}
\begin{split}
&\varliminf_{N\to\infty}\; \bb P^{H}_{\mu_N} \Big[\sup_{0\leq t\leq T}\,\Big\vert\,   
\<\pi^N_t, G_t \> \,-\,  \<\pi^N_0,G_0 \> \,-\, \int_0^t  \,  
\<\pi^N_s ,(\p_s +\Delta )G_s\> \,ds \\
& -\!2\!\!\int_0^t\!\!\<\chi(\pi^N_s\!*\!\ioe) ,\p_u H_s\p_u G_s \> ds -\!\!\int_0^t\!\!\!\Big\{\eta_s^{\eps N}\!(0)\partial_u G_s(0^+\!)\!-\!\eta_s^{\eps N}\!(-1)\partial_u G_s(0^-\!) \!\Big\}ds\\
&-\!\!\!\int_0^t\!\!\!\Big\{\!\eta_s^{\eps N}\!(-1)(1\!-\!\eta_s^{\eps N}\!(0)\!)e^{\delta H_s(0)}\!-\!\eta_s^{\eps N}\!(0)(1\!-\!\eta_s^{\eps N}\!(-1)\!)
e^{-\delta H_s(0)}\!
\Big\}\delta G_s(0)ds 
\Big\vert\!>\!\pfrac{\zeta}{4}\! \Big]\!.\\
\end{split}
\end{equation*}
Adding and subtracting $N^2 L_{N,s}^H\<\pi_s^N,\,G_s\>$, we bound the previous probability by the sum of
\begin{equation}\label{eq40}
\varlimsup_{N\to\infty}\! \!\bb P^{H}_{\mu_N}\! \Big[\sup_{0\leq t\leq T}\!\Big\vert  
\<\pi^N_t, G_t \> -  \<\pi^N_0,G_0 \> -\!\! \int_0^t\!\!    
\<\pi^N_s ,\p_s G_s\>+N^2 L_{N,s}^H\<\pi_s^N,G_s\> ds \Big\vert\!>\!\pfrac{\zeta}{8} \Big]
\end{equation}
and
\begin{equation}\label{second prob}
\begin{split}
& \varlimsup_{N\to\infty} \bb P^{H}_{\mu_N} \Big[\sup_{0\leq t\leq T}\Big\vert  
 \int_0^t  \,  N^2 L_{N,s}^H\<\pi_s^N,\,G_s\> ds
- \int_0^t    \<\pi^N_s ,\Delta G_s\> ds \\
&-2\int_0^t\<\chi(\pi^N_s*\ioe) ,\,\p_u H_s\p_u G_s  \>\, ds\\
&-\int_0^t\Big\{\eta_s^{\eps N}(0)\partial_u G_s(0^+)-\eta_s^{\eps N}(-1)\partial_u G_s(0^-) \Big\}\,ds \\
&-\int_0^t\Big\{\eta_s^{\eps N}(-1)(1-\eta_s^{\eps N}(0))e^{\delta H_s(0)}\\
&
 -\eta_s^{\eps N}(0)(1-\eta_s^{\eps N}(-1))e^{-\delta H_s(0)}\Big\}
\delta G_s(0)ds \Big\vert>\zeta/8 \Big].
\end{split}
\end{equation}
 The expression inside the first probability
is the martingale $M^{H}_{N,t}(G)$ defined in \eqref{MH}. Since its quadratic variations goes to zero, by Doob's inequality, the limit \eqref{eq40} is zero.
By the formula \eqref{formula_gerador} for $N^2 L_{N,s}^H\<\pi_s^N,\,G_s\>$ and a  few applications of Proposition \ref{replaceasy} we obtain that \eqref{second prob}  is null,  proving the claim \eqref{QH}.

\begin{proposition}\label{prop charat asy} 
Fix a measurable profile $\rho_0 : \bb T \to [0,1]$ and consider a sequence $\{\mu_N : N\ge 1\}$ of probability 
measures on $\{0,1\}^{\bb T_N}$ associated to $\rho_0$ in the sense of \eqref{f09}. Then any limit point of 
$\bb Q_{\mu_N}^{H}$ will be concentrated on absolutely continuous paths $\pi_t(du)=\rho(t,u)du$,
with positive density $\rho_t$ bounded by $1$, such that $\rho$ is a weak solution of \eqref{edpasy} 
with initial condition $\rho_0$.
\end{proposition}

\begin{proof}
 Let $\{G_i: i\geq{1}\}$ be a countable dense set of functions on $\C$,
 with respect to the norm $\|G\|_{\infty}+\|\partial_u G\|_\infty+\|\partial_u^2G\|_\infty$.
 Provided by \eqref{QH} and intercepting a countable number of sets of probability one, we can extend \eqref{QH} for
all functions $G\in \C$ simultaneously.
\end{proof}
 
\section{Large deviations upper bound}\label{upper}
The proof of the large deviations upper bound  is constructed by an optimization over a class of
mean-one positive martingales, which must be functions of the process, or, as in our case, 
close to functions of the process.
In the Section \ref{section radon} we have obtained a good candidate to be the mean-one positive martingale, the Radon-Nikodym derivative
of the measure $\bb P^H_{\mu_N}$  with respect to $\bb P_{\mu_N}$.
Since $\radon$ is not a function of the empirical measure, the first step  is  to show that it 
is superexponentially close to a function of the empirical measure. 
\subsection{Radon-Nikodym derivative (continuation)}
To write \eqref{radon_primeira_versao} in a simpler form, let us introduce some notation. Given
$H\in \C$,  consider the linear functional 
$\ell^{^{\textit{int}}}_H:\DM\to \bb R$ 
\begin{equation}\label{l}
\ell^{^{\textit{int}}}_H(\pi)\;=\;\<\pi_T,H_T\>-\<\pi_0,H_0\>-\int_0^T \<\pi_t,(\partial_t +\Delta) H_t\>\, dt\,.
\end{equation}
With this notation and recalling  \eqref{g1} and \eqref{g2},  we can rewrite  $\radon$ as
\begin{equation}\label{radon3}
\begin{split}
& \exp{\!\Bigg\{ N\Bigg[\ell^{^{\textit{int}}}_H(\pi^N)\!-\!\!\int_0^T\!\!\!\pfrac{1}{2N}\!\! \sum_{x\neq -1}\!\!\big\{\tau_x g_1( \eta_t)\!+\!
\tau_x g_2( \eta_t)\big\}(\partial_u H_t)^2(\pfrac{x}{N})  dt}\\
& -\!\! \int_0^T\!\!\! \big\{\eta_t(0)\p_u H_t(\pfrac{0}{N})\!-\!\eta_t(-1)\p_u H_t(\pfrac{-1}{N})\big\}dt\\
& - \int_0^T\!\! \big\{\tau_{-1} g_1( \eta_t)\big(e^{\delta_N H_{-1}}-1\big)+
\tau_{0} g_2( \eta_t)\big(e^{-\delta_N H_{-1}}-1\big)\big\}\,dt\Bigg]\! +\! NO_{H,T}(\pfrac{1}{N})\Bigg\}.\\
\end{split}
\end{equation}
We  begin by defining a set where the Radon-Nikodym derivative $\radon$ is close to a function of the empirical measure.  Consider
\begin{equation*}
\begin{split}
W^{1}_{N, \eps}(t,\eta)\;:=\;V^{F_1,F_2}_{N, \eps}(t,\eta)\,,\qquad &W^{2}_{N, \eps}(t,\eta)\;:=\; V^{G_1,G_2}_{N, \eps}(t,\eta)\,,\\
W^{3}_{N, \eps}(t,\eta)\;:=\;\hat{V}^{\p_u H,-1}_{N, \eps}(t,\eta)\,,\qquad & W^{4}_{N, \eps}(t,\eta)\;:=\; \hat{V}^{\p_u H,0}_{N, \eps}(t,\eta)\,,
\end{split}
\end{equation*}
where $V$ and $\hat{V}$ have been defined in Proposition \ref{super} and Corollary \ref{super one} considering
$ F_1(t,u)=\pfrac{1}{2}(\partial_u H_t)^2(u)$, $F_2(t,\pfrac{-1}{N})= e^{\delta_N H_{-1}}-1$,
$ G_1(t,u)=\pfrac{1}{2}(\partial_u H_t)^2(u)$ and $G_2(t,\pfrac{-1}{N})= e^{-\delta_N H_{-1}}-1$.
Define the set
\begin{equation}\label{set B}
\begin{split}
B^H_{\delta,\eps}=\Bigg\{&\eta\in\Ddiscreto\,; \,\,\,
\Big\vert \int_0^T W^{i}_{N, \eps}(t,\eta_t)\,dt\,\Big\vert\leq \delta\,, i=1,2,3,4\,\Bigg\}\,.\\
\end{split}
\end{equation}
From  Proposition \ref{super} and Corollary \ref{super one}, this set $B^H_{\delta,\eps}$ has probability superexponentially close to one, i.e.,
for each $\delta>0$,
\begin{equation}\label{B}
 \varlimsup_{\eps\downarrow 0}\varlimsup_{N\to\infty}\pfrac 1N \log 
 \bb P_{\mu_N}\Big[(B^H_{\delta,\eps})^{\complement}\Big]\;=\;-\infty\,.
\end{equation}
 In view of identity \eqref{identity} and 
expression \eqref{radon3}, restricted to the set $B^{H}_{\delta,\eps}$ the Radon-Nikodym derivative $\radon$ 
is equal to
\begin{equation}\label{RN in B}
\begin{split}
 &\exp{\Bigg\{} N\Bigg[\ell^{^{\textit{int}}}_H(\mathcal{A})+ \,O_{H,T}(\pfrac{1}{N})\,+\, O(\delta)\\
 &
-\! \!\int_0^T\!\!\!\!\pfrac{1}{2N}\!\! \!\sum_{x\neq -1}\!\!\Big\{\tilde{g}_1\Big( \mathcal{A}(\pfrac{x}{N}),\mathcal{A}(\pfrac{x+1}{N})\Big) \!+\!\tilde{g}_2 \Big(\mathcal{A}(\pfrac{x}{N}),\mathcal{A}(\pfrac{x+1}{N})\Big)\Big\}(\partial_u H_t)^2(\pfrac{x}{N})  dt\\
& - \!\int_0^T \Big[\mathcal{A}(\pfrac{0}{N})\p_u H_t(\pfrac{0}{N})-\mathcal{A}(\pfrac{-1}{N})\p_u H_t(\pfrac{-1}{N})\Big]\,dt\\
& - \!\int_0^T \tilde{g}_1\Big( \mathcal{A}(\pfrac{-1}{N}),\mathcal{A}(\pfrac{0}{N})\Big)(e^{\delta_N H_{-1}}-1)\,dt\\
&  -\! \int_0^T \tilde{g}_2\Big( \mathcal{A}(\pfrac{-1}{N}),\mathcal{A}(\pfrac{0}{N})\Big)(e^{-\delta_N H_{-1}}-1)\,dt \,\Bigg]\,\Bigg\}\,,
\end{split}
\end{equation}
where $\mathcal{A}=\pi_t^N*\iota_\eps$. At this point we have a function of the empirical measure modulo some small errors. 
Unfortunately, this is not enough to handle with limits on boundary terms. The reason is simple, 
the convolution $\pi^N*\ioe$ is a function (not a measure anymore) but not a \emph{smooth} function, therefore not necessarily possessing well-behaved side limits.
Hence, the next step is to replace  $\pi^N*\ioe$ by $(\pi^N*\iog)*\ioe$, where $\iog$ is a smooth approximation of identity to be defined next. Notice that $\iog$ shall not be misunderstood with $\ioe$ defined in \eqref{iota}.

Fix $f:\bb T\to \bb R_+$ a 
continuous function with support  contained in $[-\pfrac{1}{4},\pfrac{1}{4}]$,
$0\leq f \leq 4$, $f(0)>0$, $\int f=1$ and symmetric around zero,  in other words,  satisfying  $f(u)=f(1-u)$ for all $u\in \bb T$. Define  the continuous approximation of identity
$\iog$  by
$
\iog(u)=\pfrac{1}{\gamma} f(\pfrac{u}{\gamma})$.

At this point, we  need some approximation estimates to be presented in three next lemmas.
Recall that $\ell^{^{\textit{int}}}_H$ is the linear functional defined in \eqref{l}. 
\begin{lemma}\label{tildeiota}
$
\vert (\pi^N_t*\ioe)(v)-((\pi^N_t*\iog)*\ioe)(v)\vert\leq \pfrac{\gamma}{\eps}\,,
$
uniformly in $v\in\bb T$, $N\in \bb N$, and $t\in[0,T]$.
\end{lemma}
\begin{proof}%[Proof of Lemma \ref{tildeiota}]
 Writing the expression 
$\vert (\pi^N_t*\iota_\eps)(v)-((\pi^N_t*\iota_\gamma)*\iota_\eps)(v)\vert$ as
\begin{equation*}
\Big\vert\pfrac{1}{N}\sum_{x\in\bb T_N}\eta_t(x) \iota_\eps (\pfrac{x}{N}, v)
- \int_{\bb T}\pfrac{1}{N}
\sum_{x\in\bb T_N}\eta_t(x)
\iota_\gamma(u-\pfrac{x}{N}) \iota_\eps (u, v)\,du \Big\vert \,.
\end{equation*}
Using the rule of maximum of one particle per site, the last expression is bounded by
\begin{equation*}
\pfrac{1}{N}\sum_{x\in\bb T_N}\Big\vert \iota_\eps (\pfrac{x}{N}, v)
- \int_{\bb T} \iota_\gamma(u-\pfrac{x}{N}) \iota_\eps (u, v)\,du\Big\vert \,.
\end{equation*}
Fix $N$, $v$ and $\eps$, then $\iota_\eps (\cdot, v)$ is the indicator
 function of an open interval $(z,z+\eps)$, for $z=v$ or $z=1-\eps$.
 The summand above is possibly not zero only if $\pfrac{x}{N}$
 belongs to the open intervals $(z-\pfrac{\gamma}{4},z+\pfrac{\gamma}{4})$
 or $(z+\eps-\pfrac{\gamma}{4},z+\eps+\pfrac{\gamma}{4})$. 
The summands are bounded by $\pfrac{1}{\eps}$, and the number of non zero
 summands is of order $\gamma N$, which  concludes the proof.
\end{proof}

\begin{lemma}\label{lhtilde} 
$
 \ell^{^{\textit{int}}}_H(\pi^N) =
 \ell^{^{\textit{int}}}_H\Big((\pi^N*\iog)*\ioe\Big) + 
O_H(\eps)+O_H(\pfrac{\gamma}{\eps})\,,
$
uniformly in $N\in \bb N$.
\end{lemma}
\begin{proof}%[Proof of Lemma \ref{lhtilde}]
First we compare $\ell^{^{\textit{int}}}_H(((\pi^N*\iota_\gamma)*\iota_\eps))$ with
$\ell^{^{\textit{int}}}_H((\pi^N*\iota_\eps)) $. Using the Lemma \ref{tildeiota}, we obtain the difference
between this functions is
\begin{equation*}
\begin{split}
&\Bigg\vert\Big\<((\pi^N_T*\iota_\gamma)*\iota_\eps)-
(\pi^N_T*\iota_\eps),H_T\Big\>-
\Big\<((\pi^N_0*\iota_\gamma)*\iota_\eps)-(\pi^N_0*\iota_\eps),H_0\Big\>\\
&-\int_0^T \Big\<((\pi^N_t*\iota_\gamma)*\iota_\eps)-
(\pi^N_t*\iota_\eps),(\partial_t +\Delta) H_t\Big\> \,dt\Bigg\vert
\leq C(H) \pfrac{\gamma}{\eps}\,.
\end{split}
\end{equation*}
Then, we need only analyze the expression  below
\begin{equation*}
\begin{split}
&\Big\vert\ell^{^{\textit{int}}}_H((\pi^N*\iota_\eps))-\ell^{^{\textit{int}}}_H(\pi^N) \Big\vert=\Big\vert\Big\<(\pi^N_T*\iota_\eps)-\pi_T^N,H_T\Big\>-
\Big\<(\pi_0^N*\iota_\eps)-\pi_0^N,H_0\Big\>\\
&-\int_0^T \Big\<\,(\pi_t^N*\iota_\eps)-
\pi_t^N,(\partial_t +\Delta) H_t\Big\> \, dt\Big\vert\,.
\end{split}
\end{equation*}
We handle only the first term, because the others terms are similar. Thus,
\begin{equation*}
\begin{split}
&\Big\<(\pi_t^N*\iota_\eps),H_t\Big\>=
\int_{\bb T}\!(\pi_t^N*\iota_\eps)( v)H_t(v)dv
= \int_{\bb T}\!\pfrac{1}{N}\!\!\sum_{y\in {\bb T}_N}
\eta_t(y)\iota_\eps^a(\pfrac{y}{N},v)H_t(v)dv\\&
 = \pfrac{1}{N}\sum_{y\in {\bb T}_N}\eta_t(y)
\int_{\bb T}H_t(v)\,\iota_\eps^a(\pfrac{y}{N},v)\,dv =\<\pi^N_t,H_t\> +O_H(\eps)\,.
\end{split}
\end{equation*}
This approximation holds uniformly in time and $N$, since 
$H\in C^{1,2}([0,T]\times\overline{(0,1)}\,)$ 
and there is at most one particle per site.
Therefore,
\begin{equation*}
\vert\ell^{^{\textit{int}}}_H(\pi^N*\iota_\eps)-
\ell^{^{\textit{int}}}_H(\pi^N) \vert=O_H(\eps)\,.
\end{equation*}

\end{proof}

%\begin{figure}[h]
% \centering
%\input{fig3.pstex_t}  
%  % fig1.pstex_t: 0x0 pixel, 0dpi, 0.00x0.00 cm, bb=
%\caption{$\iota_\eps (\cdot, a^-)$ and  $\iota_\eps (\cdot, \pfrac{a_N}{N})$}
%\end{figure}
%\begin{figure}[h]
% \centering
%\input{fig4.pstex_t}  
%  % fig1.pstex_t: 0x0 pixel, 0dpi, 0.00x0.00 cm, bb=
%\caption{$\iota_\eps (\cdot, \pfrac{ a_N+1}{N})$ and $\iota_\eps (\cdot, a^+)$}
%\end{figure}
% \begin{figure}[h]
%  \centering
% \input{fig5.pstex_t}  
%   % fig1.pstex_t: 0x0 pixel, 0dpi, 0.00x0.00 cm, bb=
% \caption{}
% \end{figure}

\begin{lemma}\label{g1g2}
The function 
$ 
\Big\vert\,\tilde{g}_i\Big(\!(\pi^N_t*\ioe)(\pfrac xN),
(\pi^N_t*\ioe)(\pfrac{x+1}{N})\!\Big)-
\tilde{g}_i\Big(\! ((\pi^N_t*\iog)* \ioe) ( \pfrac xN),
((\pi^N_t*\iog)* \ioe) ( \pfrac{x+1}{N})\!\Big)\!\,\Big\vert
$
is  $O(\pfrac{\gamma}{\eps})$ for $i=1,2$.
\end{lemma}
\begin{proof}%[Proof of Lemma \ref{g1g2}]
This proof follows by  the definition of $\tilde{g}_1$ and $\tilde{g}_2$
 (see \eqref{g1} and \eqref{g2}), the triangular inequality and 
the Lemma \ref{tildeiota}.
\end{proof}

Lemmas \ref{tildeiota}, \ref{lhtilde} and \ref{g1g2} allow to replace
 $\pi^N_t$ by $(\pi^N_t*\iog)$ 
in the expression of Radon-Nikodym derivative \eqref{RN in B} modulus small errors.
Hence, restricted to the set $B^H_{\delta,\eps}$,  the Radon-Nikodym derivative 
$\radon$ becomes
\begin{equation}\label{RN in B gamma}
\begin{split}
 &\exp{\Bigg\{} N\Bigg[\ell^{^{\textit{int}}}_H(\mathcal{B}) + \,O_{H,T}(\pfrac{1}{N})\,+\,O(\delta)\,+\,O_H(\eps)\,+\,O_H(\pfrac{\gamma}{\eps})\\
 &  -\! \int_0^T\!\!\!\!\!\pfrac{1}{2N}\! \!\sum_{x\neq -1}\!\!\Big\{\tilde{g}_1\Big( \mathcal{B} ( \pfrac xN),
\mathcal{B} ( \pfrac{x+1}{N})\Big)\!+\!\tilde{g}_2 \Big(\mathcal{B} ( \pfrac xN),
\mathcal{B} ( \pfrac{x+1}{N})\Big)\Big\}(\partial_u H_t)^2(\pfrac{x}{N})  \,dt\\
& - \int_0^T \!\!\Big[\mathcal{B}(\pfrac{0}{N})\p_u H_t(\pfrac{0}{N})-
\mathcal{B}(\pfrac{-1}{N})\p_u H_t(\pfrac{-1}{N})\Big]\,dt\\
& - \int_0^T \!\tilde{g}_1\Big( \mathcal{B}(\pfrac{-1}{N}),\mathcal{B}(\pfrac{0}{N})\Big)(e^{\delta_N H_{-1}}-1)\,dt\\
&  - \int_0^T \!\tilde{g}_2\Big( \mathcal{B}(\pfrac{-1}{N}),\mathcal{B}(\pfrac{0}{N})\Big)(e^{-\delta_N H_{-1}}-1)\,dt\,\Bigg]\,\Bigg\}\,,
\end{split}
\end{equation}
where $\mathcal{B}=(\pi^N_t*\iog)* \ioe$.
Recall   $\chi(u)=u(1-u)$. The next three lemmas allow to replace the sum  involving  $\tilde{g}_i$ by an integral in $\chi$ and to make a little adjustment at the boundaries. 
%ts proofs will be omitted as well.
\begin{lemma}\label{g1g2Riemann}The difference 
 \begin{equation*}
\begin{split}
&\Big\vert\,\pfrac{1}{N}
\sum_{x\neq -1}  \tilde{g}_i\Big(\! ((\pi^N_t*\iog)* \ioe) ( \pfrac xN),
((\pi^N_t*\iog)* \ioe) ( \pfrac{x+1}{N})\!\Big) (\partial_u H_t)^2(\pfrac{x}{N})\\
&
-\int_{\bb T}  \chi\Big(\! ((\pi^N_t*\iog)* \ioe) ( v)\!\Big)(\partial_u H_t)^2(v)\,dv\,\Big\vert
\,,
\end{split}
 \end{equation*}
 can be denoted by  some function $R_N^1(H,t,\eps,\gamma)$, which goes to zero, when $N\to\infty$,
uniformly in  $t\in[0,T]$, with  $i=1,2$.
\end{lemma}
\begin{proof}%[Proof of Lemma \ref{g1g2Riemann}]  
Consider $i=1$. To simplify notation, denote
\begin{equation*}
f^N(\pfrac{x}{N}):= \tilde{g}_1\Big(\! ((\pi^N_t*\iota_\gamma)* \iota_\eps) ( \pfrac xN),
((\pi^N_t*\iota_\gamma)* \iota_\eps) ( \pfrac{x+1}{N})\!\Big)
\end{equation*}
and
\begin{equation*}
g^N(v):= \tilde{g}_1\Big(\! ((\pi^N_t*\iota_\gamma)* \iota_\eps) ( v),
((\pi^N_t*\iota_\gamma)* \iota_\eps) ( v)\!\Big)=\chi\Big(((\pi^N_t*\iota_\gamma)* \iota_\eps) ( v) \Big)\,.
\end{equation*}
From the definition of $\iota_\eps$, if $x\neq a_N$, 
\begin{equation*}
 \vert (\varrho*\iota_\eps)(\pfrac{x}{N})-(\varrho*\iota_\eps)(v)\vert\leq\frac{\Vert\varrho\Vert_{\infty}}{\eps N}\,,\qquad 
\forall v\in[\pfrac{x}{N},\pfrac{x+1}{N}]\,,
\end{equation*}
where $\varrho$ is any bounded function defined on the torus. The same inequality is still valid
with $x+1$ replacing $x$ in left side of inequality. Since 
$\Vert\pi^N_t*\iota_\gamma \Vert_\infty\leq 4$, if $x\neq a_N$,
\begin{equation*}
 \vert f^N(\pfrac{x}{N})-g^N(v)\vert = O(\pfrac{1}{\eps N}), \qquad
\forall v\in[\pfrac{x}{N},\pfrac{x+1}{N}]\,.
\end{equation*}
Then, 
\begin{equation*}
\begin{split}
&\Big\vert\pfrac{1}{N}\sum_{x\neq a_N}f^N(\pfrac{x}{N})(\partial_u H_t)^2(\pfrac{x}{N})
-\int_{\bb T}g^N(v)(\partial_u H_t)^2(v)\,dv\Big\vert\\
&\leq\Big\vert\pfrac{1}{N}
\sum_{x\neq a_N}f^N(\pfrac{x}{N})\Big[(\partial_u H_t)^2(\pfrac{x}{N})-
N\int_{\bb T}\textbf{1}_{[\frac{x}{N},\frac{x+1}{N})}(v)(\partial_u H_t)^2(v)\,dv\Big]\Big\vert\\
&+\Big\vert
\sum_{x\neq a_N}\!\!\int_{\bb T}\textbf{1}_{[\frac{x}{N},\frac{x+1}{N})}(v)\Big[f^N(\pfrac{x}{N})
-  g^N(v)\Big](\partial_u H_t)^2(v)\,dv\Big\vert\\
&+\Big\vert
\int_{\bb T}\textbf{1}_{[\frac{a_N}{N},\frac{a_N+1}{N})}(v)g^N(v)(\partial_u H_t)^2(v)\,dv\Big\vert\\
&\leq\pfrac{1}{N}\sum_{x\neq a_N}\Big\vert (\partial_u H_t)^2(\pfrac{x}{N})-
N\int_{\bb T}\textbf{1}_{[\frac{x}{N},\frac{x+1}{N})}(v)(\partial_u H_t)^2(v)\,dv\Big\vert\\
&+O(\pfrac{1}{\eps N})\int_{\bb T}\vert(\partial_u H_t)^2(v)\vert\,dv+\pfrac{1}{N}\Vert (\partial_u H_t)^2\Vert_{\infty}\,.
\end{split}
 \end{equation*}
 Since $H$ belongs to $C^{1,2}([0,T]\times\overline{(0,1)}\,)$, the first sum goes to zero, when $N\to \infty$,
 which finishes the proof.
\end{proof}

\begin{lemma}\label{a_Nderivada} Denote by $R_N^2(H,t,\eps,\gamma)$ the following expression:
 \begin{equation*}
\begin{split}
&   \Big\vert\,((\pi^N_t*\iog)* \ioe)(\pfrac{0}{N})\p_u H_t(\pfrac{0}{N})-
((\pi^N_t*\iog)* \ioe)(\pfrac{-1}{N})\p_u H_t(\pfrac{-1}{N})\\
&  -((\pi^N_t*\iog)* \ioe)(0^+)\p_u H_t(0^+)-
((\pi^N_t*\iog)* \ioe)(0^-)\p_u H_t(0^-)\,\Big\vert\,.
\end{split}
\end{equation*}
 Then  $R_N^2(H,t,\eps,\gamma)$ goes to zero, when $N$ increases to $\infty$, uniformly in  $t\in[0,T]$.
\end{lemma}
\begin{proof}%[Proof of Lemma \ref{a_Nderivada}]
 This proof follows by fact that $\iota_\eps (\cdot, \pfrac{-1}{N})$, 
$\iota_\eps (\cdot, \pfrac{0}{N})$, $\p_u H_t(\pfrac{-1}{N})$
and $\p_u H_t(\pfrac{0}{N})$ converges to  $\iota_\eps (\cdot, 0^-)$, $\iota_\eps (\cdot, 0^+)$, $\p_u H_t(0^-)$
and $\p_u H_t(0^+)$, respectively, as $N$ increases to infinity.
\end{proof}

\begin{lemma}\label{a_Ng1g2} The expression bellow
\begin{equation*}
\begin{split}
&\Big\vert\,\tilde{g}_1\Big(\! ((\pi^N_t*\iog)* \ioe) ( 0^-),
((\pi^N_t*\iog)* \ioe) ( 0^+)\!\Big)
(e^{\delta H_t(0)}-1)
\\&
-\tilde{g}_1\Big(\! ((\pi^N_t*\iog)* \ioe) ( \pfrac{-1}{N}),
((\pi^N_t*\iog)* \ioe) ( \pfrac{0}{N})\!\Big)(e^{\delta_N H_{-1}}-1)\,\Big\vert
\end{split}
\end{equation*}
is a function  $R_N^3(H,t,\eps,\gamma)$, which goes to zero, when $N$ increases to $\infty$,  uniformly in  $t\in[0,T]$.
Analogous statement for $\tilde{g}_2$.
\end{lemma}

\begin{proof}%[Proof of Lemma \ref{a_Ng1g2}]
 We only analyze the first statement, the second one is just the same argument.
 By definition of $\tilde{g}_1$, the expression in the left side of the first equality is bounded above by
\begin{equation*}
\begin{split}
 &\Big\vert ((\pi^N_t*\iota_\gamma)* \iota_\eps) ( \pfrac{-1}{N})
(e^{\nabla_N H_{-1}}-1)- ((\pi^N_t*\iota_\gamma)* \iota_\eps) ( 0^-)
(e^{H_t(0^+)-H_t(0^-)}-1)\Big\vert\\
& +\Big\vert ((\pi^N_t*\iota_\gamma)* \iota_\eps) ( \pfrac{-1}{N})
((\pi^N_t*\iota_\gamma)* \iota_\eps) ( \pfrac{0}{N})(e^{\nabla_N H_{-1}}-1)\\
&- ((\pi^N_t*\iota_\gamma)* \iota_\eps) ( 0^-)
((\pi^N_t*\iota_\gamma)* \iota_\eps) ( 0^+)(e^{H_t(0^+)-H_t(0^-)}-1)\Big\vert\,.
\end{split}
\end{equation*}
The conclusion follows by fact that $\iota_\eps (\cdot, \pfrac{-1}{N})$, 
$\iota_\eps (\cdot, \pfrac{0}{N})$ 
and $e^{\nabla_N H_{-1}}-1$ converges to  $\iota_\eps (\cdot, 0^-)$, $\iota_\eps (\cdot, 0^+)$ and
 $e^{H_t(0^+)-H_t(0^-)}-1$, respectively, as $N$ increases to infinity.

\end{proof}
Denote  $R_N(H,T,\eps,\gamma)$ the errors from the lemmas  \ref{g1g2Riemann}, \ref{a_Nderivada} and \ref{a_Ng1g2}, notice that
\begin{equation}\label{error}
\lim_{N\to\infty}R_N(H,T,\eps,\gamma)=0\,.
\end{equation}
By means of these lemmas,  we can 
rewrite the  expression \eqref{RN in B gamma} of the Radon-Nikodyn derivative $\radon$ on the set $B^H_{\delta,\eps}$ as
\begin{equation}\label{endRN}
 \begin{split}
& \exp{\Bigg\{} N\Bigg[\ell^{^{\textit{int}}}_H(\mathcal{B})
 - \int_0^T\!\!\int_{\bb T} 
 \chi(\mathcal{B}( v))\,(\partial_u H_t)^2(v)\,dv\,dt\\
 & - \int_0^T \!\!\Big[\mathcal{B}(0^+)\p_u H_t(0^+)-
\mathcal{B}(0^-)\p_u H_t(0^-)\Big]\,dt\\
& - \int_0^T \!\tilde{g}_1\Big( \mathcal{B}(0^-),
\mathcal{B}(0^+)\Big)(e^{\delta H_t(0)}-1)\,dt\\
&  - \int_0^T \!\tilde{g}_2\Big( \mathcal{B}(0^-),
\mathcal{B}(0^+)\Big)(e^{-\delta H_t(0)}-1)\,dt\\
& + \,R_N(H,T,\eps,\gamma)\,+\,O(\delta)\,+\,O_H(\eps)\,+\,O_H(\pfrac{\gamma}{\eps})\,\Bigg]\,\Bigg\}\,,
\end{split}
\end{equation}
where $\mathcal{B}=(\pi^N_t*\iog)* \ioe$ as before.

Now, we observe  that the  functional $\ell_H$ defined in \eqref{ell} and the functional $\ell^{^{\textit{int}}}_H$ given in Definition \eqref{l} are related by
\begin{equation*}
\begin{split}
  \ell_H(\pi)&=\ell^{^{\textit{int}}}_H(\pi)
 -\int_0^T\{\rho_t(0^+)\partial_u H_t(0^+)-\rho_t( 0^-)\partial_u H_t(0^-)\}\,dt\\
 &+\int_0^T(\rho_t(0^+)-\rho_t(0^-))(H_t(0^+)-H_t(0^-))\,dt\,.
\end{split}
\end{equation*}
Moreover, because of its smoothness, $(\pi^N*\iog)* \ioe$ has finite energy, see Definition \ref{energy}.
Recalling  Definition \ref{J def} of the functional $J_H$,
 and expression \eqref{endRN}, we conclude that $\radon$ restricted to $B_{\delta,\eps}^H$ is 
\begin{equation}\label{radon J}
\exp{\Big\{ N\Big[J_H\Big((\pi^N*\iog)*\ioe\Big)
 + R_N(H,T,\eps,\gamma)+O(\delta)+O_H(\eps)+O_H(\pfrac{\gamma}{\eps})\Big]\Big\}}\,.
\end{equation}
Let us proceed to the next step. It is not difficult to see that the set $\{\pi \in \DM\;;\;\mc E(\pi)<\infty\}$  is not closed in the concerning topology (the Skorohod topology on $\DM$). 
 This is an obstacle in order to apply the Minimax Lemma, see \cite[Lemma 3.3, page 364]{kl}, which is an important device in the proof of the large deviations upper bound. To invoke the Minimax Lemma, the functional $J_H$ should be lower semi-continuous\footnote{About signs and conventions: in \cite[Lemma 3.3, page 364]{kl} the statement is about an upper continuous functional, but the functional $J_\beta$ appearing there  corresponds  to minus our functional $J_H$ here.}, what is not true precisely because the set $\{\pi\in \DM\;;\;\mc E(\pi)<\infty\}$ is not closed.

To overcome this obstacle, we begin by introducing the next sets.
\begin{definition}\label{As}
Let $A_{k,l}$, $A_{k,l}^{\eps}$, and $A_{k,l}^{\eps,\gamma}$  be the subsets of trajectories  given by
\begin{equation*}
\begin{split}
 &A_{k,l}\;=\;\{\pi\in \DM\,;\,\,\,\max_{1\leq j\leq k}\mc E_{H_j}(\pi)\leq l\}\,,
\\&
A_{k,l}^{\eps}\;=\;\left\{\pi\in \DM\,;\,\,\,\pi* \ioe
\in A_{k,l}\right\}\,,
\\&
A_{k,l}^{\eps,\gamma}\;=\;\left\{\pi\in \DM\,;\,\,\,(\pi*\iog)* \ioe
\in A_{k,l}\right\}\,.
\end{split}
\end{equation*}
\end{definition}
\begin{proposition}\label{closed}
For fixed $\eps,\gamma,k,l$, the set $A_{k,l}^{\eps,\gamma}$ is closed.
\end{proposition}
\begin{proof}
It is sufficient to show that  the function $\psi:  \DM\to \overline{\bb R}$
 given by $\psi(\pi)=\mc E_{H_j}((\pi^N*{\iog})* \ioe) $ is  continuous.
 Let $\{\pi^n_t;t\in[0,T]\}_n$ converging to
 $\{\pi_t;t\in[0,T]\}$ on $\DM$. Therefore, $\pi^n_t\stackrel{\omega^*}{\to}\pi_t$,
 almost surely in time.
 For such $t$, 
$ \pi_t*  \iog = \lim_{n\to\infty}\pi^n_t*  \iog $, since   $\iog$ is a continuous function.
By the Dominated Convergence Theorem,
\begin{equation}\label{limite}
 ((\pi_t*\iog)* \ioe)(v)
\;=\;\int_{\bb T}\lim_{n\to\infty}(\pi^n_t* \iog) (u)\, \ioe(u,v)\,du
\;=\; \lim_{n\to\infty}((\pi^n_t*\iog)* \ioe)(v)\,.
\end{equation}
Again by the Dominated Convergence Theorem,
\begin{equation*}
\begin{split}
\Big \<\!\!\Big\< \partial_u H_j,\;(\pi_t*\iog)* \ioe \Big\>\!\!\Big\>
&\; =\;\int_0^T\int_{\bb T}\partial_u H_j(t,v)((\pi_t*\iog)* \ioe)(v)\,dv\,dt\\
&\;=\;\lim_{n\to\infty}\int_0^T\!\!\int_{\bb T}\partial_u H_j(t,v)
((\pi^n_t*\iog)* \ioe)(v)\,dv\,dt\\
&\;=\;
\lim_{n\to\infty}\Big \<\!\!\Big\< \partial_u H_j,\;
(\pi^n*\iog)* \ioe\,\Big\>\!\!\Big\>\,.
\end{split}
\end{equation*}
\end{proof}

\begin{proposition}\label{A}
For fixed $k$, and $l$,
\begin{equation*}
\varlimsup_{\eps\downarrow 0}\varlimsup_{\gamma\downarrow 0}
\varlimsup_{N\to\infty}\pfrac{1}{N}\log \bb P_{\mu_N}\Big[\pi^N\in (A_{k,l}^{\eps,\gamma})^{\complement}\Big]
\;\leq\; -l+K_0T \,.
\end{equation*}
\end{proposition}
\begin{proof}
For all $r>0$,
\begin{equation*}
\begin{split}
 \bb P_{\mu_N}\Big[\max_{1\leq j\leq k}\mc E_{H_j}\big(&(\pi^N*\iog)* \ioe\,\big)
\geq l\,\Big]
\leq    \bb P_{\mu_N}\Big[\max_{1\leq j\leq k}\mc E_{H_j}\big(\pi^N* \ioe\,\big)\;\geq\; l-r\,\Big]\\
&+\bb P_{\mu_N}\Big[\max_{1\leq j\leq k}\mc E_{H_j}\Big((\pi^N*\iog)*\ioe -\pi^N*\ioe\Big)\;\geq\; r\,\Big]\,.
\end{split}
\end{equation*}
By Lemma \ref{tildeiota},  we have that
\begin{equation*}
 \max_{1\leq j\leq k}\mc E_{H_j}\Big((\pi^N*\iog)*\ioe -\pi^N*\ioe\Big){\color{blue}{\leq}}
\max_{1\leq j\leq k}\Big\<\!\!\Big\<\partial_u H_j,(\pi^N*\iog)*\ioe -\pi^N*\ioe\Big\>\!\!\Big\>
\leq \pfrac{C\gamma}{\eps},
\end{equation*}
where $C=C(\{H\}_{1\leq j\leq k})$. Therefore,
\begin{equation*}
\bb P_{\mu_N}\Big[\max_{1\leq j\leq k}
\mc E_{H_j}\big((\pi^N*\iog -\pi^N)*\ioe\,\big)\geq r\,\Big]\;\leq\;
 \bb P_{\mu_N}\Big[\pfrac{C\gamma}{\eps}\;\geq\; r\,\Big]\,,
\end{equation*}
which is zero for $\gamma$  small enough.
Hence,
\begin{equation*}
\begin{split}
&\varlimsup_{\gamma\downarrow 0}\varlimsup_{N\to\infty}\pfrac 1N \log \bb P_{\mu_N}\Big[\max_{1\leq j\leq k}
\mc E_{H_j}\big((\pi^N*\iog)* \ioe\,\big)
\geq l\,\Big]\\
&\leq\; \varlimsup_{N\to\infty}\pfrac 1N \log \bb P_{\mu_N}\Big[\max_{1\leq j\leq k}
\mc E_{H_j}\big(\pi^N* \ioe\,\big)\geq l-r\,\Big]\,.
\end{split}
\end{equation*}
By Corollary \ref{4.7}, we get
\begin{equation*}
\varlimsup_{\eps\downarrow 0}\varlimsup_{\gamma\downarrow 0}\varlimsup_{N\to\infty}\pfrac 1N \log
 \bb P_{\mu_N}\Big[\max_{1\leq j\leq k}\mc E_{H_j}\big((\pi^N*\iog)* \ioe\,\big)
\;\geq\; l\,\Big]\;\leq\; -l+K_0 T+r\,.
\end{equation*}
Since $r$ is arbitrary, the proof is finished.
\end{proof}

In  \eqref{radon J} appears the term $(\pi^N*\iog)*\ioe$ and we would like to take  $\gamma \downarrow 0$ and $\eps\downarrow 0$. To avoid  technical problems that would come into scene from the fact $\pi^N_t$ does not have density with respect to the Lebesgue measure,  we define below another family of sets.

Fix a sequence $\{ F_i\}_{i\geq 1}$ of smooth non negative functions dense in the subset of non-negative
functions $C(\bb T)$ with respect to the uniform topology.
For $i\geq 1$ and $j\geq 1$, we define the set
\begin{equation}\label{D_i}
 D_i^j\; =\; \Big\{\pi\in \DM\;;\;0\leq \<\pi_t,F_i\>\leq \int_{\bb T} F_i(u)\,du
+\pfrac{1}{j}\Vert F'_i\Vert_\infty ,\,
0\leq t\leq T\Big\}\,,
\end{equation}
and for $m\geq 1$ and $j\geq 1$, let
$ E_m^j=\bigcap_{i=1}^{m} D_i^j$.
\begin{proposition}\label{E}
 It holds:
\begin{itemize}
 \item[\textbf{(i)}] Given $i\geq 1$ and $j\geq 1$, the set $D_i^j$ is a closed subset of $\DM$;
 \item[\textbf{(ii)}] $\DMO=\cap_{j\geq 1}\cap_{m\geq 1} E_m^j$;
 \item[\textbf{(iii)}] Given $m\geq 1$ and $j\geq 1$, 
$\varlimsup_{N\to\infty}\pfrac{1}{N}\log \bb P_{\mu_N}[\pi^N\in (E_m^j)^{\complement}]=-\infty $\,.
\end{itemize}
\end{proposition}
\begin{proof}
\textbf{(i)} Since $F_i$ continuous,  the function $\pi\mapsto\sup_{0\leq t\leq T}
\<\pi_t,F_i\>$ is continuous.

\textbf{(ii)} The inclusion $\DMO\subset\cap_{j\geq 1}\cap_{m\geq 1} E_m^j$ is trivial.
The inclusion on the other hand follows by approximating indicators functions of open intervals by a suitable sequence
in $\{F_i\}_{i\geq 1}$ and in $j$. 

\textbf{(iii)} The probability $\bb P_{\mu_N}[\pi^N\in (E_m^j)^{\complement}]$ is 
\begin{equation*}
\bb P_{\mu_N} \Big[\bigcup_{i=1}^{m}\Big\{\pfrac{1}{N}\sum_{x\in \bb T_N} F_i(\pfrac{x}{N})\eta_t(x) > 
\int_{\bb T} F_i(u)\,du+\pfrac{1}{j}\Vert  F'_i\Vert_\infty,\,\textrm{for some}\, t\in[0,T]\Big\}\Big].
\end{equation*}
From the elementary inequality
\begin{equation*}
\Big\vert \pfrac{1}{N}\sum_{x\in \bb T_N} F_i(\pfrac{x}{N})-\int_{\bb T} F_i(u)\,du\Big\vert\;\leq\; 
\sum_{x\in \bb T_N} \int_{[\frac{x}{N},\frac{x+1}{N})}\vert F_i(\pfrac{x}{N})-F_i(u)\vert\,du 
\leq\pfrac{\Vert F_i'\Vert_\infty}{N}\,,
\end{equation*}
and by the fact that there is at most one particle per site, we conclude that  $\bb P_{\mu_N}\big[\pi^N\in (E_m^j)^{\complement}\big]$ vanishes for $N$ sufficiently large,  concluding the proof.
\end{proof}
Keeping in mind that  $\mc E((\pi*\iog)*\ioe)<\infty$, for all $\pi\in \DM$, define
\begin{equation}\label{Jmuito}
J_{H,\gamma,\eps,\zeta}^{k,l,m,j}(\pi)\;=\;\left\{\begin{array}{cl}
\hat{J}_H\Big((\pi*\iog)*\ioe\Big), &  \mbox{if}\,\,\,\,\pi\in A_{k,l}^{\zeta,\gamma}\cap E_m^j \,,\\ 
+\infty, &\mbox{otherwise.}
\end{array}
\right. 
\end{equation}
Finally, $\radon$ restricted to the set 
$\{\pi^N \in  A_{k,l}^{\zeta,\gamma}\cap E_m^j\}\cap B_{\delta,\eps}^H$ is
\begin{equation}\label{radon J indices}
\exp{\Big\{ N\Big[J_{H,\gamma,\eps,\zeta}^{k,l,m,j}(\pi^N)
 + R_N(H,T,\eps,\gamma)+O(\delta)+O_H(\eps)+O_H(\pfrac{\gamma}{\eps})\Big]\Big\}}\,.
\end{equation}
This is the appropriate form  for the Radon-Nikodym derivative to be used in the next section.

\subsection{Upper bound for compact sets}

We start by studying the upper bound {\color{blue}{for}} open sets.
Let $\mc O\subseteq\DM$ be an open set and fix a function  $H\in \C$. Then
\begin{equation*}
\begin{split}
 &\varlimsup_{N\to\infty}\pfrac{1}{N}\log \bb Q_{\mu_N}[\mc O] 
= \varlimsup_{N\to\infty}\pfrac{1}{N}\log \bb P_{\mu_N}[\pi^N\in\mc O]\\
  & \leq  \max\Big\{\!\varlimsup_{N\to\infty}\!\pfrac{1}{N}
\log \bb P_{\mu_N}[\{\pi^N\in\mc O\cap A_{k,l}^{\zeta,\gamma}\!\cap\! E_m^j\}\!
\cap\! B^H_{\delta,\eps}],  R_k^l(\zeta,\gamma),R_m^j, R_H^{\delta}(\eps)\!\Big\}
\end{split}
 \end{equation*}
 where we have  denoted
 \begin{equation*}
\begin{split}
& R_k^l(\zeta,\gamma)=\varlimsup_{N\to\infty}\pfrac{1}{N}
\log \bb P_{\mu_N}[\{\pi^N\in (A_{k,l}^{\zeta,\gamma})^{\complement}\}]\,,
\\&
 R_m^j=\varlimsup_{N\to\infty}\pfrac{1}{N}
\log \bb P_{\mu_N}[\{\pi^N\in (E_m^j)^{\complement}\}]\,,
\\&
 R_H^{\delta}(\eps) = \varlimsup_{N\to\infty}\,\pfrac{1}{N}
\log \bb P_{\mu_N}[(B^H_{\delta,\eps})^{\complement}]\,.
\end{split}
\end{equation*}
\medskip

\noindent
By Propositions \ref{A} and  \ref{E} and the limit \eqref{B}, the expressions above satisfy
\begin{equation*}
\varlimsup_{\zeta\downarrow 0}\,\varlimsup_{\gamma\downarrow 0}\, R_k^l(\zeta,\gamma)\;\leq\; -l+K_0T\,,
\quad
R_m^j=-\infty\,,
\quad\textrm{and}\quad
 \varlimsup_{\eps\downarrow 0}\, R_H^{\delta}(\eps)\;=\;-\infty\,.
\end{equation*}
Transforming the measure by the Radon-Nikodym derivative and recalling its expression \eqref{radon J indices},
\begin{equation*}
\begin{split}
& \bb P_{\mu_N}\Big[\{\pi^N\in\mc O\cap A_{k,l}^{\zeta,\gamma}\cap E_m^j\}
\cap B^H_{\delta,\eps} \Big]\!
\!=\! \bb E^H_{\mu_N}\Big[\Big(\pfrac{\textrm{d}\bb P^H_{\mu_N}}{\textrm{d}\bb P_{\mu_N}}\Big)^{-1}\!
\textbf 1_{\{\pi^N\in \mc O\cap A_{k,l}^{\zeta,\gamma}\cap E_m^j\}\cap B^H_{\delta,\eps}}\Big]\\
&\!=\! \bb E^H_{\mu_N}\Bigg[\exp{\Big\{\! N\Big[ \!-\!J_{H,\gamma,\eps,\zeta}^{k,l,m,j}(\pi^N)  
\!+\!R_N(H,T,\eps,\gamma)\!+\!O(\delta) \!+\! O_H(\eps)\!+\!O_H(\pfrac{\gamma}{\eps})\Big]\!\Big\}}
\textbf 1_{{\textbf{D}}}\Bigg],
\end{split}
\end{equation*}
being $\textbf{D}:=\{\pi^N\in \mc O\cap A_{k,l}^{\zeta,\gamma}\cap E_m^j\}\cap B^H_{\delta,\eps}$.
Therefore,
\begin{equation*}
\begin{split}
 &\pfrac{1}{N}\log \bb P_{\mu_N}[\{\pi^N\in \mc O\cap A_{k,l}^{\zeta,\gamma}\cap E_m^j\}\cap B^H_{\delta,\eps}]\\
&\leq  \sup_{\pi\in\mc O}\{-J_{H,\gamma,\eps,\zeta}^{k,l,m,j}(\pi)\} +R_N(H,T,\eps,\gamma)+O(\delta) + O_H(\eps)+O_H(\pfrac{\gamma}{\eps})\,.
\end{split}
\end{equation*}
By \eqref{error}, for all $\gamma,\eps,\zeta, \delta>0$, for all $k,l,m,j\in \bb N$ and  $H\in \C$, we have
\begin{equation*}
\begin{split}
&\varlimsup_{N\to\infty}\pfrac{1}{N}\log \bb Q_{\mu_N}[\mc O]\\
  &\leq \max\Big\{ \!\sup_{\pi\in\mc O}
\{-J_{H,\gamma,\eps,\zeta}^{k,l,m,j}(\pi)\} \!+\!O(\delta) \!+\! O_H(\eps)\!+\!O_H(\pfrac{\gamma}{\eps}),
R_k^l(\zeta,\gamma),\,R_m^j,\,R_H^{\delta}(\eps)\!\Big\}\\
& =\max\Big\{ \sup_{\pi\in\mc O}
\{-J_{H,\gamma,\eps,\zeta}^{k,l,m,j}(\pi)\}  +O(\delta) + O_H(\eps)+O_H(\pfrac{\gamma}{\eps}),
R_k^l(\zeta,\gamma),\,R_H^{\delta}(\eps)\Big\}\,.
\end{split}
\end{equation*}
Since we do not have any restrictions on the parameters,
 we can optimize over  $\gamma,\eps,\zeta,\delta,k,l,m,j,H$, which yields
\begin{equation}\label{infsup}
\begin{split}
&\varlimsup_{N\to\infty}\pfrac{1}{N}\log \bb Q_{\mu_N}[\mc O]\\
&\leq\!\!\!\! \inf_{\at{\gamma,\eps,\zeta,\delta,}{k,l,m,j,H}}\max\Big\{
\sup_{\pi\in\mc O}\{-J_{H,\gamma,\eps,\zeta}^{k,l,m,j}(\pi)\}  
\!+\!O(\delta) \!+\! O_H(\eps)\!+\!O_H(\pfrac{\gamma}{\eps}),
R_k^l(\zeta,\gamma),R_H^{\delta}(\eps)\!\Big\}
\\
&=\!\!\!\!\inf_{\at{\gamma,\eps,\zeta,\delta,}{k,l,m,j,H}}\!\sup_{\pi\in\mc O}
\max\Big\{\!\!-\!\!J_{H,\gamma,\eps,\zeta}^{k,l,m,j}(\pi)  
\!+\!O(\delta) \!+\! O_H(\eps)\!+\!O_H(\pfrac{\gamma}{\eps}),
R_k^l(\zeta,\gamma),\,R_H^{\delta}(\eps)\Big\}\,.
\end{split}
\end{equation}

\begin{proposition}\label{cont}
 For fixed $\gamma,\eps,\zeta,\delta,k,l,m,j,H$, the functional 
\begin{equation*}
\max\Big\{-J_{H,\gamma,\eps,\zeta}^{k,l,m,j}(\pi)  
+O(\delta) + O_H(\eps)+O_H(\pfrac{\gamma}{\eps}),\,R_k^l(\zeta,\gamma),\,R_H^{\delta}(\eps)\Big\}
\end{equation*}
is upper semi-continuous in $\DM$.
\end{proposition}
\begin{proof}
In the maximum above, the only term that depends on $\pi$    is
 $J_{H,\gamma,\eps,\zeta}^{k,l,m,j}(\pi)$.
By the
 Propositions  \ref{closed} and \ref{E}, it is enough to prove the
 continuity of $\hat{J}((\pi*\iog)* \ioe)$ in $\DM$.
 
 Let $\pi^n\to\pi$ in the topology of $\DM$. 
In particular, $\pi^n_t$ converges weakly$^{*}$ to $\pi_t$ in $\mc M$, for almost all $t \in [0,T]$.
  According to \eqref{limite} and iterated applications
 of Dominated Convergence Theorem we can assure the continuity of 
$\hat{J}((\pi*\iog)* \ioe)$.
\end{proof}

Provided by the proposition above, we may apply the  Minimax Lemma \cite[Lemma A2.3.3]{kl},
interchanging supremum with infimum in \eqref{infsup}, and passing to compacts sets.
Then, for all $\mc K\subset \DM$ compact, 
\begin{equation}\label{compact}
\begin{split}
 &\varlimsup_{N\to\infty}\frac{1}{N}\log \bb Q_{\mu_N}[\mc K]\\
 &\leq\sup_{\pi\in\mc K}
\inf_{\at{\gamma,\eps,\zeta,\delta,}{k,l,m,j,H}}\!
\max\Big\{\!-\!J_{H,\gamma,\eps,\zeta}^{k,l,m,j}(\pi)  
\!+\!O(\delta)\! +\! O_H(\eps)\!+\!O_H(\pfrac{\gamma}{\eps}),R_k^l(\zeta,\gamma),R_H^{\delta}(\eps)\Big\}.
\end{split}
\end{equation}
The next result connects $J_H(\pi)$ and  $J_{H,\gamma,\eps,\zeta}^{k,l,m,j}(\pi)$.
\begin{proposition}\label{label} For all $\pi\in \DM$,
\begin{equation*}
\varlimsup_{\eps\downarrow 0}
\varlimsup_{l\to\infty}\varlimsup_{k\to\infty}
\varlimsup_{\zeta\downarrow 0}
\varlimsup_{\gamma\downarrow 0}
\varlimsup_{j\to\infty}\varlimsup_{m\to\infty}
J_{H,\gamma,\eps,\zeta}^{k,l,m,j}(\pi)
\;\geq\; J_H(\pi)\,.
\end{equation*}
\end{proposition}

\begin{proof} Recal  \eqref{Jmuito} and fix $\pi\in \DM$.  We claim that
\begin{equation}\label{57}
\varlimsup_{j\to\infty}\varlimsup_{m\to\infty}
J_{H,\gamma,\eps,\zeta}^{k,l,m,j}(\pi)\;=\;
\begin{cases}
\;\hat{J}_H((\pi*\iog)*\ioe), &  
\mbox{if}\,\,\,\,\pi\in A_{k,l}^{\zeta,\gamma}\cap \DMO \\ 
\;+\infty, &\mbox{otherwise}
\end{cases}
\;.
\end{equation}
The equality above derives from the fact that if $\pi\notin\DMO$, there exist  $m$ and $j$ such that $\pi\notin E_m^j$. To check this,
apply the definition of a absolute continuity with respect
to the Lebesgue measure. This proves \eqref{57}.

Let us step to the limit in $\gamma$. We claim that
\begin{equation}\label{eps}
\varlimsup_{\gamma\downarrow 0}
\begin{cases}\!
\hat{J}_H((\pi\!*\!\iog)*\ioe), \!\!\!\!&  
\mbox{if }\pi\in A_{k,l}^{\zeta,\gamma}\cap \DMO  \\ 
\!+\infty, &\mbox{otherwise}
\end{cases}
\!\geq\!
\begin{cases}
\!\hat{J}_H(\pi\!*\!\ioe), \!\!\!\!&  
\mbox{if }\pi\in A_{k,l+1}^{\zeta}\cap \DMO. \\ 
\!+\infty, &\mbox{otherwise}
\end{cases} 
\end{equation}
If $\pi\notin A_{k,l}^{\zeta,\gamma}\cap \DMO $ for all $\gamma$, the inequality \eqref{eps} is obvious.
From Definition \ref{As}, if $\pi\in A_{k,l}^{\zeta,\gamma}\cap \DMO $, it is immediate that
\begin{equation*}
 \max_{1\leq j \leq k}\mc E_{H_j}\big(\,\pi*\ioz\,\big)
\;\leq\; l+\max_{1\leq j \leq k}
\Big\<\!\!\Big\< \,\partial_u H_j,\;\,\pi*\ioz-(\pi*\iog)*\ioz\,\Big\>\!\!\Big\>\,.
\end{equation*}
For fixed $\zeta$ and $k$, we can find $\gamma$  small enough
 in such a way 
 \begin{equation*}
 \max_{1\leq j \leq k}\mc E_{H_j}\big(\,\pi*\ioz\,\big)
\;\leq\; l+1\,,
\end{equation*}
 implying $\pi \in A_{k,l+1}^{\zeta}\cap \DMO$. Besides, for fixed $\eps>0$, the double convolution $(\pi*\iog)*\ioe$ converges
uniformly to $\pi*\ioe$, leading to 
\begin{equation*}
\lim_{\gamma \downarrow 0} \hat{J}_H((\pi*\iog)*\ioe)\;=\; \hat{J}_H(\pi*\ioe)
\end{equation*}
and hence proves \eqref{eps}.  
The ensuing step is to take the limit in $\zeta\downarrow 0$.
We claim that
\begin{equation}\label{59}
\varlimsup_{\zeta\downarrow 0}
\begin{cases}
 \hat{J}_H(\pi*\ioe), &  
\!\!\!\mbox{if }\pi\in A_{k,l+1}^{\zeta}\cap \DMO\\ 
 +\infty, &\mbox{otherwise}
\end{cases}
 \geq \! 
\begin{cases}
\hat{J}_H(\pi*\ioe), & \!\!\! 
\mbox{if }\pi\in A_{k,l+2}\cap \DMO\\ 
+\infty, &\mbox{otherwise}
\end{cases}.
\end{equation}
In fact, if $\pi\in A_{k,l+1}^{\zeta}\cap \DMO$, then
\begin{equation*}
\begin{split}
\max_{1\leq j\leq k} \mc E_{H_j}(\pi)&\;=\;\max_{1\leq j\leq k} \mc E_{H_j}(\,\pi*\ioz\,)+
\max_{1\leq j \leq k}
\Big\<\!\!\Big\<\, \partial_u H_j,\;\,\pi-\pi*\ioz\,\Big\>\!\!\Big\>
\\
&
\;\leq\; l+1+\max_{1\leq j \leq k}\int_0^T\!\!\int_{\bb T} \partial_u H_j(t,u)\big(\rho_t(u)-(\pi_t*\ioz)(u)\big)\,du\,dt\,.
\end{split}
\end{equation*}
By the Lebesgue Differentiation Theorem, it is possible to choose small $\zeta$ such that the integral term in the right hand side of above is smaller than $1$. This proves \eqref{59}. 
 Taking  the limit in $k\to\infty$ in the right hand side of  \eqref{59}, we  obtain
\begin{equation}\label{60}
\varlimsup_{k\to \infty}
\begin{cases}
\;\hat{J}_H(\pi*\ioe), &  
\mbox{if}\,\,\pi\in A_{k,l+2}\cap \DMO\\ 
\;+\infty, &\mbox{otherwise}
\end{cases}
 = 
\begin{cases}
\;\hat{J}_H(\pi*\ioe), &  
\mbox{if}\,\,\mc E(\pi)\leq l+2\\ 
\;+\infty, &\mbox{otherwise}
\end{cases}
,
\end{equation}
because  $\{\pi;\,\mc E(\pi)\leq l+2\}\subset\DMO$.
Next,  taking the limit in $l\to\infty$ in the right hand side of \eqref{60}, we get
\begin{equation*}
\varlimsup_{l\to \infty}
\begin{cases}
\;\hat{J}_H(\pi*\ioe), &  
\mbox{if}\,\,\,\,\mc E(\pi)\leq l+2\\ 
\;+\infty, &\mbox{otherwise}
\end{cases}
\quad \geq \quad 
\begin{cases}
\;\hat{J}_H(\pi*\ioe), &  
\mbox{if}\,\,\,\,\mc E(\pi)<\infty\\ 
\;+\infty, &\mbox{otherwise}
\end{cases}
\;.
\end{equation*}
Finally, taking the limit when $\eps\downarrow 0$ in the right hand side of above, it yields
\begin{equation*}
 \varlimsup_{\eps\downarrow 0} 
\begin{cases}
\;\hat{J}_H(\pi*\ioe), &  
\mbox{if}\,\,\,\,\mc E(\pi)<\infty\\\ 
\;+\infty, &\mbox{otherwise}
\end{cases}\qquad =\qquad J_H(\pi)\;,
\end{equation*}
where we have used that, for $\pi\in\{\pi;\,\mc E(\pi)<\infty\}$ it holds that  $\pi_t(du)=\rho_t(u)du$,
 where $\rho$ has 
well-defined left and right side limits around zero.
\end{proof}

\begin{proposition}[Upper bound for compact sets] 
For every $\mc K$ compact subset of $\DM$,
 \begin{equation*}
 \varlimsup_{N\to\infty}\pfrac{1}{N}\log \bb Q_{\mu_N}[\mc K]
\;\leq\; -\inf_{\pi\in\mc K} I(\pi)\,.
 \end{equation*}
\end{proposition}
\begin{proof}
 Proposition \ref{label} can be restated in the form
\begin{equation*}
\varliminf_{\eps\downarrow 0}
\varliminf_{l\to\infty}
\varliminf_{k\to\infty}
\varliminf_{\zeta\downarrow 0}
\varliminf_{\gamma\downarrow 0}
\varliminf_{j\to\infty}\varliminf_{m\to\infty}
- J_{H,\gamma,\eps,\zeta}^{k,l,m,j}(\pi)
\;\leq\; -J_H(\pi)\,,
\end{equation*}
for all $\pi\in \DM$.
Plugging  this into \eqref{compact} leads  
to
\begin{equation*}
 \varlimsup_{N\to\infty}\pfrac{1}{N}\log \bb Q_{\mu_N}[\mc K]
\leq \sup_{\pi\in\mc K} \inf_{H}\,\{ -J_H(\pi)\}=  -\inf_{\pi\in\mc K} \sup_{H}\, J_H(\pi)= -\inf_{\pi\in\mc K} I(\pi)\,.
 \end{equation*}
\end{proof}

\subsection{Upper bound for closed sets}\label{upper closed}
\begin{proposition}[Upper bound for closed sets] 
For every $\mc C$ closed subset of $\DM$,
 \begin{equation*}
 \varlimsup_{N\to\infty}\pfrac{1}{N}\log \bb Q_{\mu_N}[\mc C]
\;\leq\; -\inf_{\pi\in\mc C} I(\pi)\,.
 \end{equation*}
\end{proposition}
 By exponential tightness, we mean that there exists compact sets  $K_n\subset \DM$ such that
\begin{equation*}
 \varlimsup_{N\to\infty}\pfrac{1}{N}\log  \bb Q_{\mu_N}[K_n^\complement]\;\leq\; -n\;,\quad\quad\forall\; n\in \bb N\;.
\end{equation*}
It is well known that the upper bound for closed sets is an immediate consequence of upper bound for compact sets plus exponential
tightness. We include the proposition below for sake of completeness.
\begin{proposition}
If the sequence  of probabilities $\{ \bb Q_{\mu_N}\}_{N\geq 1}$ is exponentially tight and  the inequality
\begin{equation*}
 \varlimsup_{N\to\infty}\pfrac{1}{N}\log  \bb Q_{\mu_N}[K]\;\leq\; - \inf_{\pi\in K} I(\pi)
\end{equation*}
holds for any compact set $K$, then $\{\bb Q_{\mu_N}\}_{N\geq 1}$ satisfies 
\begin{equation*}
 \varlimsup_{N\to\infty}\pfrac{1}{N}\log  \bb Q_{\mu_N}[C]\;\leq\; - \inf_{\pi\in C} I(\pi)\;,
\end{equation*}
for any closed set $C$.
\end{proposition}
\begin{proof}
Let $C$ be a closed set. Since $ \bb Q_{\mu_N}C]\;\leq\;  \bb Q_{\mu_N}[C\cap K_n]+ \bb Q_{\mu_N}[K_n^\complement]$ and $C\cap K_n$
is compact,
\begin{equation*}
\begin{split}
 \varlimsup_{N\to\infty}\pfrac{1}{N}\log  \bb Q_{\mu_N}[C]
& \leq \max\left\{\varlimsup_{N\to\infty}\pfrac{1}{N}\log  \bb Q_{\mu_N}[C\cap K_n],
\varlimsup_{N\to\infty}\pfrac{1}{N}\log  \bb Q_{\mu_N}[K_n^\complement]\!\right\}\\
 &\leq \max\Big\{- \inf_{\pi\in C\cap K_n} I(\pi),-n\Big\}\leq
\max\Big\{\!- \inf_{\pi\in C} I(\pi),-n\Big\}.
\end{split}
\end{equation*}
Since $n$ is arbitrary, the inequality follows.
\end{proof}

The rest of this section is concerned about exponential tightness, which we
claim it is a consequence of next lemma:

\begin{lemma}\label{l1}
 For  $\eps>0$, $\delta>0$ and $H\in C^2(\bb T)$, 
 denote
$$\mc C_{H,\delta,\eps}:=\Big\{\pi\in \DM\,;\, \sup_{s\leq t\leq s+\delta}|\<\pi_t,H\>-\<\pi_s,H\>|\,\leq\,\eps,\,\;\forall s\in[0,T]\Big\}\,.
$$Then, for every $\eps>0$ and every  function $H\in C^2(\bb T)$, the following limit holds:
\begin{equation*}
 \lim_{\delta\downarrow 0}\varlimsup_{N\to\infty}\pfrac{1}{N}\log \bb Q_{\mu_N}
\big[\pi\notin \mc C_{H,\delta,\eps}\big]\,=\,-\infty\,.
\end{equation*}
\end{lemma}
Indeed, suppose the statement above. Let $\{H_\ell\}_{\ell\in \bb N}\subset C^2(\bb T)$ be a dense set of functions in $C(\bb T)$ for the uniform topology.
For each $\delta>0$ and  $\ell,m\in\bb N$,  denote by $C_{\ell,\delta,\frac{1}{m}}$ the set $\mc C_{H_\ell,\delta,\eps}$ with $\eps=\frac{1}{m}$.
Assuming   Lemma \ref{l1}, in particular we have that
\begin{equation}\label{eqq01}
 \lim_{\delta\downarrow 0}\varlimsup_{N\to\infty}\pfrac{1}{N}\log \bb Q_{\mu_N}\Big[\pi\not\in C_{\ell,\delta,\frac{1}{m}} \Big]\,=\,-\infty\,, \quad\forall\, \ell,m\geq 1.
\end{equation}
Fix  positive integers $\ell,m$. In view of \eqref{eqq01}, for any $n\in\bb N$ we can find  $\delta_0=\delta_0(\ell,m,n)>0$ such that
 \begin{equation*}
\varlimsup_{N\to\infty}\pfrac{1}{N}\log \bb Q_{\mu_N}\Big[\pi\not\in C_{\ell,\delta,\frac{1}{m}} \Big]\,\leq\,-n \, m\,\ell\,, \quad \forall \delta\in (0,\delta_0]\,.
\end{equation*}
 Hence, for each $ \delta\in (0,\delta_0]$ there exists
  $N_\delta=N_\delta(\delta,\ell,m,n)\in \bb N$ such that
\begin{equation*}
 \bb Q_{\mu_N}\Big[\pi\not\in C_{\ell,\delta,\frac{1}{m}} \Big]\,\leq\, e^{-N\,n\, m\,\ell}\,, \quad\forall\, N\geq N_\delta.
\end{equation*}
At this point, some efforts are necessary in order to remove the restriction above on $N$ (by suitably re-defining $\delta$). This is the content of the claim:

{\emph{Claim:}} For all positive integers $\ell,m,n$, there exists $\tilde{\delta}=\tilde{\delta}(\ell,m,n)>0$ such that
$$  \bb Q_{\mu_N}\Big[\pi\notin C_{\ell,\tilde{\delta},\frac{1}{m}} \Big]\,\leq\, e^{-N\,n\,m\,\ell}\,,\quad\forall N\in\bb N.$$

To prove this claim, we start by observing that,
 if  $0<\delta_1<\delta_2$, then
$C_{\ell,\delta_2,\frac{1}{m}} \subseteq C_{\ell,\delta_1,\frac{1}{m}}$. Hence
\begin{equation}\label{delllta}
\big[\pi\not\in C_{\ell,\delta_1,\frac{1}{m}} \big]\;\subseteq\; \big[\pi\not\in C_{\ell,\delta_2,\frac{1}{m}} \big]\,,~~~\mbox{ for }0<\delta_1<\delta_2\,.
\end{equation}
Now, denoting
  $N_0=N_{\delta_0}(\ell,m,n)$ (which depends only on $\ell, m,n$, because $\delta_0$ is a function of $\ell,m,n$), we have that
\begin{equation}\label{delllta2}
 \bb Q_{\mu_N}\Big[\pi\not\in C_{\ell,\delta,\frac{1}{m}} \Big]\,\leq
 \bb Q_{\mu_N}\Big[\pi\not\in C_{\ell,\delta_0,\frac{1}{m}} \Big]\,\leq\, e^{-N\,n\, m\,\ell}\,,
\end{equation}
$\forall\, \delta\in (0,\delta_0]$ and $\forall\, N\geq N_0$.

Observe that, for fixed $\ell,m\in \bb N$, we have that 
$C_{\ell,\delta,\frac{1}{m}}\nearrow\DM$ as $\delta\searrow 0$, which is true because the set $\DM$ is composed of c\`adl\`ag trajectories.
 Since the sets $\big[\pi\not\in C_{\ell,\delta,\frac{1}{m}}\big]$ decrease  to the empty set as $\delta\searrow 0$, then for each fixed $N$, the 
probability  
   $\bb Q_{\mu_N}[\pi\not\in C_{\ell,\delta,\frac{1}{m}} ]$ decreases to zero as $\delta\searrow 0$.
 Therefore, for each fixed $N\in\bb N$, we can choose 
\begin{equation}\label{deltatilde}
\tilde{\delta}_N=\tilde{\delta}_N(\ell,m,n)\;\leq\; \delta_0(\ell,m,n)=\delta_0
\end{equation} 
such that
\begin{equation}\label{delllta1}
 \bb Q_{\mu_N}\Big[\pi\not\in C_{\ell,\tilde{\delta}_N,\frac{1}{m}} \Big]\,\leq\, e^{-N\,n\,m\,\ell}\,.
\end{equation}
Denote now
$$\tilde{\delta}:=\min_{N< N_0}\tilde{\delta}_N\leq \delta_0\,.$$ 
Let $N\in \bb N$. 
If $N< N_0$, then, by $\tilde{\delta}\leq \tilde{\delta}_N$, \eqref{delllta} and \eqref{delllta1}, we have that
$$  \bb Q_{\mu_N}\Big[\pi\not\in C_{\ell,\tilde{\delta},\frac{1}{m}} \Big]\,\leq\,\bb Q_{\mu_N}\Big[\pi\not\in C_{\ell,\tilde{\delta}_N,\frac{1}{m}} \Big]\,\leq\, e^{-N\,n\,m\,\ell}\,.$$
Furthermore, if $N\geq N_0$,  the construction $\tilde{\delta}\leq \delta_0$ (see \eqref{delllta2} and \eqref{deltatilde}) assures that
$$  \bb Q_{\mu_N}\Big[\pi\not\in C_{\ell,\tilde{\delta},\frac{1}{m}} \Big]\,\leq\, e^{-N\,n\,m\,\ell}\,,$$ finishing the proof of the claim.

\bigskip

Keeping in mind that our goal is to prove that the sequence $\bb Q_{\mu_N}$ is exponentially tight, we define
\begin{equation*}
 K_n\;=\;\bigcap_{\ell\geq 1,m\geq 1} C_{\ell,\tilde{\delta},\frac{1}{m}}\,,
\end{equation*}
which is a intersection of closed sets, hence closed as well. 
 In order to   prove that  $K_n$ is a compact set for each $n\geq 1$,  we
use a version of Arzel\`a-Ascoli theorem, which states  that a set of functions $K_n\subset \DM$ is relatively compact if it is uniformly bounded, and
\begin{equation}\label{limlim}
\lim_{\delta\rightarrow 0} \sup_{{\pi\in K_n}}\inf_{\{t_i\}} \max_i \sup_{s,t\in[t_{i-1},t_i)} d(\pi_s,\pi_t)=0\,,
\end{equation}
where the infimum is taken over all partitions $0=t_0<t_1<\ldots<t_r$ with $t_i-t_{i-1}>\delta$ and $d$ is the metric on $\mc M$. 
 We start by observing that $ K_n$ is uniformly bounded, because 
$ K_n\subset \DM$ (c.f. the Definition of $\mc M$ in \eqref{M}). 
The limit \eqref{limlim} is a consequence of  
\begin{equation}\label{limlim1}\lim_{\delta\rightarrow 0}\sup_{\pi\in K_n}\sup_{|t-s| \leq \delta}d(\pi_s,\pi_t)=0\,.
\end{equation} 
 To prove the limit above we start by observing that if
 $\pi\in K_n$ and $|t-s| \leq \tilde\delta$ (we can suppose without loss of generality that $s\leq t$, thus $s\leq t\leq s+\tilde\delta$), then
$${\displaystyle|\<\pi_t,H_\ell\>-\<\pi_s,H_\ell\>|\,\leq\,\pfrac 1m\,, }\quad \forall \; \ell,m\in \bb N\,.$$
Now, recalling  that the metric $d$ on $\mc M$ is
$$d(\pi_s,\pi_t)\;=\;\sum_{\ell\in \bb N}\frac{1}{2^\ell}\frac{|\<\pi_t,H_\ell\>-\<\pi_s,H_\ell\>|}{1+|\<\pi_t,H_\ell\>-\<\pi_s,H_\ell\>|}\;\leq\; \sum_{\ell\in \bb N}\pfrac{1}{2^\ell}\,|\<\pi_t,H_\ell\>-\<\pi_s,H_\ell\>|\,,$$
we 
have, for $\pi\in K_n$ and $|t-s| \leq \tilde\delta$, that
$d(\pi_s,\pi_t)\leq \pfrac 1m\,,$ for all $m\in\bb N$, leading to  \begin{equation}\label{kkkk}
\sup_{\pi\in K_n}\sup_{|t-s| \leq \tilde{\delta}}d(\pi_s,\pi_t)\leq \pfrac 1m\,, \quad \mbox{for all}\quad m\in\bb N\,.
\end{equation} 
 Since $\delta\mapsto \sup_{|t-s| \leq \delta}d(\pi_s,\pi_t)$ is decreasing on $\delta$ (for $\pi$ fixed), 
 the inequality \eqref{kkkk} holds for $\delta\leq \tilde{\delta}$ in place of $ \tilde{\delta}$.   Therefore, the limit \eqref{limlim1} follows.

  Since $K_n$ is  relatively compact and closed, we conclude  that $K_n$ is a compact set.
Furthermore, by construction of the set $K_n$ and the last claim, we have that
\begin{equation*}
 \bb Q_{\mu_N}\Big[\pi\not\in K_n\Big] \;\leq\;  \sum_{\at{\ell\geq 1}{m\geq 1}}e^{-N\,n\,m\,\ell} \;\leq\; C\,e^{-N\,n}\,,
\end{equation*}
where $C$ is a constant not depending in the parameters. In particular, 
\begin{equation*}
 \varlimsup_{N\to\infty}\pfrac{1}{N}\log \bb Q_{\mu_N}\Big[\pi\not\in K_n \Big]\;\leq\;-n\,,
\end{equation*}
which is the exponential tightness. Therefore, it only remains to prove the Lemma~\ref{l1}.

\begin{proof}[Proof of Lemma \ref{l1}]

 Fix $\eps>0$ and $H\in C^2(\bb T)$. Recalling the definition of the set $\mc C_{H,\delta,\eps}$, we can rewrite the set $\big[\pi\notin\mc C_{H,\delta,\eps}\big]$ as
\begin{equation*}
\Big\{\pi\in\DM;\;\,\sup_{s\leq t
\leq s+\delta}
|\<\pi_t,H\>-\<\pi_{s},H\>|\;>\;\eps,\, \mbox{ for some } s\in[0,T]\,\Big\}\;.
\end{equation*} 
Consider the partition of the interval $[0,T]$ with mesh size equal to $\delta$.  For each $s\in[0,T]$  there exists $k\in\{0,\dots,\lfloor T\delta^{-1}\rfloor\}$ such that $k\delta\leq s<(k+1)\delta$. 
Thus,
\begin{equation*}
\begin{split}
&\sup_{s\leq t
\leq s+\delta}
|\<\pi_t,H\>-\<\pi_{s},H\>|\\
\leq&\sup_{s\leq t\leq (k+1)\delta}
|\<\pi_t,H\>-\<\pi_{s},H\>|\;\;+\;\sup_{(k+1)\delta\leq t\leq s+\delta}
|\<\pi_t,H\>-\<\pi_{s},H\>|\,.
 \end{split}\end{equation*}
 Adding and subtracting $\<\pi_{k\delta},H\>$ in  both    terms above, and adding and subtracting  $\<\pi_{(k+1)\delta},H\>$ in the second term, we bound the last expression   by
\begin{equation*}\begin{split}&
\;4\sup_{k\delta\leq t
\leq (k+1)\delta}\;\;
\big|\<\pi_t,H\>-\<\pi_{k\delta},H\>\big|\\&\;+ \;\sup_{(k+1)\delta\leq t
\leq (k+2)\delta}
\big|\<\pi_t,H\>-\<\pi_{(k+1)\delta},H\>\big|\,.
  \end{split}\end{equation*}
Then,
\begin{equation*}
\Big\{\pi;\;\,\sup_{s\leq t
\leq s+\delta}
|\<\pi_t,H\>-\<\pi_{s},H\>|\;>\;\eps,\, \mbox{ for some } s\in[0,T]\,\Big\}
\subseteq
\bigcup_{k=0}^{\lfloor T\delta^{-1}\rfloor}A_{k,\delta,\eps}^{H,N}\,,
\end{equation*} 
 where
$$A_{k,\delta,\eps}^{H,N}\;=\;\Big\{\,\sup_{k\delta\leq t
\leq (k+1)\delta}
\big|\<\pi_t,H\>-\<\pi_{k\delta},H\>\big|\;>\;\eps/5\,\Big\}\,.$$
Thus, for all $\delta>0$,
\begin{equation}\label{ineq}
\varlimsup_{N\to\infty}\pfrac{1}{N}\log \bb Q_{\mu_N}
\Big[\,\pi\notin \mc C_{\ell,\delta,\frac{1}{m}}\Big]\;\leq\;
\varlimsup_{N\to\infty}\pfrac{1}{N}\log \sum_{k=0}^{\lfloor T\delta^{-1}\rfloor}\bb Q_{\mu_N}
\big[A_{k,\delta,\eps}^{H,N}\big] \,.
\end{equation}
Since
\begin{equation}\label{bound_log}
\varlimsup_N N^{-1} \log \{a_N + b_N\}\;= \;
\max\Big\{\varlimsup_N N^{-1} \log a_N,\varlimsup_N N^{-1} \log b_N\Big\}\,,
\end{equation} the limit in the right-hand side of \eqref{ineq} is bounded from above by
\begin{equation*}
\max_{k\in\{0,\dots,\lfloor T\delta^{-1}\rfloor\}}
\varlimsup_{N\to\infty}\pfrac{1}{N}\log \bb Q_{\mu_N}
\big[A_{k,\delta,\eps}^{H,N}\big] \,.
\end{equation*}
Then, in order to prove the Lemma \ref{l1},  it is enough to show that 
\begin{equation}\label{l2}
 \lim_{\delta\downarrow 0}\max_{k\in\{0,\dots,\lfloor T\delta^{-1}\rfloor\}}\varlimsup_{N\to\infty}\pfrac{1}{N}\log \bb Q_{\mu_N}\Big[A_{k,\delta,\eps}^{H,N}\,\Big]\;=\;-\infty\,.
\end{equation}
 We begin by observing that $A_{k,\delta,\eps}^{H,N}=B_{k,\delta,\eps}^{H,N}\cup B_{k,\delta,\eps}^{-H,N}$, where
\begin{equation*}
B_{k,\delta,\eps}^{H,N}\;=\;\Big\{\,\sup_{k\delta\leq t \leq (k+1)\delta}
\<\pi_t,H\>-\<\pi_{k\delta},H\>\;>\;\eps/10\,\Big\} \,.
\end{equation*}
Hence, recalling \eqref{bound_log}, to obtain \eqref{l2} it is  sufficient to assure that 
\begin{equation}\label{l22}
 \lim_{\delta\downarrow 0}\max_{k\in\{0,\dots,\lfloor T\delta^{-1}\rfloor\}}\varlimsup_{N\to\infty}\pfrac{1}{N}\log \bb Q_{\mu_N}\Big[B_{k,\delta,\eps}^{H,N}\,\Big]\;=\;-\infty\,,
\end{equation}
for any $H\in C^2(\bb T)$ and $\eps>0$. To obtain the claim above we  analyze the limit \break $\varlimsup_{N\to\infty}\frac{1}{N}\log \bb Q_{\mu_N}\Big[B_{k,\delta,\eps}^{H,N}\,\Big]$ for fixed   $k$, $\delta$, $\eps$ and $H$. Let $a>0$. Denote  
\begin{equation*}
\begin{split}
 & M^{a,H}_t=\exp{\Big\{aN\Big[\<\pi^N_t,H\>-\<\pi^N_0,H\>}
- \int_0^tU_N^a(H,s,\eta_s)\, ds\Big]\Big\}\,,
\end{split}
\end{equation*}
where
\begin{equation*}
\begin{split}
 U_N^a(H,s,\eta_s)=
\pfrac{1}{aN} \,e^{-aN\<\pi^N_s,H\>}(\partial_s\!+\!N^2\LN) 
e^{aN\<\pi^N_s,H\>}\,.
\end{split}
\end{equation*}
Note that $ M^{a,H}_t$ is a positive mean one martingale with respect to the natural filtration.
And,
$\{M^{a,H}_t/M^{a,H}_{k\delta}\}_{t\geq k\delta}$ is also a positive  mean one martingale.
Adding and subtracting the integral part, we get 
$$\bb Q_{\mu_N}[B_{k,\delta,\eps}^{H,N}\,]\;\leq\; \bb Q_{\mu_N}[C_{k,\delta,\eps}^{a,H,N}\,]+\bb Q_{\mu_N}[D_{k,\delta,\eps}^{a,H,N}\,]\,,$$ where
\begin{equation*}\label{l3}
C_{k,\delta,\eps}^{a,H,N}\;=\;\Big\{\;\sup_{k\delta\leq t\leq (k+1)\delta}\;
\frac{1}{aN}\log \Bigg(\frac{M^{a,H}_t}{M^{a,H}_{k\delta}}\; \Bigg)\;>\;\,\eps/20\; \Big\}
\end{equation*}
and
\begin{equation*}
D_{k,\delta,\eps}^{a,H,N}\;=\;\Big\{\;\sup_{k\delta\leq t\leq (k+1)\delta}\;
\int_{k\delta}^t U_N^a(H,s,\eta_s)\,ds\;>\;\,\eps/20\;\Big\}\;.
\end{equation*}
By the considerations above and again  \eqref{bound_log}, we have that
\begin{equation}\label{ABC_2}
\begin{split}
&\varlimsup_{N\to\infty}\frac{1}{N}\log \bb Q_{\mu_N}\Big[B_{k,\delta,\eps}^{H,N}\,\Big]\\
&\leq\; \max\Big\{\varlimsup_{N\to\infty}\frac{1}{N}\log \bb Q_{\mu_N}\Big[C_{k,\delta,\eps}^{a,H,N}\,\Big],\,\varlimsup_{N\to\infty}\frac{1}{N}\log \bb Q_{\mu_N}\Big[D_{k,\delta,\eps}^{a,H,N}\,\Big]\Big\}\;,
\end{split}
\end{equation}
 for all $\delta>0$ and $k\in\{0,1,\dots,\lfloor T\delta^{-1}\rfloor\}$. 
Since $H\in C^2(\bb T)$, by  Taylor expansion 
it is easy\footnote{One can do similar computations of those in the Subsection \ref{section radon}.} to verify that
$
 |\int_{k\delta}^t U_N^a(H,s,\eta_s)\,ds |
$
is bounded by $C(a,H)\delta$, for all $t\in[k\delta,(k+1)\delta]$. 
Thus, if we take $\delta\in(0,\tilde{C})$ with $\tilde{C}:=\eps/(20\,C(a,H))$, then $\bb Q_{\mu_N}[D_{k,\delta,\eps}^{a,H,N}]=0$ for all $k\in\{0,1,\dots,\lfloor T\delta^{-1}\rfloor\}$, and therefore the inequality \eqref{ABC_2} becomes
\begin{equation*}\label{AB}
\varlimsup_{N\to\infty}\frac{1}{N}\log \bb Q_{\mu_N}\Big[B_{k,\delta,\eps}^{H,N}\,\Big]\leq \varlimsup_{N\to\infty}\frac{1}{N}\log \bb Q_{\mu_N}\Big[C_{k,\delta,\eps}^{a,H,N}\,\Big]\;,
\end{equation*}
provided $\delta<\tilde{C}$.
We handle now the set $C_{k,\delta,\eps}^{a,H,N}$ in the following way:
\begin{equation*}\label{ABC}
\begin{split}
 \bb Q_{\mu_N}\Big[C_{k,\delta,\eps}^{a,H,N}\,\Big] 
& \; = \; \bb Q_{\mu_N}\Big[\,\sup_{k\delta\leq t \leq (k+1)\delta}
\frac{M^{a,H}_t}{M^{a,H}_{k\delta}}>e^{aN\eps /20}\,\Big] \;\leq\; \frac{1}{e^{aN\eps /20}}\;,
\end{split}
\end{equation*}
where in  last inequality we have used Doob's inequality since $\{M^{a,H}_t/M^{a,H}_{k\delta}\}_{t\geq k\delta}$  is a mean one positive martingale.  Thus 
\begin{equation*}
\begin{split}
 &\varlimsup_{N\to\infty}\pfrac{1}{N}\log
\bb Q_{\mu_N}\Big[B_{k,\delta,\eps}^{H,N}\,\Big] \leq -a\eps /20\;,
\end{split}
\end{equation*}
for all  $a>0$, $\eps>0$, $H\in C^2(\bb T)$, $\delta\in (0,\tilde{C})$ and  $k=0,1,\dots,\lfloor T\delta^{-1}\rfloor$. Fix $a>0$.  Taking the  limit $\delta\searrow 0$  in the inequality above gives us 
\begin{equation*}
\varlimsup_{\delta\downarrow 0}\max_{k\in\{0,\dots,\lfloor T\delta^{-1}\rfloor\}}\varlimsup_{N\to\infty}\pfrac{1}{N}\log \bb Q_{\mu_N}\Big[B_{k,\delta,\eps}^{H,N}\,\Big]\,\leq\,\frac{-a\eps}{20}\,.
\end{equation*}
Now, taking the limit when $a\to+\infty$  leads  to \eqref{l22}, finishing the proof.

\end{proof}

\section{Large deviations lower bound for smooth profiles}\label{lower}

Next, we  obtain a non-variational  formulation of the rate functional $I$ for profiles $\rho$ whose are solutions of the hydrodynamical equation  for some perturbation $H\in \C$.

\begin{proposition}\label{charact_H_suave}
  Given  $H\in \C$, let  $\rho^H$ be the unique  weak solution of \eqref{edpasy}. Then,
\begin{equation}\label{charact_H_suave_eq}
\begin{split}
& I(\rho^H)\,:=\,\sup_G\hat{J}_G(\rho^H)\, =\,\hat{J}_H(\rho^H)\\
 &=  \,  \int_0^T\big\<\chi(\rho_t^H),(\partial_u H_t)^2\big\>\,dt 
  +\int_0^T\rho_t^H(0^-) \big(1-\rho_t^H( 0^+)\big)\,\bs \Gamma\big(\delta H_t(0)\big)\,dt \\
   & +\int_0^T\rho_t^H(0^+) \big(1-\rho_t^H(0^-)\big)\,\bs \Gamma\big(-\delta H_t(0)\big)\,dt\,,
\end{split}
\end{equation}
where $\bs \Gamma(y)=1-e^{y}+y\,e^y$,  $\;\forall\, y\in\bb R$.
\end{proposition}
Although of quite simple proof, this result has a deep interpretation. The functional $-\hat{J}_G(\rho)$ has the meaning of being  \emph{the price} to observe the profile $\rho$ when we perturb the system by $G$. The equality $\sup_G\hat{J}_G(\rho^H)=\hat{J}_H(\rho^H)$ says  that the minimum cost to observe the profile $\rho$ is reached  by picking up the perturbation $G=H$, where $H$ is such that $\rho=\rho^H$, i.e., such that $\rho$ is a solution of \eqref{edpasy}. 
\begin{proof}
Replacing the integral equation \eqref{eqint12} in the definition of $\hat{J}$ given in \eqref{J hat}, we get
\begin{equation*}
\begin{split}
\hat{J}_G(\rho^H) \; =\; & \int_0^T\big\<\chi(\rho_t^H),(\partial_u H_t)^2\big\>\,dt
-\int_0^T\big\<\chi(\rho_t^H),(\partial_u H_t-\partial_u G_t)^2\big\>\,dt \\ &+\int_0^T\rho_t^H(0^-) \big(1-\rho_t^H( 0^+)\big)\,\bar{\Gamma}\big(\delta G_t(0),\,\delta H_t(0)\big)\,dt\\
 &+\int_0^T\rho_t^H(0^+) \big(1-\rho_t^H(0^-)\big)\,\bar{\Gamma}\big(-\delta G_t(0),\,-\delta H_t(0)\big)\,dt\,,
\end{split}
\end{equation*}
where $\bar{\Gamma}(x,y)= 1-e^{x}+x\,e^{y}$,  $\;\forall x,y\in\bb R$. 
Let  $y\in \bb R$ fixed. The function $x\mapsto \bar{\Gamma}(x,y)$ assumes its maximum  at $x=y$. Therefore, $I(\rho^H)=\sup_{G}\hat{J}_G(\rho^H)=\hat{J}_H(\rho^H)$.
Noticing that  $\bs \Gamma(y)=\bar{\Gamma}(y,y)$ we arrive at \eqref{charact_H_suave_eq}.
\end{proof}
\begin{remark}\rm As natural,  if  $\lambda$ is the unique weak solution of \eqref{edp}, then the rate functional vanishes at $\lambda$.
In fact, given $G\in \C$, we have $\ell_G(\lambda)=0$ because $\lambda$ satisfies the integral equation \eqref{eqint1}.
Since $\psi(u)=e^u-u-1\geq 0$, it yields   $\hat{J}_G(\lambda)\leq 0$. And $\hat{J}_G(\lambda)= 0$ if $G$ is constant.
\end{remark}
By Proposition \ref{charact_H_suave}, profiles that are solution of \eqref{edpasy} for some $H$ provides a special representation for the rate functional. This motivates the next definition.
\begin{definition}
Denote by $\DME$ the subset of $\DMO$ consisting of all paths $\pi_t(du)=\rho_t(u)\,du$ for which there exists some 
$H\in \C$ such that  $\rho=\rho^H$ is the unique weak solution of \eqref{edpasy}. 
\end{definition}
We begin by  proving  the lower bound for trajectories in $\DME$. In the following we present the lower bound in the  set of smooth trajectories,  $\DMS$.
\begin{proposition}\label{lower subset}
 Let $\mc O$ be an open set of $\DM$. Then
\begin{equation*}
 \varliminf_{N\to\infty}
\frac{1}{N}\log\bb Q_{\mu_N}[\,\mc O\,]\;\geq\; -\inf_{\pi\in \mc O\cap \DME}I(\pi)\;.
\end{equation*}
\end{proposition}

\begin{proof}
This proof is essentially the same as that
found in \cite{kl}. Fix the open set $\mc O$.  
 Given $\pi\in \mc O\cap \DME$, by definition there exists $H\in \C$ such that 
$\pi_t(du)=\rho^H_t(u)\,du$, where $\rho^H$ is the weak solution of \eqref{edpasy}.
Denote by $\bb P_{\mu_N}^{H,\mc O}$ the probability on the space $\Ddiscreto$ defined by
\begin{equation*}
 \bb P_{\mu_N}^{H,\mc O}[A]\;=\;\frac{\bb P_{\mu_N}^{H}[A\,,\, \pi^N\in \mc O]}{\bb P_{\mu_N}^{H}[\pi^N\in \mc O]}\;,
\end{equation*}
for any $A$ measurable subset  of $\Ddiscreto$. Within  this definition,
\begin{equation*}
\frac{1}{N}\log\bb Q_{\mu_N}[\mc O]\;=\;
\frac{1}{N}\log\bb E_{\mu_N}^{H,\mc O}\Big[\,\radonNinv \,\Big]+\frac{1}{N}\log\bb Q_{\mu_N}^H[\mc O]\;.
\end{equation*}
Since $\mc O$ is a open set that contains $\rho^H$, by the Proposition \ref{hid asy} the second term in the right hand side of above converges to zero
as $N$ increases to infinity. Since the logarithm is a concave function, by Jensen's inequality the first  term in the right hand side of above is bounded from below by
\begin{equation*}
\bb E_{\mu_N}^{H ,\mc O}\Big[\frac{1}{N}\log\radonNinv \Big]\,.
\end{equation*}
Adding and subtracting the indicator function of the set $\{\pi^N\in{\mc O}^\complement\}$, the last expectation becomes
\begin{equation}\label{65a}
\frac{1}{\bb Q_{\mu_N}^{H}[\mc O]}\Bigg\{ -\frac{1}{N}\bs{H} \big(\bb P_{\mu_N}^{H}\vert\bb P_{\mu_N}\big)-
\bb E_{\mu_N }^H\Big[\frac{1}{N}\log\radonNinv\,\textbf{1}_{\{\pi^N\in{\mc O}^\complement\}} \Big]\Bigg\}\,,
\end{equation}
where 
\begin{equation}\label{ent}
\bs{H} \big(\bb P_{\mu_N}^{H}\vert\bb P_{\mu_N}\big)
\;:=\;\bb E_{\mu_N}^{H}\Big[\log\radonN\,\Big]
\;=\;-\bb E_{\mu_N}^H\Big[\log\radonNinv\,\Big]
\end{equation}
 is the so-called relative entropy of $\bb P_{\mu_N}^{H}$ with respect to $\bb P_{\mu_N}$.
Again by Proposition \ref{hid asy}  we have that $\bb Q_{\mu_N}^H[\mc O]$ converges to one as $N$ increases to  infinity. By \eqref{RN}
the expression $\frac{1}{N}\log\radonNinv$ is bounded, hence the second term inside braces in \eqref{65a} vanishes as $N$ increases to  $\infty$. Thus
\begin{equation*}
 \varliminf_{N\to\infty}
\frac{1}{N}\log\bb Q_{\mu_N}[\mc O]\;\geq\;\lim_{N\to\infty} -\frac{1}{N}\bs{H} \big(\bb P_{\mu_N}^{H}\vert\bb P_{\mu_N}\big)
\;=\; -I(\rho^H)\,,
\end{equation*}
where the last equality has an importance for itself and for this reason it is postponed to the Lemma \ref{entropy} proved next.
\end{proof}

\begin{lemma}\label{entropy}
Let  $H\in \C$.
Then
\begin{equation*}
 \lim_{N\to\infty}\frac{1}{N}\bs{H} \big(\bb P_{\mu_N}^{H}\vert\bb P_{\mu_N}\big)\;=\;I(\rho^H)\,,
\end{equation*}
where  $\rho^H$ is the unique weak solution of \eqref{edpasy}. 
\end{lemma}

\begin{proof}
Using the formula \eqref{ent} for the relative entropy, we get
\begin{equation}\label{ent1} 
\frac{1}{N}\bs{H} \big(\bb P_{\mu_N}^{H}\vert\bb P_{\mu_N}\big)\;=\;
\frac{1}{N}\bb E_{\mu_N}^{H}\Big[\log\radonN\,\textbf{1}_{B^H_{\delta, \eps}}\Big]+\frac{1}{N}\bb E_{\mu_N}^{H}\Big[\log\radonN\,\textbf{1}_{(B^H_{\delta, \eps})^\complement}\Big]\,,
\end{equation}
where the set $B^H_{\delta, \eps}$ was defined in \eqref{set B}. We  claim that the event
$(B^H_{\delta, \eps})^\complement$  is superexponentially small with respect to $\bb P_{\mu_N}^{H}$.
Indeed, by \eqref{RN} we have
\begin{equation*}
 \bb P_{\mu_N}^{H}\Big[(B^H_{\delta, \eps})^\complement\Big]\;=\;
\bb E_{\mu_N}\Big[\,\radonN \,\textbf{1}_{(B^H_{\delta, \eps})^\complement}\Big]
\;\leq\; e^{C(H,T)N}\bb P_{\nu_\alpha^N}\Big[(B^H_{\delta, \eps})^\complement\Big]
\end{equation*}
and then by \eqref{B} we get
\begin{equation*}
 \varlimsup_{\eps\downarrow 0}\varlimsup_{N\to\infty}\frac 1N \log 
 \bb P_{\mu_N}^H\Big[(B^H_{\delta,\eps})^{\complement}\Big]\;=\;-\infty\,.
\end{equation*}
Provided by the limit above and  the fact that $\frac{1}{N}\log\radonN$ is bounded, the right hand
side of \eqref{ent1} is 
\begin{equation}\label{entrop}
\frac{1}{N}\bb E_{\mu_N}^{H}\Big[\log\radonN\,\textbf{1}_{B^H_{\delta, \eps}}\Big]+o_N(1)\,,
\end{equation}
for all $\delta>0$ and each   small enough $\eps=\eps(\delta)$. 
Applying the expression \eqref{radon J} for the Radon-Nikodym derivative,  $\frac{1}{N}\log\radonN$ on the  set $B^H_{\delta, \eps}$ is equal to 
\begin{equation*}
\hat{J}_H\big((\pi^N*\iota_\gamma^{\textrm{s}})*\iota_\eps\big)
 + O_{H,T,\eps,\gamma}(\pfrac{1}{N})+O(\delta)+O_H(\eps)+O_H(\pfrac{\gamma}{\eps})\,,
\end{equation*}
for all $\delta>0$ and all $\eps$ and $\gamma$ small enough. Since this expression is bounded and the probability of 
$(B^H_{\delta, \eps})^\complement$ with respect to $\bb P_{\nu_\alpha^N}^{H}$ vanishes as $N$ increases to infinity, 
the expression \eqref{entrop} becomes
\begin{equation*}
\bb E_{\mu_N}^{H}\Big[\hat{J}_H\big((\pi^N*\iota_\gamma^{\textrm{s}})*\iota_\eps\big)\Big]
 + O_{H,T,\eps,\gamma}(\pfrac{1}{N})+O(\delta)+O_H(\eps)+O_H(\pfrac{\gamma}{\eps})+o_N(1)\,,
\end{equation*}
for all $\delta>0$ and all $\eps$ and $\gamma$ small enough. For fixed $\eps$ and $\gamma$, the map
 $\rho\mapsto \hat{J}_H\big((\rho*\iota_\gamma^{\textrm{s}})*\iota_\eps\big)$ is continuous with respect to the Skorohod topology, see the Proposition \ref{cont}. Moreover, by Proposition \ref{hid asy} the sequence  $\bb Q_{\mu_N}^{H}$ converges weakly to the probability concentrated on the weak solution of 
\eqref{edpasy}. In particular, as $N$ increases to infinity, the previous expectation converges to
\begin{equation*}
\hat{J}_H\big((\rho^H*\iota_\gamma^{\textrm{s}})*\iota_\eps\big)
 +O(\delta)+O_H(\eps)+O_H(\pfrac{\gamma}{\eps})\,.
\end{equation*}
Letting $\gamma\downarrow 0$, then taking $\eps\downarrow 0$, finally $\delta\downarrow 0$ and then invoking Lemma \ref{charact_H_suave} concludes the proof.
\end{proof}

Since  weak solutions of \eqref{edpasy} for some $H$ implies the special representation \eqref{charact_H_suave_eq} for the rate functional, it is natural to study in what conditions a profile $\rho$ can be written as a solution of \eqref{edpasy}. This is the content of the next proposition. Notice that the first equation in \eqref{eliptic} ahead is nothing else than the partial differential equation \eqref{edpasy} rearranged.

\begin{proposition}\label{eliptic lemma}
Let  $\rho\in \C$  such that
 $0<\eps\leq \rho \leq 1-\eps$,
 for some $\eps>0$.  Then, there exists a unique  (strong) solution $H\in \C$ 
    of the elliptic equation 
   
\begin{equation}\label{eliptic}
\begin{cases}
\vspace{0.3cm}
 \; \p_u^2 H_t(u)\,+ \,\pfrac{\p_u\big(\chi(\rho_t(u))\big)}{\chi(\rho_t(u))}\,\p_u H_t(u)\, = \,\pfrac{\Delta \rho_t(u)\,-\, \partial_t \rho_t(u) }{2\,\chi(\rho_t(u))} \,,\,  \forall u\in(0,1) \\
\vspace{0.3cm}

\; \p_u H_t (0)=\pfrac{1 }{2\,\chi(\rho_t(0))}\big[B e^{\delta H_t(0)}- Ce^{-\delta H_t(0)}
+\p_u\rho_t(0)\big]\\
\vspace{0.3cm}

\; \p_u H_t (1)= \pfrac{1 }{2\,\chi(\rho_t(1))}\big[Be^{\delta H_t(0)}- Ce^{-\delta H_t(0)}
+\p_u\rho_t(1)\big]\\
\vspace{0.2cm}

 \; H_t(0) \,=\, 0\\
\end{cases}
\end{equation}
  where
$ B= B(\rho_t)=\rho_t(1)(1-\rho_t(0))$ and $C=C(\rho_t)=\rho_t(0)(1-\rho_t(1))$, for all $t\in[0,T]$. Above we are denoting $0=0^+$ and $1=0^-$. 
\end{proposition}

\begin{proof} For fixed time, the first equation in \eqref{eliptic} is a linear second order  ordinary differential equation  in $H$. The only work is to adjust the solution to satisfy the boundary conditions. 
Let $z_0\in \bb R$  be the unique solution of the transcendental equation
$ z\;=\; (Be^{-z}-Ce^z)\,\alpha+A$, where 
\begin{equation*}
 \alpha=\alpha(\rho_t):=\int_0^1\frac{1 }{2\,\chi(\rho_t(v))}\,dv\,,
\end{equation*}
\begin{equation*}
A=A(\rho_t):=\int_0^1\frac{\p_u\rho_t(v)-\p_t\int_0^v\rho_t(w)\,dw}{2\,\chi(\rho_t(v))}\,dv\,,
\end{equation*}
and $B>0$ and $C>0$ are those ones in the statement of the proposition.
Let
\begin{equation*}
 H_t(u):=(Be^{-z_0}-Ce^{z_0})\int_0^u\frac{1 }{2\chi(\rho_t(v))}dv+
\int_0^u\frac{\p_u\rho_t(v)-\p_t\int_0^v\rho_t(w)\,dw}{2\chi(\rho_t(v))}\,dv,
\end{equation*}
for all $t\in[0,T]$. It can be directly checked that $H$ is the solution of  \eqref{eliptic}.
\end{proof}
Recalling the definition of $\DMS$ given in the Theorem \ref{t03} and the definition of $\DME$, Proposition \ref{eliptic lemma} can be resumed as:
\begin{corollary}\label{coro} The set
\begin{equation*}
\DMS \cap \{\pi\in \DM\;;\; \pi_t(du)=\rho_t(u)du\,, \;\textrm{ with } \eps\leq \rho\leq 1-\eps \;\textrm{ for some } \eps>0\}
\end{equation*}
is contained in $\DME$.
 \end{corollary}

Despite not convex in general, the rate functional $I$ obtained in our model is convex in some sense.  This is subject of the next proposition, to be used in a density argument.

\begin{proposition}\label{convex}
Let $\rho,\lambda\in\DM$ with $ I(\rho)$ and $I(\lambda)$  finite such that 
 $\big(\rho_t(0^+)-\lambda_t(0^+)\big)\big(\rho_t(0^-)-\lambda_t(0^-)\big)\geq 0$, 
almost surely in $t\in[0,T]$. Then, for $\theta\in [0,1]$,
\begin{equation}\label{eq_convex}
 I(\theta\rho+(1-\theta)\lambda)\;\leq\;\theta\, I(\rho)+(1-\theta)\,I(\lambda)\,.
\end{equation}
\end{proposition}
\begin{proof}
Let $\theta\in [0,1]$.
We claim that  
\begin{equation}\label{667}
 \hat{J}_H(\theta\rho+(1-\theta)\lambda)\;\leq\;\theta \hat{J}_H(\rho)+(1-\theta)\hat{J}_H(\lambda)\,,
\end{equation}
for any $H\in\C$. Recall that $\hat{J}_H(\rho)$ is the sum of  linear part in $\rho$, namely
\begin{equation*}
\begin{split}
& \ell_H(\rho)  -\int_0^T\Big\{\rho_t(0^-)\, \psi(\delta H_t(0))\,
+\,\rho_t( 0^+) \,\psi\big(-\delta H_t(0)\big)\Big\}\,dt\,,
\end{split}
\end{equation*}
plus a convex part in $\rho$, namely $-\int_0^T\<\chi(\rho_t),(\partial_u H_t)^2\>\,dt$, and 
\begin{equation}\label{67}
\begin{split}
& \int_0^T\rho_t(0^-)\, \rho_t( 0^+)\,\Big\{\psi(\delta H_t(0))+\psi\big(-\delta H_t(0)\big)\Big\}\,dt\,,
\end{split}
\end{equation}
wherefore we only need to care about this last term. 
Since $\psi(x)=e^x-x-1\geq 0$, we have that  $\psi(\delta H_t(0))+\psi(-\delta H_t(0))\geq 0$. 
Let $f: \bb R^2\to \bb R$ be the function defined by $f(x,y)=xy$. If $(x_1,y_1)$ and $(x_2,y_2)$ are two points of $\bb R^2$ such that $(x_2-x_1)(y_2-y_1)\geq 0$, then
\begin{equation}\label{68}
f\Big(\theta (x_1,y_1) +(1-\theta) (x_2,y_2) \Big)\;\leq \; \theta f(x_1,y_1) +(1-\theta) f(x_2,y_2)\,.
\end{equation} 
To see this, just note that $f$ is convex along lines of the form $y=ax+b$, provided $a>0$.
The inequality \eqref{68} applied to \eqref{67} permits to conclude the inequality \eqref{667}, which in his hand leads to \eqref{eq_convex}.
\end{proof}

\begin{proposition}\label{1aprox}
Let $\pi\in\DM$ with $I(\pi)<\infty$.  There exists a sequence
 $\{\pi^\eps\}_{\eps>0}$  in $\DMO$
such that 
 $\pi^\eps$ converges to $\pi$ in $\DM$ and  
$\pi^\eps_t(du)=\rho^\eps_t(u)\,du$ with $\eps\leq \rho^\eps_t(u)\leq 1-\eps$. Moreover, $
\varlimsup_{\eps\downarrow 0}I(\pi^\eps)\leq I(\pi)$.
\end{proposition}
\begin{proof}
 Let $\pi\in\DM$ with $I(\pi)<\infty$, then $\pi_t(du)=\rho_t(u)\,du$ and $0\leq \rho\leq 1$. 
Consider $\tilde{1}(t,u)=1$ and $\tilde{0}(t,u)=0$, for all $t\in[0,T]$ and $u\in \bb T$.
Define $\rho^\eps=\eps\tilde{1}+(1-2\eps)\rho+\eps\tilde{0}$ and $\pi^\eps_t(du)=\rho^\eps_t(u)du$.
By Lemma \ref{convex}, $I(\pi^\eps)\leq\eps I(\tilde{1})+(1-2\eps)I(\rho)+\eps I(\tilde{0})$. Hence
$\varlimsup_{\eps\downarrow 0}I(\pi^\eps)\leq I(\pi)$. 

\end{proof}

We are in position to prove the lower bound for smooth profiles.
\begin{proof}[Proof of the Theorem \ref{t03}, item (ii).]
Fix $\pi\in \DMS\cap \mc O$  and consider the sequence $\pi^\eps_t(du)=\rho^\eps_t(u)du$, where $\rho^\eps_t(u)=\eps+ (1-2\eps)\rho_t(u)$, as in the proof of the  Proposition \ref{1aprox}.  That is, such that $\eps<\rho^\eps<1-\eps$ with $\rho^\eps\in\C$. By Corollary \ref{coro} and since $\mc O$ is open,  we have that $\pi^\eps\in  \DME\cap \mc O$ for small enough $\eps>0$.

 By Proposition \ref{lower subset}, 
\begin{equation*}
 \varliminf_{N\to\infty}
\frac{1}{N}\log\bb Q_{\mu_N}[\,\mc O\,]\;\geq\; -\inf_{\lambda\in \mc O\cap \DME }I(\lambda)\;\geq\; -I(\pi^\eps)\,.
\end{equation*}
Taking the limit infimum in the right hand side of inequality above and using the Lemma \ref{1aprox}, we get 
\begin{equation*}
 \varliminf_{N\to\infty}
\frac{1}{N}\log\bb Q_{\mu_N}[\,\mc O\,]
\;\geq\;   -\varlimsup_{\eps\to 0}I(\pi^\eps)\;\geq\;  -I(\pi)\,.
\end{equation*}
Since $\pi$ is an arbitrary  trajectory on the set $\mc O\cap \mc D^{\mc S}_{ \mc M_0}$, we can optimize over all elements in this set, obtaining therefore
\begin{equation*}
 \varliminf_{N\to\infty}
\frac{1}{N}\log\bb Q_{\mu_N}[\,\mc O\,]
\;\geq\; \sup_{\pi\in \mc O\cap \mc D^{\mc S}_{ \mc M_0}} -I(\pi)\;=\; -\inf_{\pi\in \mc O\cap \mc D^{\mc S}_{ \mc M_0}} I(\pi)\,,
\end{equation*}
which finishes the proof.
\end{proof}

\appendix
\section{Uniqueness of strong solutions}\label{unique asy}
As aforementioned, we have assumed uniqueness of weak solutions of \eqref{edpasy}, a delicate problem in the area of partial differential equations for which we have no argument. In this appendix we present uniqueness of strong solutions of \eqref{edpasy}.
\begin{theorem}
Let $\rho_0:\bb R\to [0,1]$ be measurable profile. Then, there exists at most one strong solution of the partial differential equation \eqref{edpasy}.
\end{theorem}
\begin{proof} 
We will describe a general situation that includes the PDE \eqref{edpasy}. Let $u_1$ and $u_2$ two strong  solutions of
\begin{equation*}
\begin{cases}
\; \p_t u =\p_x^2 u + F(t,x,u,\p_x u)\\
\;  u(0,x)=\bar{u}(x)\\
\;  \p_x u(0)=H_0(t,x,u(0),u(1))\\
\;  \p_x u(1)=H_1(t,x,u(0),u(1))\\
\end{cases}
\end{equation*}
where $F$, $H_0$, $H_1$ are smooth functions.
Let $v=u_1-u_2$. Hence $v(0,x)=0$ and
$
\p_t v= \p_x^2 v + L$,
where 
\begin{equation*}
L\;=\;F(t,x,u_1,\p_x u_1)-F(t,x,u_2,\p_x u_2)\,.
\end{equation*}
By smoothness, there exists a constant $C>0$ such that hold the estimates
\begin{equation*}
\begin{split}
& |F(t,x,u_1,\p_x u_1)-F(t,x,u_2,\p_x u_2)|\;\leq\; C(|v|+|\p_x v|) \,,\\
&|H_i(t,x,u_1(0),u_1(1))-H_i(t,x,u_2(0),u_2(1))|\;\leq\; C(|v(0)|+|v(1)|)\,,\\
\end{split}
\end{equation*}
for $i=0,1$. 
An application of Young's inequality implies that, for all $\eps>0$, there exists $A(\eps)>0$ such that
\begin{equation}\label{eq29}
|v(0)|^2+|v(1)|^2 \;\leq\; \eps \int_0^1(\p_x v)^2\,dx+ A(\eps)\int_0^1 v^2\,dx\,,
\end{equation}
for any time $t>0$. Define $q(t)=\int_0^1v^2(t,x)\,dx$. Then
\begin{equation*}
\begin{split}
& q'(t)\;=\;2\int_0^1 v\,\p_t v\,dx  \;=\; 2 \int_0^1 v\,\p_x^2 v\, dx +2 \int_0^1 v\,L\, dx\\
&  = \; 2\,v(1)\,\p_xv(1)- 2\,v(0)\,\p_xv(0)-2\int_0^1 (\p_x v)^2\,dx +  2 \int_0^1 v\,L\, dx\,.
\end{split}
\end{equation*}
Thus, by previous estimates,
\begin{equation*}
\begin{split}
q'(t)\;\leq\; &   C_1 \Big((v(0))^2+|v(0)|\,|v(1)|+(v(1))^2\Big)-2\int_0^1 (\p_x v)^2\,dx\\
&+C_1\int_0^1 v^2\,dx+ C_1\int_0^1 |v|\,|\p_x v|\,dx\,.
\end{split}
\end{equation*}
Again by Young's Inequality,
\begin{equation*}
\begin{split}
q'(t)\;\leq\; &  -2\int_0^1 (\p_x v)^2\,dx + C_2 \Big((v(0))^2+(v(1))^2\Big)\\
&+C_2\int_0^1 v^2\,dx+ \beta \int_0^1 (\p_x v)^2\,dx\;.\\
\end{split}
\end{equation*}
where $\beta$ can be chosen small as necessary. Recalling \eqref{eq29} with small $\eps$ gives us
\begin{equation*}
q'(t)\;\leq\;  -\frac{1}{2}\int_0^1(\p_x v)^2\,dx + C_3 \int_0^1 v^2\,dx
\end{equation*}
implying $q'(t)\leq C_3\, q(t)$. Noticing that $q(0)=0$, Gronwall's inequality finishes the proof.
\end{proof}

%\section*{Acknowledgements}

%\begin{thebibliography}{9}
%
%\bibitem{r1}
%\textsc{Billingsley, P.} (1999). \textit{Convergence of
%Probability Measures}, 2nd ed.
%Wiley, New York.
%\MR{1700749}
%
%
%\bibitem{r2}
%\textsc{Bourbaki, N.}  (1966). \textit{General Topology}  \textbf{1}.
%Addison--Wesley, Reading, MA.
%
%\bibitem{r3}
%\textsc{Ethier, S. N.} and \textsc{Kurtz, T. G.} (1985).
%\textit{Markov Processes: Characterization and Convergence}.
%Wiley, New York.
%\MR{838085}
%
%\bibitem{r4}
%\textsc{Prokhorov, Yu.} (1956).
%Convergence of random processes and limit theorems in probability
%theory. \textit{Theory  Probab.  Appl.}
%\textbf{1} 157--214.
%\MR{84896}
%
%\end{thebibliography}

\end{document}